\documentclass[a4paper,reqno,12pt]{amsart}
\usepackage{fancyhdr}
\usepackage{amsmath}
\usepackage{dsfont}
\usepackage{hyperref}
\usepackage[mathscr]{eucal}
\usepackage[cp1251]{inputenc}
\usepackage[english]{babel}
\usepackage{cite,enumerate,float,indentfirst}
\usepackage{graphicx}
\usepackage{xcolor}
\usepackage{latexsym,a4,mathrsfs,amsthm,amsmath,amssymb,url}
\usepackage{amsfonts}
\usepackage{amssymb}
\usepackage{epstopdf}

\newcommand\NoBlackBoxes{\global\overfullrule0pt}
\NoBlackBoxes
\numberwithin{equation}{section}

\newtheorem{theorem}{Theorem}[section]
\newtheorem{lemma}[theorem]{Lemma}

\theoremstyle{remark}
\newtheorem*{remark}{Remark}

\newcommand{\R}{\mathbb{R}}
\newcommand{\C}{\mathbb{C}}
\newcommand{\T}{\mathbb{T}}

\newcommand{\J}{\mathbb{J}}

\newcommand{\Cond}{{\bf (C0)}}
\newcommand{\CondTwo}{{\bf (C1)}}

\newcommand{\X}{{\bf X}}
\newcommand{\U}{{\bf U}}
\newcommand{\A}{{\bf A}}

\newcommand{\B}{{\bf B}}

\newcommand{\I}{{\bf I}}
\newcommand{\RR}{{\bf R}}
\newcommand{\V}{{\bf V}}

\newcommand{\W}{{\bf W}}

\newcommand{\uu}{{\bf u}}

\DeclareMathOperator{\Tr}{Tr}
\DeclareMathOperator{\Rank}{Rank}

\DeclareMathOperator{\E}{\mathbb{E}}

\DeclareMathOperator{\Pb}{\mathbb{P}}
\DeclareMathOperator{\one}{\mathds{1}}
\DeclareMathOperator{\re}{Re}
\DeclareMathOperator{\imag}{Im}

\begin{document}

\vspace{1in}

\title[Local Semicircle Law]{\bf Local Semicircle law Under Moment conditions. \\ Part I: The Stieltjes transform}

\author[F. G{\"o}tze]{F. G{\"o}tze}
\address{Friedrich G{\"o}tze\\
 Faculty of Mathematics\\
 Bielefeld University \\
 Bielefeld, Germany
}
\email{goetze@math.uni-bielefeld.de}

\author[A. Naumov]{A. Naumov}
\address{Alexey A. Naumov\\
 Faculty of Computational Mathematics and Cybernetics\\
 Lomonosov Moscow State University \\
 Moscow, Russia \\
 and Institute for Information Transmission Problems of the Russian Academy of Sciences (Kharkevich Institute), Moscow, Russia
 }
\email{anaumov@cs.msu.su}

\author[A. N. Tikhomirov]{A. N. Tikhomirov}
\address{Alexander N. Tikhomirov\\
 Department of Mathematics\\
 Komi Research Center of Ural Division of RAS \\
 Syktyvkar, Russia
 }
\email{tikhomirov@dm.komisc.ru}

\thanks{All authors were supported by CRC 701 “Spectral Structures and Topological Methods in
Mathematics”. A. Tikhomirov wast supported by RFBR N~14-01-00500 and by Program of UD RAS, project No 15-16-1-3. A.~Naumov was supported by  RFBR N~16-31-00005 and President's of Russian Federation Grant for young scientists N~4596.2016.1.  }

\keywords{Random matrices, Local semicircle law, Stieltjes transform}

\date{\today}

\begin{abstract}
We consider a random symmetric matrix $\X = [X_{jk}]_{j,k=1}^n$ in which the upper triangular entries are independent identically distributed random variables with mean zero and unit variance. We additionally suppose that
$\E |X_{11}|^{4 + \delta} =: \mu_{4+\delta} < \infty$ for some $\delta > 0$. Under these conditions we show that the typical distance between the Stieltjes transform of the empirical spectral distribution (ESD) of the matrix $n^{-\frac{1}{2}} \X$ and Wigner's semicircle law is of order $(nv)^{-1}$, where $v$ is the distance in the complex plane to the real line. Furthermore we  outline applications which are deferred to a subsequent paper, such as the rate of convergence in probability of the ESD to the distribution function of the semicircle law, rigidity of the eigenvalues and  eigenvector delocalization.
\end{abstract}

\maketitle

\section{Introduction and main result}
We consider a random symmetric matrix $\X = [X_{jk}]_{j,k=1}^n$ where the upper triangular entries are independent random variables with mean zero and unit variance. We will be mostly interested in  limiting laws for the
eigenvalues and eigenvectors of large $n\times n$ symmetric random matrices in the asymptotic limit as $n$ goes to infinity.

For the symmetric matrix $\W: = \frac{1}{\sqrt n} \X$ we denote its $n$ eigenvalues in the increasing order as
$$
\lambda_1(\W) \le ... \le \lambda_n(\W)
$$
and introduce the eigenvalue counting function
$$
N_I(\W):= |\{1 \le k \le n: \lambda_k(\W) \in I\}|
$$
for any interval $I \subset \R$, where $|A|$ denotes the number of elements in the set $A$. Note that  we shall sometimes  omit $\W$ from the notation of $\lambda_j(\W)$.

It is well known since the pioneering work of E. Wigner~\cite{Wigner1958} that for any interval $I \subset \R$ of fixed length and independent of $n$
\begin{equation}\label{eq: wigner's semicircle law global regime}
\lim_{n \rightarrow \infty} \frac{1}{n}\E N_I(\W) = \int_I g_{sc}(\lambda) \, d\lambda,
\end{equation}
where
$$
g_{sc}(\lambda) := \frac{1}{2\pi} \sqrt{4 - \lambda^2} \one[|\lambda| \le 2]
$$
is the density function of Wigner's semicircle law.  Here and in what follows we denote by $\one[A]$ the indicator function of the set $A$. Wigner considered the special case when all $X_{jk}$ take only two values $\pm 1$ with equal probabilities. To prove~\eqref{eq: wigner's semicircle law global regime} he used the moment method which may be described as follows. Since $g_{sc}$ is compactly supported it is uniquely determined by the sequence of its moments given by
$$
\beta_k = \begin{cases}
\frac{1}{m+1} \binom{2m}{m}, &k = 2m,\\
0, &k = 2m + 1. \\
\end{cases}, \qquad k \geq 1.
$$
We remark here that $\beta_{2m}, m \geq 1$, are Catalan numbers. To establish the convergence~\eqref{eq: wigner's semicircle law global regime} one needs to show that
$$
\lim_{n \rightarrow \infty } \int_{-\infty}^{\infty} \lambda^k \, d F_n(\lambda) = \int_{-2}^{2} \lambda^k g_{sc}(\lambda) \, d\lambda,
$$
where $F_n(\lambda) : = \frac{1}{n} N_{(-\infty, \lambda]}(\W)$ is the empirical spectral distribution function. Further details may be found in~\cite{BaiSilv2010}.

Wigner's semicircle law has been extended in various aspects. For example, L. Arnold in~\cite{Arnold1967} proved almost sure (a.s.) convergence of $F_n$ to the semicircle law but under additional moment assumptions on the matrix entries. More general conditions of convergence to Wigner's semicircle law were established by L.~Pastur in~\cite{Pastur1973}. For all $\tau > 0$ we define Lindeberg's ratio for the random matrix $\X$ by the relation
\begin{equation}\label{lind condition}
L_n(\tau) := \frac{1}{n^2} \sum_{j,k=1}^n \E X_{jk}^2 \one[|X_{jk}| \geq \tau \sqrt n] .
\end{equation}
Pastur has shown that the convergence of Lindeberg's ratio to zero is sufficient for the convergence in probability to the semicircle law. V.~Girko in~\cite{Girko1990},~\cite{Girko1985uspexi} extended Pastur's result to a.s. convergence and stated that~\eqref{lind condition} is also necessary condition. This result, in particular, implies that if $X_{jk}, 1 \le j \le k \le n$, are independent identically
distributed (i.i.d.) random variables and have zero mean and unit variance, then $F_n$
converges a.s. to Wigner's semicircle law. We remark that all these results were established for symmetric random matrices with independent, not necessary identically distributed entries, but assuming that $\E X_{jk}^2 = 1$ for all $1 \le j \le k \le n$. This limitation has been overcome in a sequence of papers, see, for example,~\cite{Shly1996},~\cite{Naumov2013gaussiancase} and~\cite{GotNauTikh2012prob}. In  the last years there has been increasing interest in random matrices with dependent entries. For some models it has been shown that Wigner's semicircle law holds as well for the matrices with  dependent entries, see, for example,~\cite{GotTikh2006},\cite{GotNauTikh2012prob}, ~\cite{BannaMerlPeligrad2015}.

All these results hold for intervals $I$ of fixed length, independent of $n$, which typically contain a  macroscopically large number of eigenvalues, which means a number of order $n$.
Unfortunately for  smaller intervals where the number of eigenvalues cease to be macroscopically large  the moment method does not apply
and one needs the Stieltjes transform of the empirical spectral distribution function $F_n$, which is is given by
$$
m_n(z) := \int_{-\infty}^\infty \frac{d F_n(\lambda)}{\lambda - z}   = \frac{1}{n}\Tr (\W - z \I)^{-1} = \frac{1}{n} \sum_{j=1}^n \frac{1}{\lambda_j(\W) - z},
$$
where $z = u + i v, v \geq 0$. The Stieltjes transform is an appropriate tool to  study
spectral densities, since
taking the imaginary part of $m_n(z)$ we get
$$
\imag m_n(u + i v) = \int_{-\infty}^\infty \frac{v}{(\lambda - u)^2 + v^2} \, d F_n(\lambda) = \frac{1}{v}\int_{-\infty}^\infty K\left(\frac{u-\lambda}{v}\right) \, d F_n(\lambda)
$$
which is the kernel density estimator with kernel $K$ and bandwidth $v$. For a meaningful estimator of the spectral density we cannot allow the distance $v$ to the real line, that is the bandwidth  of the kernel density estimator, to be smaller than the typical $\frac{1}{n}$ -distance between eigenvalues.
Hence, in what follows we shall be mostly interested in the situations when $v \gg \frac{1}{n}$.

Under rather general conditions, like convergence  of Lindeberg's ratio~\eqref{lind condition}, to zero for fixed $v > 0$ one may establish the convergence  of  $m_n(z)$ to the the Stieltjes transform of Wigner's semicircle law which is given by
$$
s(z) = \int_{-\infty}^\infty \frac{g_{sc}(\lambda)\, d\lambda}{\lambda-z}
$$
and may be calculated explicitly by
\begin{equation}\label{eq: Stieltjes transform of Wigner's semicircle law}
s(z) = -\frac{z}{2} + \sqrt{\frac{z^2}{4} - 1},
\end{equation}
see, for example,~\cite{BaiSilv2010} for a simple  explanation.

It is much more difficult  to establish the convergence in the region $1 \gg v \gg \frac{1}{n}$. Significant progress in that direction was recently made in a series of results by L.~Erd{\"o}s, B.~Schlein, H.-T.~Yau and et al., \cite{ErdosSchleinYau2009},~\cite{ErdosSchleinYau2009b} ,~\cite{ErdosSchleinYau2010},~\cite{ErdKnowYauYin2013}, showing
that with high probability uniformly in $u \in \R$
\begin{equation}\label{fluctuations of m_n around s}
|m_n(u+iv) - s(u+iv)| \le \frac{\log^\beta n}{nv}, \quad \beta > 0,
\end{equation}
which they called  {\it local semicircle law}. It means  that the fluctuations of $m_n(z)$ around $s(z)$ are of order $(nv)^{-1}$ (up to a logarithmic factor). The value of $\beta$ may depend on $n$, to be exact $\beta: = \beta_n = c\log \log n$, , where $c > 0$ denotes some constant. To prove~\eqref{fluctuations of m_n around s} in those papers ~\cite{ErdosSchleinYau2009},~\cite{ErdosSchleinYau2009b},~\cite{ErdosSchleinYau2010}
it was assumed that the distribution of $X_{jk}$ for all $1 \le j, k \le n$ has sub-exponential tails. Moreover in~\cite{ErdKnowYauYin2013} this assumption had been relaxed to requiring $\E |X_{jk}|^p \le \mu_p$ for all $p \geq 1$, where $\mu_p$ are some constants. Since there is meanwhile an extensive literature on the local semicircle law we refrain from  providing a complete list here and refer the  reader to the surveys of L.~Erd{\"o}s~\cite{Erdos2011} and T. Tao, V. Vu,~\cite{TaoVu2011surv}.

Combining the results and arguments  of the papers~\cite{ErdKnowYauYin2013a},~\cite{ErdKnowYauYin2012}
	with the more recent results of~\cite{LeeYin2014} it follows  that ~\eqref{fluctuations of m_n around s} holds under the condition that $\E |X_{jk}|^{4 + \delta} =: \mu_{4 + \delta} < \infty$. (L.~Erd\"os and H.-T.~Yau private communication). To explain it in more details,
	 assume that  $1 \le j,k \le n$ $|X_{jk}| \le n^\frac12 q^{-1}$ with probability larger than $1 - e^{-n c}$ for some $c > 0$. Here, $q$ may depend on $n$ and usually $n^\phi \le q \le n^{1/2}\log^{-1} n$ for some $\phi > 0$. This means that all $X_{jk}$ are bounded in absolute value by some quantity from $\log n$ to $n^{1/2 - \phi}$.  Define the following region in the complex plane
$$
\mathbb S(C) = \{z = u + i v \in \C:  |u| \le 5, 10 \ge v \ge \varphi^C n^{-1} \},
$$
where $\varphi: =(\log n)^{\log \log n}$ and $C>0$ and let
\begin{equation}\label{eq: gamma definition}
\gamma: = \gamma(u):=  ||u| - 2|.
\end{equation}
In~\cite{ErdKnowYauYin2013a}[Theorem~2.8] it is shown that there exist positive constants $C$ and $c$ such that for all $q \geq \varphi^C$
\begin{align}\label{eq: compare results}
&\Pb \left(\bigcap_{z \in \mathbb S(C)}  \left\{|m_n(z) - s(z)| \le \varphi^C \left(\min\left(\frac{1}{q^2\sqrt{\gamma + v}}, \frac{1}{q}\right) + \frac{1}{nv}\right) \right\}    \right)\nonumber\\ &\qquad\qquad\qquad\qquad\qquad\qquad\qquad\qquad\qquad\qquad\qquad\qquad\qquad\qquad\geq 1 - n^{C} e^{-c \varphi^3}.
\end{align}
Assume now that $\E |X_{jk}|^{4 + \delta} < C$. In~\cite{ErdKnowYauYin2012} (Lemma~7.6 and~7.7) the initial matrix has been replaced by a matrix $\tilde \X$ matching the first moments of $\X$ but having sub-exponential decaying tails. In the recent paper~\cite{LeeYin2014}[Lemma~5.1] it has been shown in particular  that there exist such a matrix $\tilde \X$ such that
$\E X_{jk}^s = \E \tilde X_{jk}^s$ for $s = 1,...,4$ and $|\tilde X_{jk}| \le  C \log n$ (this means that $q = O(n^{1/2} \log^{-1} n$)). Finally, it remains to estimate the difference $m_n(z) - s(z)$ in terms of $\tilde m_n(z) - s(z)$, where $\tilde m_n(z)$ is the Stieltjes transform corresponding to $ \tilde \X$.

The aim of  this paper is to give self-contained proof of~\eqref{fluctuations of m_n around s} assuming  that $\E |X_{jk}|^{4 + \delta} =: \mu_{4 + \delta} < \infty$.  We apply different techniques, which allow to reduce the power of $\log n$ (see the definition of $\varphi$ above) from $\beta = c\log \log n$ to some constant small constant independent of $n$.  We also extend the recent results of F.~G{\"o}tze and A.~Tikhomirov in~\cite{GotTikh2015} and~\cite{GotzeTikh2014rateofconv}, with   $\delta = 4$, where  we required $8$ moments
	together with $u$ lying in the support of the semicircle law. Our work and many details of the proof were motivated by a recent paper of C.~Cacciapuoti, A.~Maltsev and B.~Schlein,~\cite{Schlein2014}, where the authors improved the log-factor dependence in~\eqref{fluctuations of m_n around s} in the sub-Gaussian case. We mention that in the latter case one has $\E |X_{jk}|^p \le (C\sqrt{p})^p$ for all $p \geq 1$.

To control the distance between $m_n(z)$ and $s(z)$ one may estimate $\E |m_n(z) - s(z)|^p$. Instead of a combinatorial approach to deal with the last quantity we apply the method developed in~\cite{GotTikh2015},~\cite{GotzeTikh2014rateofconv} which is a Stein type method. In addition we apply the  descent method for the estimation of the moments of diagonal entries of the resolvent using a few {\it multiplicative steps }  introduced in~\cite{Schlein2014}. In earlier work of L.~Erd{\"o}s, B.~Schlein, H.-T.~Yau and et al. mentioned above a  larger number of {\it additive steps} of descent had been used. For the details of our technique see Section~\ref{Sketch of the proof}.

\subsection{Main result}

We consider a random symmetric matrix $\X = [X_{jk}]_{j,k=1}^n$ in which $X_{jk}, 1 \le j \le k \le n$, are independent identically distributed random variables with
$\E X_{11} = 0$ and $\E X_{11}^2 = 1$. We additionally suppose that
$$
\E|X_{11}|^{4+\delta} < C.
$$
for some $\delta > 0$ and some constant $C>0$. In this case we say that the matrix $\X$ satisfies the conditions $\Cond$.

We introduce the following quantity depending on $\delta$
$$
\alpha: = \alpha(\delta) = \frac{2}{4+\delta},
$$
which will control the level of truncation of the matrix entries.

Without loss of generality we may assume that $\delta \le 4$ since otherwise all bounds will be independent of $\alpha$. This means that we may assume that
$$
\frac{1}{4} \le \alpha < \frac{1}{2}.
$$
The following theorem
about approximation of Stieltjes transforms is the main result of this paper.
\begin{theorem}\label{th:main}
Assume that the conditions $\Cond$ hold and let $V > 0$ be some constant. \\
$(i)$ There exist positive constants $A_0, A_1$ and $C$ depending on $\alpha$ and $V$  such that
$$
\E |m_n(z) - s(z)|^p \le \left(\frac{Cp^2}{nv}\right)^p,
$$
for all $1 \le p \le A_1 (nv)^{\frac{1-2\alpha}{2}}$, $V \geq v \geq A_0 n^{-1}$ and $|u| \le 2+v$. \\

\noindent
$(ii)$ For any $u_0 > 0$ there exist positive constants $A_0, A_1$ and $C$ depending on $\alpha, u_0$ and $V$
such that
$$
\E |\imag m_n(z) - \imag s(z)|^p \le \left(\frac{Cp^2}{nv}\right)^p,
$$
for all $1 \le p \le A_1 (nv)^{\frac{1-2\alpha}{2}}$, $V \geq v \geq A_0 n^{-1}$ and $|u| \le u_0$.
\end{theorem}

\begin{remark}
Let us complement the results stated above by the following remarks.

\noindent  1. The methods used in our proof (see section \ref{Sketch of the proof} below) differ from those used in \cite{ErdKnowYauYin2013a},~\cite{ErdKnowYauYin2012} and~\cite{LeeYin2014} which are outlined above. In particular we may also rewrite our result in terms of probability bounds. Indeed, applying Markov's inequality we may rewrite, for example, the first estimate in the following form
\begin{equation}\label{main result part 1 1 probability}
\Pb \left( |m_n(z) - s(z)| \geq \frac{K}{nv}\right) \le \left(\frac{Cp^2}{K}\right)^p,
\end{equation}
for all $1 \le p \le A_1 (nv)^{\frac{1-2\alpha}{2}}$, $V \geq v \geq A_0 n^{-1}$ and $|u| \le 2+v$. For application we are interested in the range of $v$, such that~\eqref{main result part 1 1 probability} is valid for fixed $p$. It is clear that $V \geq v \geq C p^\frac{2}{1-2\alpha} n^{-1}$. Since we are interested in polynomial estimates we need to take $p$ of order $\log n$, which implies that $V \geq v \geq C n^{-1} \log^\frac{2}{1-2\alpha} n$. At the same time $K$ in~\eqref{main result part 1 1 probability} should be of order $\log^2 n$. Comparing with~\eqref{fluctuations of m_n around s} we get $\beta = 2$. If we would like to have the bound $n^{-c \log \log n}$ we should take $\beta = 3$. 

\noindent 2.  We conjecture that the result of Theorem~\ref{th:main} holds with $\delta = 0$ which corresponds to the case of finite fourth moment only.

\noindent 3.  The case of non-identically distributed $X_{jk}$,  can be dealt with by our methods, but the details are more involved and we omit its proof here.

\noindent 4. Note that it is possible to reduce the power $\frac{1-2\alpha}{2}$ assuming that at least eight moments of the matrix entries are finite. This situation corresponds to $\delta = 4$ {\it and} $\alpha = \frac{1}{4}$.

\noindent 5. On the right hand side of the estimates in (i) and (ii) the  dependence on the power $p$  is given sharpened to $p^p$ compared to $p^{p^2}$ in the main theorem of~\cite{Schlein2014}.

\end{remark}

Applications of Theorem~\ref{th:main} outside the limit spectral interval, that is for $u \geq 2$,  require  stronger bounds on $\imag \Lambda_n$.
We say that the set of conditions $\CondTwo$ holds if $\Cond$ are satisfied and $|X_{jk}| \le D n^\alpha, 1 \le j,k \le n$, where $D: =D(\alpha)$ is some positive constant depending on $\alpha$ only.
\begin{theorem}\label{th: stronger bound for imag part}
Assume that the conditions $\CondTwo$ hold and $u_0 > 2$ and $V > 0$. There exist positive constants $A_0, A_1$ and $C$ depending on $\alpha, u_0$ and $V$ such that
$$
\E |\imag m_n(z) - \imag s(z)|^p \le \frac{C^p p^p }{n^p (\gamma + v)^p} +  \frac{C^p p^{3p}}{(nv)^{2p} (\gamma + v)^\frac{p}{2}}  + \frac{C^p}{n^p v^\frac{p}{2} (\gamma + v)^\frac{p}{2}} + \frac{C^p p^p}{(nv)^\frac{3p}{2} (\gamma + v)^\frac{p}{4} },
$$
for all $1 \le p \le A_1 (nv)^{\frac{1-2\alpha}{2}}$, $V \geq v \geq A_0 n^{-1}$ and $2 \le |u| \le u_0$.
\end{theorem}

\subsection{Applications of the main result}

In a subsequent paper we shall apply Theorem~\ref{th:main} to prove a series of results which we will formulate and discuss now. We start with the rate of convergence in probability. Let us denote by
$$
\Delta_n^{*}: = \sup_{x \in \R} |F_n(x) - G_{sc}(x)|,
$$
where $G_{sc}(x): = \int_{-\infty}^{x} g_{sc}(\lambda)\, d\lambda$. It was proved by F. G{\"o}tze and A. Tikhomirov in~\cite{GotTikh2003}
that assuming $\max_{1 \le j,k \le n} \E |X_{jk}|^{12} = : \mu_{12} < \infty$, one may obtain the following estimate
$$
\E \Delta_n^{*} \le \mu_{12}^\frac16 n^{-\frac12}.
$$
In particular this estimate implies by Markov's inequality that
\begin{equation}\label{eq: rate of convergence}
\Pb \left( \Delta_n^{*} \geq K \right ) \le \frac{\mu_{12}^\frac16}{ K n^{\frac12}}.
\end{equation}
It is easy to see from the previous bound that one may take $K \gg n^{-\frac{1}{2}}$. This result has been extended by Bai  and et al., see~\cite{BaiHuPanZhou2011}. showing that instead of the twelfth finite moments it suffices to require finiteness of six moments.  Applying Theorem~\ref{th:main} we may obtain the following stronger bound. Namely, assuming the conditions $\Cond$ for $1 \le p \le c(\alpha) \log n$ we get an error bound of next order
\begin{equation*}
\E^\frac{1}{p} [\Delta_n^{*}]^p \le n^{-1} \log^\frac{2}{1-2\alpha} n .
\end{equation*}
Similarly to~\eqref{eq: rate of convergence} this inequality implies that
\begin{equation}\label{eq: rate of convergence optimal}
\Pb \left( \Delta_n^{*} \geq K  \right) \le \frac{C^p\log^\frac{2p}{1-2\alpha} n }{K^p n^p},
\end{equation}
which is optimal up to a power of  logarithm since on may take $K \gg n^{-1}$ (see also~\cite{GotzeTikh2011rateofconv},~\cite{TaoVu2013} and~\cite{GotzeTikh2014rateofconv} using additional assumptions). A direct corollary of the last bound is the following estimate
$$
\Pb \left (\left|\frac{N[x - \frac{\xi}{2 n}; x +  \frac{\xi}{2 n}  ]}{\xi} - g_{sc}(x) \right | \geq \frac{ K}{\xi}\right) \le \frac{C^p\log^\frac{2p}{1-2\alpha} n }{K^p n^p},
$$
valid for all $\xi > 0$, where $N[I]: = N_I(\W)$ with $I: = [E - \frac{\xi}{2 n}; E +  \frac{\xi}{2 n}  ]$. This means the semicircle law holds on a short scale as well.

Another application of Theorem~\ref{th:main} is the following result which shows the rigidity of the eigenvalues. Let us define the  quantile  position of the $j$-th eigenvalue by
$$
\gamma_j: \quad \int_{-\infty}^{\gamma_j} g_{sc}(\lambda) \, d\lambda = \frac{j}{n}, \quad 1 \le j \le N.
$$
We will show that under conditions $\Cond$ with high probability for all $1 \le j \le n$  the following inequality holds
\begin{equation}\label{eq: rigidity}
|\lambda_j - \gamma_j| \leq K  [\min(j, n-j+1)]^{-\frac{1}{3}}n^{-\frac{2}{3}},
\end{equation}
where $K$ is of logarithmic order. It is easy to see that up to a logarithmic factor the distance between $\lambda_j$ and $\gamma_j$ is of order $n^{-1}$ in the bulk of spectrum and of order $n^{-\frac23}$ at the edges. To prove~\eqref{eq: rigidity} we shall apply Theorem~\ref{th: stronger bound for imag part} as well. For previous results we refer the interested reader to~\cite{Gustavsson2005},~\cite{ErdKnowYauYin2013}[Theorem~7.6],~\cite{ErdKnowYauYin2013a}[Theorem~2.13],~\cite{GotzeTikh2011rateofconv}[Remark~1.2],~\cite{LeeYin2014}[Theorem~3.6] and~\cite{Schlein2014}[Theorem~4].

We may also show delocalization of  eigenvectors of $\W$. Let us denote by $\uu_j: = (u_{j1}, ... , u_{jn})$ the eigenvectors of $\W$ corresponding to the eigenvalue $\lambda_j$.  Assuming condition $\Cond$ we have that with high probability
$$
\max_{1 \le j,k \le n} |u_{jk}|^2 \leq \frac{K }{n}.
$$
For previous and related results we refer the interested reader to the corresponding theorems in~\cite{ErdosSchleinYau2009}~\cite{GotzeTikh2011rateofconv},~\cite{ErdKnowYauYin2013a} and~\cite{ErdKnowYauYin2012}.
\subsection{Sketch of the proof}\label{Sketch of the proof}
Let $\Lambda_n(z): = m_n(z) - s(z)$. Applying Lemma~\ref{appendix inequality for lambda} (see \cite{Schlein2014}[Proposition~2.2]) it is shown in Section~\ref{proof of the main result} that one may estimate $\E |\Lambda_n(z)|^p$ and $\E |\imag \Lambda_n(z)|^p$ via $\E |T_n(z)|^p$ (see definition~\eqref{definition of T}) choosing one of the bounds depending on whether $z$ is near the edge of the spectrum or away from it.

To estimate $\E |T_n(z)|^p$ we extend the methods developed in~\cite{GotzeTikh2014rateofconv}[Theorem~1.2] and~\cite{GotTikh2015}. The bound for $\E |T_n(z)|^p$ is given in Theorem~\ref{th: general bound}. The crucial step in~\cite{GotTikh2015},~\cite{GotzeTikh2014rateofconv} was to show that finiteness of  eight moments suffices to show that $\max_{1\le j \le n}\E |\RR_{jj}(z)|^p \le C_0^p$ for all $1 \le p \le C(nv)^\frac14$ and $v \geq v_0$. The proof of this fact was based on the descent method developed in~\cite{Schlein2014}[Lemma~3.4] (see Lemma~\ref{main lemma} below), but used in proving bounds for moments of the diagonal entries $\RR_{jj}(z)$ only. In this paper we develop this approach to estimate the off-diagonal entries of the resolvent as well assuming $\CondTwo$ (see Lemma~\ref{lemma off diagonal entries assumption} below).
This provides improved bounds for $\E|\varepsilon_{j2}|^q$, where $\varepsilon_{j2}$ is the quadratic form defined in~\eqref {eq: R_jj representation}. We would like to emphasize that
this estimate is crucial for
the proof of Theorem~\ref{th:main}, assuming conditions $\Cond$. Similarly we establish a bound $\max_{1\le j \le n} \E |\RR_{jj}(z)|^p \le C_0^p$, which is  valid on the whole real line rather then on the support of the semicircle law only as in ~\cite{GotTikh2015},~\cite{GotzeTikh2014rateofconv}. The details may be found in Section~\ref{sec: diagonal entries}, Lemma~\ref{main lemma}.

Note that $\E |T_n(z)|^p$ is bounded in terms of $\E \imag^p \RR_{jj}$. In Section~\ref{sec: imagionary part}, Lemma~\ref{lemma: imag part of R_jj} we show that $\max_{1\le j \le n} \E \imag^q \RR_{jj}(z)$ may be estimated by $\imag^q s(z)$ with some additional correction term (see definition~\eqref{eq: Psi definition} of $\Psi(z)$). Since we can derive explicit bounds for $\imag s(z)$ inside as well as outside the limit  spectrum  we
	are able to control the size for $\E |T_n(z)|^q$ as well as $\E |\imag \Lambda_n(z)|^p$  on the whole real line  in terms of the quantity  $\gamma$ (see~\eqref{eq: gamma definition}). This is another key fact for the proof of Theorem~\ref{th: stronger bound for imag part}.

\subsection{Notations} Throughout the paper we will use the following notations. We assume that all random variables are defined on common probability space $(\Omega, \mathcal F, \Pb)$ and denote by $\E$ the mathematical expectation with respect to $\Pb$.

We denote by $\R$ and $\C$ the set of all real and complex numbers. We also define $\C^{+}: = \{z \in \C: \imag  z \geq 0\}$. Let $\T = [1, ... , n]$ denotes the set of the first $n$ positive integers. For any $\J \subset \T$ introduce $\T_{\J}: = \T \setminus \J$.

We shall systematically use for any matrix $\W$ together with its resolvent $\RR$ and Stieltjes transform $m_n$
the corresponding quantities $\W^{(\J)}, \RR^{(\J)}, m_n^{(\J)}$ for the corresponding sub matrix with entries $X_{jk}, j,k \in \T\setminus \J$.

By $C$ and $c$ we denote some positive constants, which may depend on $\alpha, u$ and $V$, but not on $n$.

For an arbitrary matrix $\A$ taking values in $\C^{n \times n}$ we define the operator norm by $\|\A\|: = \sup_{x \in \R^n: \|x\| = 1} \|\A x\|_2$, where $\|x\|_2 : = \sum_{j = 1}^n |x_j|^2$. We also define the Hilbert-Schmidt norm by $\|\A\|_2: = \Tr \A \A^{*} = \sum_{j,k = 1}^n |\A_{jk}|^2$.

\subsection{Acknowledgment} We would like to thank L.~Erd{\"o}s and H.-T.~Yau for drawing our attention to  relevant previous results and papers in connection with the results of this paper, in particular, ~\cite{ErdKnowYauYin2012},~\cite{ErdKnowYauYin2013},~\cite{ErdKnowYauYin2013a} and~\cite{LeeYin2014}.

\section{Proof of the main result}\label{proof of the main result}
We start this section with the recursive representation of the diagonal entries $\RR_{jj}(z) = (\W - z \I)^{-1}$ of the resolvent.  As noted before we shall systematically use for any matrix $\W$ together with its resolvent $\RR$, Stieltjes transform $m_n$ and etc. the corresponding quantities $\W^{(\J)}, \RR^{(\J)}, m_n^{(\J)}$ and etc. for the corresponding sub matrix with entries $X_{jk}, j,k \in \T\setminus \J$. We will often omit the argument $z$ from $\RR(z)$ and write $\RR$ instead. We may express $\RR_{jj}$ in the following way
\begin{equation}\label{eq: R_jj representation 0}
\RR_{jj} = \frac{1}{-z + \frac{X_{jj}}{\sqrt n} - \frac{1}{n}\sum_{l,k \in T_j} X_{jk} X_{jl} \RR_{kl}^{(j)}}.
\end{equation}
Let $\varepsilon_j : = \varepsilon_{1j} + \varepsilon_{2j}+\varepsilon_{3j}+\varepsilon_{4j}$, where
\begin{align*}
&\varepsilon_{1j} =  \frac{1}{\sqrt n}X_{jj}, \quad \varepsilon_{2j} = -\frac{1}{n}\sum_{l \ne k \in T_j} X_{jk} X_{jl} \RR_{kl}^{(j)}, \quad \varepsilon_{3j} = -\frac{1}{n}\sum_{k \in T_j} (X_{jk}^2 -1) \RR_{kk}^{(j)}, \\
&\varepsilon_{4j}= \frac{1}{n} (\Tr \RR - \Tr \RR^{(j)}).
\end{align*}
Using these notations we may rewrite~\eqref{eq: R_jj representation 0} as follows
\begin{equation}\label{eq: R_jj representation}
\RR_{jj} = -\frac{1}{z + m_n(z)} + \frac{1}{z + m_n(z)} \varepsilon_j \RR_{jj}.
\end{equation}
Introduce
\begin{equation}\label{eq:Lambda_b_bn}	
\Lambda_n: = m_n(z) - s(z), \quad b(z): = z + 2 s(z), \quad b_n(z) = b(z) + \Lambda_n,
\end{equation}
and
\begin{equation}\label{definition of T}
T_n: = \frac{1}{n} \sum_{j=1}^n \varepsilon_j \RR_{jj},
\end{equation}
Applying~\eqref{eq: R_jj representation}  we arrive at  the following representation for $\Lambda_n$ in terms of $T_n$ and $b_n$
$$
\Lambda_n = \frac{T_n}{z + m_n(z) + s(z)} = \frac{T_n}{b_n(z)}.
$$
From Lemma~\ref{appendix inequality for lambda} of the Appendix it follows that for all $v > 0$ and $u \le 2+v$
(using the quantities \eqref{eq:Lambda_b_bn})
\begin{equation}\label{eq: abs value lambda main result section}
|\Lambda_n| \le C \min\left\{\frac{|T_n|}{|b(z)|}, \sqrt{|T_n|}\right\}.
\end{equation}
Moreover, for all $v>0$ and $|u| \le u_0$
\begin{equation}\label{eq: abs imag lambda main result section}
|\imag\Lambda_n| \le C \min\left\{\frac{|T_n|}{|b(z)|}, \sqrt{|T_n|}\right\}.
\end{equation}
This means that in order to bound $\E |\Lambda_n|^p$ (or $\E|\imag \Lambda|^p$ respectively) it is enough to estimate $\E |T_n|^p$.

Let us introduce the following region in the complex plane:
\begin{equation}\label{eq: definition reginon D}
\mathbb D: = \{z = u+iv \in \C: |u| \le u_0, V \geq v \geq v_0: = A_0 n^{-1} \},
\end{equation}
where $u_0, V$ are arbitrary fixed positive real numbers and $A_0$ is some large constant defined below.

The following theorem provides a a general bound for $\E |T_n|^p$ for all $z \in \mathbb D$ in terms of diagonal resolvent entries. To formulate the result of the theorem we need to introduce additional notations. Let
\begin{equation}\label{definition of A}
\mathcal A(q): = \max_{|\J| \le 1} \max_{j \in \T_\J} \E^{\frac{1}{q}} \imag^q \RR^{(\J)}_{jj},
\end{equation}
where $\J$ may be an empty set or one point set. We also denote
\begin{equation}\label{eq: definition of E}
\mathcal E_p: = \frac{p^p\mathcal A^p(\kappa p)}{(nv)^p}   +  \frac{ p^{2p}}{(nv)^{2p}}  + \frac{ |b(z)|^\frac{p}{2}\mathcal A^\frac{p}{2}(\kappa p)}{(nv)^p },
\end{equation}
where $\kappa = \frac{16}{1 - 2\alpha}$.
\begin{theorem}\label{th: general bound}
Assume that the conditions $\CondTwo$ hold and $u_0 > 2$ and $V > 0$. There exist  positive constants $A_0, A_1$ and $C$
depending on  $\alpha, u_0$ and $V$ such that for all $z \in \mathbb D$ we have
\begin{equation}\label{eq: general bound}
\E|T_n|^p \le C^p \mathcal E_p,
\end{equation}
where $1 \le p \le A_1 (nv)^{\frac{1-2\alpha}{2}}$.
\end{theorem}

The proof of Theorem~\ref{th: general bound} is one of the crucial steps in the proof of the main result and will be given in the next section. Since \eqref{eq: general bound} is an estimate in terms of the imaginary part of diagonal resolvent entries we refer to it as the main 'general' bound. We finish this section with the proof of Theorems~\ref{th:main} and~\ref{th: stronger bound for imag part}.

\begin{proof}[Proof of Theorem~\ref{th:main}] From Lemmas~\ref{appendix: lemma trunc 1}--~\ref{appendix: lemma trunc 2} of the Appendix it follows that we may assume that
$$
|X_{jk}| \le D n^\alpha \text{ for all }  1 \le j,k \le n
$$
and some $D: = D(\alpha) > 0$.
		
Applying Theorem~\ref{th: general bound} we will show in the section~\ref{sec: imagionary part}, Lemma~\ref{lemma: imag part of R_jj}, that there exist constants $H_0$ depending on $u_0, V$ and $A_0, A_1$ depending on $\alpha$ and $H_0$ such that
\begin{equation}\label{eq: bound for A (alpha p)}
\mathcal A^{p}(\kappa p) \le H_0^p \imag^p s(z) + \frac{H_0^p p^{2p}}{(nv)^p}.
\end{equation}
for all $1 \le p \le A_1(nv)^\frac{1-2\alpha}{2}$ and $z \in \mathbb D$.  This inequality and Theorem~\ref{th: general bound} together imply that
\begin{align}\label{eq: T n 5 step}
\E|T_n|^p &\le \frac{C^p p^p \imag^p s(z)}{(nv)^p} +  \frac{C^p p^{3p}}{(nv)^{2p}}  + \frac{C^p |b(z)|^\frac{p}{2} \imag^\frac{p}{2}s(z)}{(nv)^p } + \frac{C^p |b(z)|^\frac{p}{2} p^p}{(nv)^\frac{3p}{2} } .
\end{align}
with some new constant $C$ which depends on $H_0$. To estimate $\E|\imag \Lambda_n|^p$ we may choose one of the bounds~\eqref{eq: abs imag lambda main result section},
depending on whether $z$ is near the edge of the spectrum or away from it. If $|b(z)|^p \geq \frac{C^p p^p}{(nv)^p}$ then we may take the bound
$$
\E|\imag \Lambda_n|^p \le \frac{C^p \E |T_n|^p}{|b(z)|^p}.
$$
The r.h.s. of the last inequality may be estimated applying~\eqref{eq: T n 5 step}.  We get
\begin{align*}
\E|\imag \Lambda_n|^p \le\frac{C^p p^p \imag^p s(z)}{(nv)^p |b(z)|^p} +  \frac{C^p p^{3p}}{(nv)^{2p}|b(z)|^p}  + \frac{C^p  \imag^\frac{p}{2}s(z)}{(nv)^p |b(z)|^\frac{p}{2}} + \frac{C^p p^p}{(nv)^\frac{3p}{2} |b(z)|^\frac{p}{2} }.
\end{align*}
Since $|b(z)|^p \geq \frac{C^p p^p}{(nv)^p}$ the last inequality may be rewritten in the following way
\begin{align*}
\E|\imag \Lambda_n|^p \le\frac{C^p p^p \imag^p s(z)}{(nv)^p |b(z)|^p}  + \frac{C^p  \imag^\frac{p}{2}s(z)}{(nv)^p |b(z)|^\frac{p}{2}} + \frac{C^p p^{2p} }{(nv)^p}.
\end{align*}
It remains to estimate the imaginary part of $s(z)$. Since
$$
\imag^p s(z) \le c^p|b(z)|^p \text { for } |u| \le 2  \text { and }  \imag^p s(z) \le \frac{c^p v^p}{|b(z)|^p} \text{ otherwise}
$$
both inequalities combined yield
\begin{equation}\label{eq: bound for Lambda}
\E|\imag \Lambda_n|^p \le \left(\frac{Cp^2}{nv}\right)^p,
\end{equation}
where we have used as well the fact that $c\sqrt{\gamma + v} \le |b(z)| \le C\sqrt{\gamma + v} $ for all $|u| \le u_0$, $0 < v \le v_1$.
If $|b(z)|^p \le  \frac{C^p p^p}{(nv)^p}$ and $\imag^p s(z) \geq \frac{C^p p^{2p}}{(nv)^p}$ then we may repeat the calculations above and get the bound~\eqref{eq: bound for Lambda}. In the case $\imag^p s(z) \leq \frac{C^p p^{2p}}{(nv)^p}$ we take the bound proportional to $|T_n|^\frac12$ and obtain the following inequality
$$
\E |\imag \Lambda_n|^p \le \E|T_n|^\frac{p}{2} \le \left(\frac{C p^2}{nv}\right)^p.
$$
Similar arguments are applicable to  $\E|\Lambda_n|^p$.
\end{proof}

\begin{proof}[Proof of Theorem~\ref{th: stronger bound for imag part}] From Theorem~\ref{th: general bound} we may  conclude that
\begin{align}\label{eq: T n 6 step}
\E|T_n|^p &\le \frac{C^p p^p \imag^p s(z)}{(nv)^p} +  \frac{C^p p^{3p}}{(nv)^{3p}}  + \frac{C^p |b(z)|^\frac{p}{2} \imag^\frac{p}{2}s(z)}{(nv)^p }
+ \frac{C^p |b(z)|^\frac{p}{2} p^p}{(nv)^\frac{3p}{2} } .
\end{align}
Applying Lemma~\ref{appendix inequality for lambda} we get
$$
\E|\imag \Lambda_n|^p \le \frac{\E |T_n|^p}{|b(z)|^p}.
$$
This inequality together with~\eqref{eq: T n 5 step} leads to
\begin{align}\label{eq: Lambda_n stronger bound}
\E|\imag \Lambda_n|^p \le\frac{C^p p^p \imag^p s(z)}{(nv)^p |b(z)|^p} +  \frac{C^p p^{3p}}{(nv)^{2p}|b(z)|^p}  + \frac{C^p  \imag^\frac{p}{2}s(z)}{(nv)^p |b(z)|^\frac{p}{2}} + \frac{C^p p^p}{(nv)^\frac{3p}{2} |b(z)|^\frac{p}{2} }.
\end{align}
Since $c\sqrt{\gamma + v} \le |b(z)| \le C\sqrt{\gamma + v} $ for all $|u| \le u_0$, $0 < v \le v_1$ and
$$
\frac{c v}{\sqrt{\gamma + v}} \le \imag s(z) \le \frac{c v}{\sqrt{\gamma + v}} \quad \text{ for all } 2 \le |u| \le u_0, 0 < v \le v_1,
$$
we finally get
\begin{align}\label{eq: Lambda_n stronger bound 2}
\E|\imag \Lambda_n|^p \le \frac{C^p p^p }{n^p (\gamma + v)^p} +  \frac{C^p p^{3p}}{(nv)^{2p} (\gamma + v)^\frac{p}{2}}  + \frac{C^p}{n^p v^\frac{p}{2} (\gamma + v)^\frac{p}{2}} + \frac{C^p p^p}{(nv)^\frac{3p}{2} (\gamma + v)^\frac{p}{4} }.
\end{align}
This bound concludes the proof of the theorem.
\end{proof}

\section{Proof of the general bound of Theorem \ref{th: general bound} }

In the next section we will show that there exist positive  constants $C_0, A_0$ and $A_1$ (explicit dependence on $\alpha, u$ and $\V$ will be given later) such that for all $z \in \mathbb D$ and $1 \le p \le A_1 (nv)^{\frac{1-2\alpha}{2}}$ the following bound holds
\begin{equation}\label{eq: main assumption about diagonal entries of resolvent}
\max_{j \in \T}\E |\RR_{jj}(z)|^p \le C_0^p.
\end{equation}
Similarly we prove that
\begin{equation}\label{eq: main assumption about diagonal entries of resolvent 2}
\E \frac{1}{|z+m_n(z)|^p} \le C_0^p.
\end{equation}
The proof is given in Lemma~\ref{main lemma}. In the current section we will assume
that~\eqref{eq: main assumption about diagonal entries of resolvent} and~\eqref{eq: main assumption about diagonal entries of resolvent 2}  hold. The rest of this section is devoted to the proof of Theorem~\ref{th: general bound}.

\begin{proof}[Proof of Theorem~\ref{th: general bound}]
For the proof we will apply the techniques developed in~\cite{GotzeTikh2014rateofconv} and~\cite{GotTikh2015}. Recalling the definition of $T_n$ (see~\eqref{definition of T}) we may rewrite it in the following way
$$
T_n = \frac{1}{n} \sum_{j=1}^n \varepsilon_{4j} \RR_{jj} + \frac{1}{n} \sum_{\nu = 1}^3 \sum_{j=1}^n \varepsilon_{\nu j} \RR_{jj}.
$$
Since
\begin{equation}\label{eq: derivative mn}
\sum_{j=1}^n \varepsilon_{4j} \RR_{jj} = \frac{1}{n}\Tr \RR^2 = m_n'(z)
\end{equation}
we get that
$$
T_n = \frac{m_n'(z)}{n}  + \frac{1}{n} \sum_{\nu = 1}^3 \sum_{j=1}^n \varepsilon_{\nu j} \RR_{jj} = \frac{m_n'(z)}{n} + \widehat T_n,
$$
where we denoted
$$
\widehat T_n:= \frac{1}{n} \sum_{\nu = 1}^3 \sum_{j=1}^n \varepsilon_{\nu j} \RR_{jj}.
$$
Let us introduce the function $\varphi(z) = \overline{z} |z|^{p-2}$. Then
\begin{align*}
\E|T_n|^p = \E T_n \varphi(T_n)  = \frac{1}{n}\E m_n'(z) \varphi(T_n) + \E \widehat T_n\varphi(T_n).
\end{align*}
Applying the result of Lemma~\ref{appendix lemma resolvent square inequalities}
we estimate the first term on the r.h.s. of the previous equation via
\begin{align}\label{eq: T_n^p bound 1}
\frac{1}{n}\big|\E m_n'(z) \varphi(T_n)\big | &\le \frac{1}{nv} \E^{\frac{1}{p}} \imag^p m_n(z) \E^{\frac{p-1}{p}} |T_n|^p \le \frac{\mathcal A(p)}{nv} \E^{\frac{p-1}{p}} |T_n|^p .
\end{align}
It follows  that
\begin{align}\label{eq: T_n^p first step}
\E|T_n|^p \le |\E \widehat T_n\varphi(T_n)| + \frac{\mathcal A(p)}{nv} \E^{\frac{p-1}{p}} |T_n|^p.
\end{align}
To simplify all calculations we shall systematically apply
the recursion inequality of Lemma~\ref{appendix eq: inequality for x power p 2} which states that if $0 < q_1 \le q_2 \le ... \le q_k < p$ and $c_j, j = 0, ... , k$ are positive numbers such that
$$
x^p \le c_0 + c_1 x^{q_1} + c_2 x^{q_2} + ... + c_k x^{q_k}
$$
then
$$
x^p \le \beta \left[ c_0 + c_1^{\frac{p}{p-q_1}} + c_2^{\frac{p}{p-q_2}} + ... + c_k^{\frac{p}{p-q_k}} \right],
$$
where
$$
\beta: = \prod_{\nu=1}^{k} 2^{\frac{p}{p-q_\nu}} \le 2^{\frac{kp}{p-q_k}}.
$$
We now apply Lemma~\ref{appendix eq: inequality for x power p 2} to inequality~\eqref{eq: T_n^p first step} with $c_0 = |\E\widehat T_n\varphi(T_n)|$, $c_1=\frac{\mathcal A(p)}{nv}$ and   $q_1 = p-1$, obtaining
\begin{equation}\label{eq: T_n^p 2 step}
\E |T_n|^p \le C^p |\E \widehat T_n\varphi(T_n)|+  \frac{C^p \mathcal A^p(p)}{(nv)^p},
\end{equation}
with some absolute constant $C>0$.
Now we consider the term $\E \widehat T_n \varphi(T_n)$. We split it into three parts with respect to $\varepsilon_{\nu j}, \nu = 1,2,3$, obtaining
$$
\E \widehat T_n \varphi(T_n) = \frac{1}{n} \sum_{\nu = 1}^3 \sum_{j=1}^n \E\varepsilon_{\nu j} \RR_{jj}\varphi(T_n) = \mathcal A_1 + \mathcal A_2 + \mathcal A_3.
$$
We may rewrite $\mathcal A_{\nu}$ as a sum of two terms of the general form
\begin{align*}
&\mathcal A_{\nu 1}: = -\frac{1}{n}\E \sum_{j=1}^n \varepsilon_{\nu j} a_n^{(j)} \varphi(T_n), \\
&\mathcal A_{\nu 2}: = \frac{1}{n} \sum_{j=1}^n \E \varepsilon_{\nu j}\left[\RR_{jj} + a_n^{(j)} \right] \varphi(T_n),
\end{align*}
where
$$
a_n(z) = \frac{1}{z + m_n(z)} \, \text{ and } \, a_n^{(j)}(z) = \frac{1}{z + m_n^{(j)}(z)}.
$$
\subsection{Bound for \texorpdfstring{$\mathcal A_{\nu 1}, \nu = 1, 2, 3$}{Bound for the first term}}
For $\nu = 1$ the bound is a direct application of Rosenthal's inequality. The estimates for $\nu = 2, 3$ are more involved.
\subsubsection{Bound for $\mathcal A_{1 1}$}
We decompose $\mathcal A_{11}$ into a sum of two terms
\begin{align*}
&\mathcal B_{11}: = -\frac{1}{n}\E \sum_{j=1}^n \varepsilon_{1 j} a_n \varphi(T_n),\\
&\mathcal B_{12}: = \frac{1}{n}\E \sum_{j=1}^n \varepsilon_{ 1 j} [a_n - a_n^{(j)}] \varphi(T_n),
\end{align*}
From H{\"o}lder's inequality and Lemma~\ref{appendix lemma sum of varepsilon_1} with $q = 1$ it follows that
\begin{align}\label{eq: A_11}
&|\mathcal B_{11}| \le C\E^{\frac{1}{2p}} \left|\frac{1}{n} \sum_{j=1}^n \varepsilon_{1j} \right|^{2p}  \E^{\frac{p-1}{p}}|T_n|^p \le
\frac{Cp}{n} \E^{\frac{p-1}{p}}|T_n|^p.
\end{align}
To estimate $\mathcal B_{12}$ we first need to bound $[a_n - a_n^{(j)}]$. Applying the Schur complement formula (see, for example,~\cite{GotzeTikh2014rateofconv}[Lemma~7.23] or~\cite{GotTikh2003}[Lemma~3.3]) we get
\begin{equation}\label{eq: distance between traces}
\Tr \RR - \Tr \RR^{(j)} = \left( 1 + \frac{1}{n} \sum_{k,l \in \T_j} X_{jl} X_{jk} [(\RR^{(j)})^2]_{kl} \right )\RR_{jj} = \RR_{jj}^{-1} \frac{ d \RR_{jj}}{dz}.
\end{equation}
This equation and Lemma~\ref{appendix lemma resolvent inequalities 1}[Inequality~\eqref{appendix lemma resolvent inequality 2}] yield that
$$
|a_n - a_n^{(j)}| \le |m_n - m_n^{(j)}| |a_n a_n^{(j)}| \le \frac{1}{nv} \imag \RR_{jj} |\RR_{jj}|^{-1} |a_n a_n^{(j)}|.
$$
We have applying H{\"o}lder's inequality and Lemma~\ref{appendix lemma sum of varepsilon_1} with $q = 2$
\begin{align*}
|\mathcal B_{12}| &\le \frac{1}{nv}\E \left| \frac{1}{n} \sum_{j=1}^n \varepsilon_{ 1 j}^2\right|^{\frac{1}{2}}  \left|\frac{1}{n}  \sum_{j=1}^n  \imag^2 \RR_{jj} |\RR_{jj}|^{-2} |a_n a_n^{(j)}|^2  \right|^{\frac{1}{2}} |T_n|^{p-1} \\
& \le \frac{C p^{\frac{1}{2}} }{n^{\frac32 }v} \E^{\frac{p-1}{p}} |T_n|^p \E^{\frac{1}{2p} }\left|\frac{1}{n}  \sum_{j=1}^n  \imag^2 \RR_{jj} |\RR_{jj}|^{-2} |a_n a_n^{(j)}|^2  \right|^{p}.
\end{align*}
In view of the definition of $\mathcal A(q)$ (see~\eqref{definition of A})  and again H{\"o}lder's inequality we obtain
$$
|\mathcal B_{12}| \le \frac{C p^{\frac{1}{2}} }{n^{\frac32 }v}  \mathcal A(4p)  \E^{\frac{p-1}{p}} |T_n|^p.
$$
Finally
\begin{equation}\label{eq: bound for A 11}
\mathcal A_{1 1} \le \left[\frac{Cp}{n}   + \frac{C p^{\frac{1}{2}} }{n^{\frac32 }v}  \mathcal A(4p)  \right]   \E^{\frac{p-1}{p}}|T_n|^p .
\end{equation}
\subsubsection{Bound for $\mathcal A_{2 1}$ and $\mathcal A_{3 1}$}
Let us introduce the following notation
$$
\widetilde T_n^{(j)} : = \E ( T_n \big| \mathfrak M^{(j)}),
$$
where $\mathfrak M^{(j)}: = \sigma\{X_{lk}, l, k \in \T_j\}$.
Since $\E (\varepsilon_{\nu j} \big| \mathfrak M^{(j)}) = 0$ for $\nu = 2, 3$ it is easy to see that
\begin{align*}
\mathcal A_{\nu 1} = \frac{1}{n} \sum_{j=1}^n \E\varepsilon_{\nu j} a_n^{(j)}[\varphi(T_n) - \varphi(\widetilde T_n^{(j)})].
\end{align*}
Applying the Newton-Leibniz formula (see Lemma~\ref{appendix Taylor formula} for details)  and the simple inequality $(x+y)^q \le e x^q + (q+1)^q y^q, x, y > 0, q \geq 1$ we get with $q=p-2$
$$
|\mathcal A_{\nu 1}| \le \mathcal B_{\nu 1} + \mathcal B_{\nu 2},
$$
where
\begin{align}\label{eq: mathcal B nu 1}
&\mathcal B_{\nu 1}: =  \frac{e p}{n} \sum_{j=1}^n \E|\varepsilon_{\nu j} a_n^{(j)}| |T_n - \widetilde T_n^{(j)} ||\widetilde T_n^{(j)} |^{p-2}, \\
\label{eq: mathcal B nu 2}
&\mathcal B_{\nu 2}: =  \frac{p^{p-2}}{n} \sum_{j=1}^n \E|\varepsilon_{\nu j} a_n^{(j)}| |T_n - \widetilde T_n^{(j)}|^{p-1}.
\end{align}
Since the derivation of estimates of these terms are rather involved we need to split them into several subsections.

\vspace{0.1in}
\noindent
{\it Representation of $T_n - \widetilde T_n^{(j)}$}.  By definition we may write the following representation
$$
T_n - T_n^{(j)} = (\Lambda_n - \Lambda_n^{(j)})(b(z) + 2\Lambda_n^{(j)}) + (\Lambda_n - \Lambda_n^{(j)})^2,
$$
where in our notations $T_n^{(j)}$ is $T_n$ with the matrix $\X$ replaced by its corresponding submatrix.
Let us denote
$$
K^{(j)}: = K^{(j)}(z): = b(z) + 2\Lambda_n^{(j)}.
$$
Since
\begin{equation*}
T_n - \widetilde T_n^{(j)} = T_n - T_n^{(j)} - \E (T_n - T_n^{(j)}\big|\mathfrak M^{(j)})
\end{equation*}
we obtain the inequality
\begin{align}\label{T - T tilde inequality}
|T_n - \widetilde T_n^{(j)}| \le |K^{(j)}| |\Lambda_n - \widetilde \Lambda_n^{(j)}| + |\Lambda_n - \Lambda_n^{(j)}|^2+  \E( |\Lambda_n - \Lambda_n^{(j)}|^2 \big|\mathfrak M^{(j)}),
\end{align}
where $\widetilde \Lambda_n^{(j)}: = \E (\Lambda_n \big|\mathfrak M^{(j)})$.
The equation~\eqref{eq: distance between traces} and Lemma~\ref{appendix lemma resolvent inequalities 1}[Inequality~\eqref{appendix lemma resolvent inequality 2}] yield that
\begin{equation}\label{eq: difference between Lambda and Lambda j}
|\Lambda_n - \Lambda_n^{(j)}| \le \frac{1}{nv} \frac{\imag \RR_{jj}}{|\RR_{jj}|}.
\end{equation}
For simplicity we denote the quadratic form in~\eqref{eq: distance between traces} by
\begin{equation}\label{eq: definition eta}
\eta_j: = \frac{1}{n} \sum_{k,l \in \T_j} X_{jl} X_{jk} [(\RR^{(j)})^2]_{kl}
\end{equation}
and rewrite it as a sum of the three terms
$$
\eta_j = \eta_{0j} + \eta_{1j} + \eta_{2j},
$$
where
\begin{align}\label{eq: definition eta0}
&\eta_{0j}: = \frac{1}{n} \sum_{k \in T_j} [(\RR^{(j)})^2]_{kk} = (m_n^{(j)}(z))', \quad \eta_{1j}: = \frac{1}{n} \sum_{k \neq l \in \T_j} X_{jl} X_{jk} [(\RR^{(j)})^2]_{kl}, \\
&\eta_{2j}: = \frac{1}{n} \sum_{k \in \T_j} [X_{jk}^2 - 1] [(\RR^{(j)})^2]_{kk}.
\end{align}
It follows from~\eqref{eq: distance between traces} and $\Lambda_n - \widetilde \Lambda_n^{(j)} = \Lambda_n - \Lambda_n^{(j)} - \E(\Lambda_n - \Lambda_n^{(j)} \big| \mathfrak M^{(j)})$ that
\begin{align*}
\Lambda_n - \widetilde \Lambda_n^{(j)} &= \frac{1 + \eta_{j0}}{n} [\RR_{jj} - \E(\RR_{jj}\big|\mathfrak M^{(j)})] + \frac{\eta_{1j} +\eta_{2j}}{n}\RR_{jj}
- \frac{1}{n}\E((\eta_{j1} + \eta_{j2})\RR_{jj}\big|\mathfrak M^{(j)}).
\end{align*}
It is easy to see that
$$
|\RR_{jj} - \E(\RR_{jj}\big|\mathfrak M^{(j)})| \le |a_n^{(j)}|( |\hat \varepsilon_j \RR_{jj}| + \E(|\hat \varepsilon_j \RR_{jj}|\big|\mathfrak M^{(j)})),
$$
where $\hat \varepsilon_j= \sum_1^3 \ \varepsilon_{\nu j}$. Applying this inequality and Lemma~\ref{appendix lemma resolvent square inequalities} we may write
\begin{align*}
|\Lambda_n - \widetilde \Lambda_n^{(j)}| &\le \frac{1 + v^{-1}\imag m_n^{(j)}(z)}{n} |a_n^{(j)}|( |\hat \varepsilon_j \RR_{jj}| + \E(|\hat \varepsilon_j \RR_{jj}|\big|\mathfrak M^{(j)})) \\
&+ \frac{|\eta_{1j} +\eta_{2j}|}{n}|\RR_{jj}|
+ \frac{1}{n}\E(|\eta_{j1} + \eta_{j2}||\RR_{jj}|\big|\mathfrak M^{(j)}).
\end{align*}
Let us introduce the following quantity
\begin{equation} \label{definition beta}
\beta : = \frac{1}{2\alpha},
\end{equation}
which will be used many times during the proof. It is easy to see that $\beta > 1$.
Denote by $\zeta$ an arbitrary random variable such that $\E |\zeta|^\frac{4\beta}{\beta-1}$ exists. Then
$$
\E(\varepsilon_{\nu j} |T_n - \widetilde T_n^{(j)}||\zeta|\big|\mathfrak M^{(j)}) \le
|K^{(j)}| [B_1 + ... + B_6] + B_7 + B_8,
$$
where
\begin{align*}
&B_1: = \frac{1 + v^{-1}\imag m_n^{(j)}(z)}{n} |a_n^{(j)}| \E( |\varepsilon_{\nu j} \hat \varepsilon_j \RR_{jj} \zeta|\big|\mathfrak M^{(j)}), \\
&B_2: = \frac{1 + v^{-1}\imag m_n^{(j)}(z)}{n} |a_n^{(j)}| \E(|\hat \varepsilon_j \RR_{jj}|\big|\mathfrak M^{(j)})) \E( |\varepsilon_{\nu j} \zeta|\big|\mathfrak M^{(j)}),\\
&B_3: = \frac{1}{n}\E(|\varepsilon_{\nu j} \eta_{1j}| |\RR_{jj} \zeta| \big|\mathfrak M^{(j)}),\\
&B_4: = \frac{1}{n}\E(|\varepsilon_{\nu j} \eta_{2j}| |\RR_{jj} \zeta| \big|\mathfrak M^{(j)}),\\
&B_5: = \frac{1}{n}\E(|\eta_{1j}| |\RR_{jj}| |\mathfrak M^{(j)}) \E(|\varepsilon_{\nu j} \zeta| \big|\mathfrak M^{(j)}),\\
&B_6: = \frac{1}{n}\E(|\eta_{2j}| |\RR_{jj}| \big|\mathfrak M^{(j)}) \E(|\varepsilon_{\nu j} \zeta| \big|\mathfrak M^{(j)}), \\
&B_7: = \frac{1}{n^2 v^2} \E(|\varepsilon_{\nu j}|\imag^2 \RR_{jj} |\RR_{jj}|^{-2} |\zeta| \big|\mathfrak M^{(j)}),\\
&B_8: = \frac{1}{n^2 v^2} \E(|\varepsilon_{\nu j}\zeta|\big|\mathfrak M^{(j)})\E(|\imag^2 \RR_{jj} |\RR_{jj}|^{-2} \big|\mathfrak M^{(j)}),
\end{align*}
where $B_7$ and $B_8$ are the result of an application of~\eqref{eq: difference between Lambda and Lambda j}.
Let us consider the first term $B_1$. By definition $B_1 \le B_{11} + B_{12} + B_{13}$, where
\begin{align*}
B_{1\mu} : = \frac{1 + v^{-1}\imag m_n^{(j)}(z)}{n} |a_n^{(j)}| \E( |\varepsilon_{\mu j}|^2 |\RR_{jj} \zeta|\big|\mathfrak M^{(j)}), \quad \mu = 1, 2, 3.
\end{align*}
For $\mu = 1$ we may apply
H{\"o}lder's   inequality, Lemma~\ref{appendix lemma varepsilon_1} with $p = 4$ and obtain
\begin{equation} \label{eq: inequality for B11}
B_{11} \le \frac{1 + v^{-1}\imag m_n^{(j)}(z)}{n^2} |a_n^{(j)}| \E^\frac12(|\RR_{jj} \zeta|^2\big|\mathfrak M^{(j)}).
\end{equation}
For $\mu = 2$ we may proceed in a similar way and apply Lemma~\ref{appendix lemma varepsilon_2 4 moment} to get
\begin{equation} \label{eq: inequality for B12}
B_{12} \le \frac{\imag m_n^{(j)} + v^{-1}\imag^2 m_n^{(j)}(z)}{n^2 v} |a_n^{(j)}| \E^\frac12(|\RR_{jj} \zeta|^2\big|\mathfrak M^{(j)}).
\end{equation}
Applying Lemma~\ref{appendix lemma varepsilon_3 small p} for $\mu = 3$ we obtain
\begin{equation} \label{eq: inequality for B13}
B_{13} \le \frac{1 + v^{-1}\imag m_n^{(j)}(z)}{n^2} \left(\frac{1}{n} \sum_{k \in \T_j} |\RR_{kk}^{(j)}|^{2\beta}\right)^{\frac{1}{\beta}} |a_n^{(j)}| \E^\frac{\beta-1}{\beta}(|\RR_{jj} \zeta|^{\frac{\beta}{\beta-1}}\big|\mathfrak M^{(j)}).
\end{equation}
Combining inequalities~\eqref{eq: inequality for B11}--~\eqref{eq: inequality for B13} we get the following bound for $B_1$
\begin{align*}
B_1 &\le \frac{1 + v^{-1}\imag m_n^{(j)}(z)}{n^2} \left(1 + \frac{\imag m_n^{(j)}(z)}{v}   + \left(\frac{1}{n} \sum_{k \in \T_j} |\RR_{kk}^{(j)}|^{2\beta}\right)^{\frac{1}{\beta}} \right)\\
&\qquad\qquad\times|a_n^{(j)}| \E^\frac{\beta-1}{2\beta}(|\RR_{jj}|^\frac{2\beta}{\beta-1}\big|\mathfrak M^{(j)}) \E^\frac{\beta-1}{2\beta}(|\zeta|^\frac{2\beta}{\beta-1}\big|\mathfrak M^{(j)}).
\end{align*}
The term $B_2$ may be estimated by the same arguments as $B_1$. We write
\begin{align*}
B_2 &\le \frac{1 + v^{-1}\imag m_n^{(j)}(z)}{n^2} \left(1 + \frac{\imag m_n^{(j)}(z)}{v} +  \left(\frac{1}{n} \sum_{k \in \T_j} |\RR_{kk}^{(j)}|^{2\beta}\right)^{\frac{1}{\beta}} \right)\\
&\qquad\qquad\times|a_n^{(j)}| \E^\frac{\beta-1}{2\beta}(|\RR_{jj}|^\frac{2\beta}{\beta-1}\big|\mathfrak M^{(j)}) \E^\frac{\beta-1}{2\beta}(|\zeta|^\frac{2\beta}{\beta-1}\big|\mathfrak M^{(j)}).
\end{align*}
We now consider the term $B_3$. If $\nu = 2$ we apply H{\"o}lder's  inequality, Lemmas~\ref{appendix lemma eta_1},~\ref{appendix lemma varepsilon_2 4 moment}, obtaining
\begin{align*}
B_3 &\le \frac{1}{n} \E^\frac14 (\varepsilon_{2j}^4 \big|\mathfrak M^{(j)}) \E^\frac12 (\eta_{1j}^2 \big|\mathfrak M^{(j)}) \E^\frac14 (|\RR_{jj} \zeta|^4 \big|\mathfrak M^{(j)}) \\
&\le \frac{ \imag^\frac12 m_n^{(j)} }{n^2 v^{\frac{1}{2}} } \left(\frac{1}{n} \Tr |\RR^{(j)}|^4 \right)^\frac12  \E^\frac18 (|\RR_{jj}|^8 \big|\mathfrak M^{(j)}) \E^\frac18 (|\zeta|^8 \big|\mathfrak M^{(j)}).
\end{align*}
By the same reasoning an application of Lemma~\ref{appendix lemma varepsilon_3 small p} with $\nu = 3$ will lead us to
\begin{align*}
B_3 &\le \frac{1}{n^2} \left(\frac{1}{n} \Tr |\RR^{(j)}|^4 \right)^\frac12  \left(\frac{1}{n} \sum_{k \in \T_j} |\RR_{kk}^{(j)}|^{2\beta}\right)^{\frac{1}{2\beta}}
\E^\frac{\beta-1}{4\beta}(|\RR_{jj}|^\frac{4\beta}{\beta-1}\big|\mathfrak M^{(j)}) \E^\frac{\beta-1}{4\beta}(|\zeta|^\frac{4\beta}{\beta-1}\big|\mathfrak M^{(j)}).
\end{align*}
Combining the last two inequalities we get the following general estimate for $B_3$
\begin{align*}
B_3 &\le \frac{1}{n^2} \left(\frac{1}{n} \Tr |\RR^{(j)}|^4 \right)^\frac12 \left [\left(\frac{1}{n} \sum_{k \in \T_j} |\RR_{kk}^{(j)}|^{2\beta}\right)^{\frac{1}{2\beta}} +\frac{\imag^\frac12 m_n^{(j)}(z)}{v^{\frac12}} \right]\\
&\qquad\qquad\times\E^\frac{\beta-1}{4\beta}(|\RR_{jj}|^\frac{4\beta}{\beta-1}\big|\mathfrak M^{(j)}) \E^\frac{\beta-1}{4\beta}(|\zeta|^\frac{4\beta}{\beta-1}\big|\mathfrak M^{(j)}).
\end{align*}
Let us consider the term $B_4$. For $\nu = 2$ we apply Lemmas~\ref{appendix lemma eta_3},~\ref{appendix lemma varepsilon_2 4 moment} and obtain
\begin{align*}
B_4 &\le \frac{1}{n} \E^\frac14 (\varepsilon_{2j}^4 \big|\mathfrak M^{(j)}) \E^\frac12 (\eta_{2j}^2 \big|\mathfrak M^{(j)}) \E^\frac14 (|\RR_{jj} \zeta|^4 \big|\mathfrak M^{(j)}) \\
&\le \frac{ \imag^\frac12 m_n^{(j)} }{n^2 v^\frac32 } \left(\frac{1}{n} \sum_{k\in \T_j} \imag^2 \RR_{kk}^{(j)} \right)^\frac12  \E^\frac18 (|\RR_{jj}|^8 \big|\mathfrak M^{(j)}) \E^\frac18 (|\zeta|^8 \big|\mathfrak M^{(j)}).
\end{align*}
By the same arguments as before we get the following estimate for the case $\nu = 3$
\begin{align*}
B_4 &\le \frac{1}{n^2 v} \left(\frac{1}{n} \sum_{k\in \T_j} \imag^2 \RR_{kk}^{(j)} \right)^\frac12   \left(\frac{1}{n} \sum_{k \in \T_j} |\RR_{kk}^{(j)}|^{2\beta}\right)^{\frac{1}{2\beta}}  \E^\frac{\beta-1}{4\beta} (|\RR_{jj}|^\frac{4\beta}{\beta-1} \big|\mathfrak M^{(j)}) \E^\frac{\beta-1}{4\beta} (|\zeta|^\frac{4\beta}{\beta-1}\big|\mathfrak M^{(j)}).
\end{align*}
Finally, the  bound for $B_4$ has the form
\begin{align*}
B_4 &\le \frac{1}{n^2 v}\left(\frac{1}{n} \sum_{k\in \T_j} \imag^2 \RR_{kk}^{(j)} \right)^\frac12  \left[ \frac{\imag^\frac12 m_n^{(j)} }{ v^\frac12 }   +\left(\frac{1}{n} \sum_{k \in \T_j} |\RR_{kk}^{(j)}|^{2\beta}\right)^{\frac{1}{2\beta}}  \right ]\\
&\qquad\qquad \times \E^\frac{\beta-1}{4\beta} (|\RR_{jj}|^\frac{4\beta}{\beta-1} \big|\mathfrak M^{(j)}) \E^\frac{\beta-1}{4\beta} (|\zeta|^\frac{4\beta}{\beta-1}\big|\mathfrak M^{(j)}).
\end{align*}
Obviously, the estimates of $B_5$ and $B_6$ are similar to
those  for $B_3$ and $B_4$ respectively. The same arguments yield that  $B_7$ may be estimated as follows
\begin{align*}
B_7 &\le \frac{1}{n^\frac52 v^2} \left [\left(\frac{1}{n} \sum_{k \in \T_j} |\RR_{kk}^{(j)}|^{2\beta}\right)^{\frac{1}{2\beta}}  +\frac{\imag^\frac12 m_n^{(j)}(z)}{v^{\frac12}} \right ] \\
&\qquad\qquad\times\E^\frac14 (\imag^8 \RR_{jj} \big|\mathfrak M^{(j)}) \E^\frac18 (|\RR_{jj}|^{-16} \big|\mathfrak M^{(j)}) \E^\frac18 (|\zeta|^8 \big|\mathfrak M^{(j)}).
\end{align*}
Since  bound for $B_8$ will be an analog of  $B_7$, it is omitted. We finally collect all bounds for $B_\mu, \mu = 1, ..., 7$ and write
\begin{align}\label{eq: representation for difference in general case}
\E(\varepsilon_{\nu j} |T_n - \widetilde T_n^{(j)}||\zeta|\big|\mathfrak M^{(j)}) & \le \left [ |K^{(j)}|[|a_n^{(j)}|  \Gamma_1 + \Gamma_2] \E^\frac{\beta-1}{4\beta} (|\RR_{jj}|^\frac{4\beta}{\beta-1} \big|\mathfrak M^{(j)})  \right. \nonumber \\
&\qquad\qquad \left. + \Gamma_3  \E^\frac18 (|\RR_{jj}|^{-16} \big|\mathfrak M^{(j)}) \right ] \E^\frac{\beta-1}{4\beta} (|\zeta|^\frac{4\beta}{\beta-1}\big|\mathfrak M^{(j)}),
\end{align}
where we denoted
\begin{align*}
&\Gamma_1: =  \frac{1 + v^{-1}\imag m_n^{(j)}(z)}{n^2} \left[1 + \frac{\imag m_n^{(j)}(z)}{v} +  \left(\frac{1}{n} \sum_{k \in \T_j} |\RR_{kk}^{(j)}|^{2\beta}\right)^{\frac{1}{\beta}}  \right], \\
&\Gamma_2: = \frac{1}{n^2} \left[\left(\frac{1}{n} \Tr |\RR^{(j)}|^4 \right)^\frac12 + \left(\frac{1}{n} \sum_{k\in \T_j} \imag^2 \RR_{kk}^{(j)} \right)^\frac12 \right] \left [\left(\frac{1}{n} \sum_{k \in \T_j} |\RR_{kk}^{(j)}|^{2\beta}\right)^{\frac{1}{2\beta}}  +\frac{\imag^\frac12 m_n^{(j)}(z)}{v^{\frac12}} \right],\\
&\Gamma_3: = \frac{1}{n^\frac52 v^2} \left[ \left(\frac{1}{n} \sum_{k \in \T_j} |\RR_{kk}^{(j)}|^{2\beta}\right)^{\frac{1}{2\beta}} +\frac{\imag^\frac12 m_n^{(j)} }{ v^\frac12 } \right ] \E^\frac14 (\imag^8 \RR_{jj} \big|\mathfrak M^{(j)}).
\end{align*}
We may now  estimate the terms $\mathcal B_{\nu 1}$  and $\mathcal B_{\nu 2}$, defined in~\eqref{eq: mathcal B nu 1} and~\eqref{eq: mathcal B nu 2}, by applying the representation~\eqref{eq: representation for difference in general case} and choosing appropriate random variables $\zeta$.

\vspace{0.1in}
\noindent
{\it Bound for $\mathcal B_{\nu 1}$}.
Recall that
$$
\mathcal B_{\nu 1}: =  \frac{e p}{n} \sum_{j=1}^n \E|\varepsilon_{\nu j} a_n^{(j)}| |T_n - \widetilde T_n^{(j)} ||\widetilde T_n^{(j)} |^{p-2}.
$$
Taking $\zeta = 1$ in~\eqref{eq: representation for difference in general case} we get
\begin{align*}
\mathcal B_{\nu 1} &=  \frac{e p}{n} \sum_{j=1}^n \E|\widetilde T_n^{(j)} |^{p-2} \E(|\varepsilon_{\nu j}| |T_n - \widetilde T_n^{(j)} |\big|\mathfrak M^{(j)}) \\
&\le \frac{e p}{n} \sum_{j=1}^n \E|\widetilde T_n^{(j)} |^{p-2}  |K^{(j)}| \Gamma_1  \E^\frac{\beta-1}{4\beta} ( |a_n^{(j)}|^\frac{8\beta}{\beta-1}   |\RR_{jj}|^\frac{4\beta}{\beta-1} \big|\mathfrak M^{(j)})\\
& + \frac{e p}{n} \sum_{j=1}^n \E|\widetilde T_n^{(j)} |^{p-2}  |K^{(j)}| \Gamma_2  \E^\frac{\beta-1}{4\beta} ( |a_n^{(j)}|^\frac{4\beta}{\beta-1}   |\RR_{jj}|^\frac{4\beta}{\beta-1} \big|\mathfrak M^{(j)})\\
&+ \frac{e p}{n} \sum_{j=1}^n \E|\widetilde T_n^{(j)} |^{p-2}  \Gamma_3    \E^\frac{\beta-1}{4\beta} ( |a_n^{(j)}|^\frac{4\beta}{\beta-1}   |\RR_{jj}|^{-\frac{8\beta}{\beta-1}} \big|\mathfrak M^{(j)})  =: \mathcal T_1 + \mathcal T_2 + \mathcal T_3.
\end{align*}
Applying  H{\"o}lder's inequality we obtain
\begin{align}\label{eq: T 1}
\mathcal T_1 \le \frac{e p}{n} \sum_{j=1}^n \E^{\frac{p-2}{p}} |T_n|^{p}  \E^\frac1p  \Gamma_1^p \E^\frac{1}{2p} |K^{(j)}|^{2p}.
\end{align}
To finish the estimate of $\mathcal T_1$ it remains to bound $\E^\frac{1}{2p} |K^{(j)}|^{2p}$ and $\E^\frac1p  \Gamma_1^p$. Recall that
$$
K^{(j)} = b(z) + 2 \Lambda_n^{(j)} = b(z) + 2(\Lambda_n^{(j)} - \Lambda_n) + 2 \Lambda_n.
$$
We claim that
\begin{equation}\label{eq: inequality for K}
\E^\frac{1}{2p} |K^{(j)}|^{2p} \le C|b(z)| + C\E^{\frac{1}{2p}}|T_n|^p + \frac{C\mathcal A(4p)}{nv}.
\end{equation}
Indeed, applying~\eqref{eq: difference between Lambda and Lambda j} we obtain
$$
\E^\frac{1}{2p} |K^{(j)}|^{2p} \le 2 \E^\frac{1}{2p}|b(z) +  2 \Lambda_n|^{2p} + \frac{2 \mathcal A(4p)}{nv}.
$$
If $|b(z)| \geq |\Lambda_n|/2$ then~\eqref{eq: inequality for K} is obvious. On the other hand, if the opposite inequality holds,  we find
$$
|\Lambda_n| \le \frac{|\sqrt{b^2(z) + 4 T_n} - b(z)|}{2} \le |b(z)| + |T_n|^\frac12.
$$
Consequently, $|\Lambda_n| \le 2|T_n|^\frac{1}{2}$ and we conclude~\eqref{eq: inequality for K}. Calculating the $p$-th moment of $\Gamma_1$ we get
\begin{align}\label{eq: gamma 1}
\E^\frac1p  \Gamma_1^p &\le C\left(\frac{1}{n^2} + \frac{\E^\frac1p \imag^{2p} m_n^{(j)}(z) }{(nv)^2} +  \frac{\E^\frac{1}{2p} \imag^{2p} m_n^{(j)}(z) }{n^2 v} \right ) \nonumber \\
&\le C \left(\frac{1}{n^2} + \frac{\mathcal A^2(2p)}{(nv)^2}  +\frac{\mathcal A(2p)}{n^2 v}\right) \le C \left(\frac{1}{n^2} + \frac{\mathcal A^2(2p)}{(nv)^2}\right),
\end{align}
where we have applied the well known Young inequality valid for all $a, b \geq 0$ and positive $s,t$ such that $\frac{1}{s} + \frac{1}{t} = 1$,
\begin{equation}\label{eq: Young inequality}
a b \le \frac{a^s}{s} + \frac{b^t}{t}.
\end{equation}
The estimate~\eqref{eq: gamma 1} may be simplified further. Indeed, from Lemma~\ref{appendix lemma inequality v le  imag s RR jj} we conclude that $v \le \imag \RR_{jj} |\RR_{jj}|^{-2}$ for all $j \in \T$ and arrive at the following inequality
\begin{equation}\label{eq: inequality for v}
v \le C \E^\frac{1}{2p} \imag^{2p} \RR_{jj} \le C \mathcal A(2p).
\end{equation}
The inequalities~\eqref{eq: inequality for K}, ~\eqref{eq: gamma 1} and~\eqref{eq: inequality for v} together imply that
\begin{align}
\mathcal T_1 &\le C\E^\frac{p-2}{p} |T_n|^p \left[ |b(z)| + \E^{\frac{1}{2p}}|T_n|^p + \frac{\mathcal A(4p)}{nv}\right ] \frac{\mathcal A^2(2p)}{(nv)^2}.
\end{align}
Analogously to~\eqref{eq: T 1} we derive a bound for the term $\mathcal T_2$
\begin{align}\label{eq: gamma 2}
\mathcal T_2 \le \frac{e p}{n} \sum_{j=1}^n \E^{\frac{p-2}{p}} |T_n|^{p}  \E^\frac1p  \Gamma_2^p \E^\frac{1}{2p} |K^{(j)}|^{2p}.
\end{align}
Applying~\eqref{eq: Young inequality} and~\eqref{eq: inequality for v} the reader will have no difficulty in showing that
\begin{align*}
\E^\frac1p  \Gamma_2^p &\le C\left(\frac{\mathcal A^\frac12(2p)}{n^2 v^\frac32} + \frac{\mathcal A(2p)}{n^2 v^2} + \frac{\mathcal A(2p)}{n^2 v} + \frac{\mathcal A^\frac32(2p)}{n^2 v^\frac32}\right) \le  \frac{C \mathcal A(2p)}{n^2 v^2}.
\end{align*}
The last inequality and~\eqref{eq: gamma 2} imply the following bound
\begin{align*}
\mathcal T_2 &\le C \E^\frac{p-2}{p} |T_n|^p \left[ |b(z)| + \E^{\frac{1}{2p}}|T_n|^p + \frac{\mathcal A(4p)}{nv}\right ] \frac{\mathcal A(2p)}{n^2 v^2} .
\end{align*}
By the same reasoning as before
\begin{align*}
\mathcal T_3 \le \E^\frac{p-2}{p} |T_n|^p \left [ \frac{\mathcal A^2(4p)}{n^\frac52 v^2} + \frac{\mathcal A^\frac52 (4p)}{(nv)^\frac52} \right] \le   \frac{C \mathcal A^2(4p) \E^\frac{p-2}{p} |T_n|^p }{n^2 v^2} .
\end{align*}

\vspace{0.1in}
\noindent
{\it Bound for $\mathcal B_{\nu 2}$}.
Recall that
$$
\mathcal B_{\nu 2}: =  \frac{p^{p-2}}{n} \sum_{j=1}^n \E|\varepsilon_{\nu j} a_n^{(j)}| |T_n - \widetilde T_n^{(j)}|^{p-1}.
$$
Taking in~\eqref{eq: representation for difference in general case} $\zeta = |T_n - \widetilde T_n^{(j)}|^{p-2}$ we get
\begin{align*}
\mathcal B_{\nu 2} &= \frac{p^{p-2}}{n} \sum_{j=1}^n \E |a_n^{(j)}| \E(|\varepsilon_{\nu j}  |T_n - \widetilde T_n^{(j)}|^{p-1} \big|\mathfrak M^{(j)}) \\ &\le \frac{p^{p-2}}{n} \sum_{j=1}^n \E |K^{(j)}|  \Gamma_1 \E^\frac{\beta-1}{4\beta} (|a_n^{(j)}|^\frac{8\beta}{\beta-1}|\RR_{jj}|^\frac{4\beta}{\beta-1} \big|\mathfrak M^{(j)}) \E^\frac{\beta-1}{4\beta} (|T_n - \widetilde T_n^{(j)}|^\frac{4\beta(p-2)}{\beta-1}\big|\mathfrak M^{(j)})\\
&+\frac{p^{p-2}}{n} \sum_{j=1}^n \E |K^{(j)}|  \Gamma_2 \E^\frac{\beta-1}{4\beta} (|a_n^{(j)}|^\frac{4\beta}{\beta-1}|\RR_{jj}|^\frac{4\beta}{\beta-1} \big|\mathfrak M^{(j)}) \E^\frac{\beta-1}{4\beta} (|T_n - \widetilde T_n^{(j)}|^\frac{4\beta(p-2)}{\beta-1}\big|\mathfrak M^{(j)}) \\
&+\frac{p^{p-2}}{n} \sum_{j=1}^n \E \Gamma_3  \E^\frac{\beta-1}{4\beta} ( |a_n^{(j)}|^\frac{4\beta}{\beta-1}   |\RR_{jj}|^{-\frac{8\beta}{\beta-1}} \big|\mathfrak M^{(j)}) \E^\frac{\beta-1}{4\beta} (|T_n - \widetilde T_n^{(j)}|^\frac{4\beta(p-2)}{\beta-1}\big|\mathfrak M^{(j)})\\
& =: \mathcal T_4 + \mathcal T_5 + \mathcal T_6.
\end{align*}
It follows from~\eqref{T - T tilde inequality} that
\begin{align*}
\E^\frac{\beta-1}{4\beta} (|T_n - \widetilde T_n^{(j)}|^\frac{4\beta(p-2)}{\beta-1}\big|\mathfrak M^{(j)}) &\le \frac{C^{p} |K^{(j)}|^{p-2}}{(nv)^{p-2}} \E^\frac{\beta-1}{8\beta} (\imag^\frac{8\beta(p-2)}{\beta-1} \RR_{jj}\big|\mathfrak M^{(j)})
\E^\frac{\beta-1}{8\beta} (\RR_{jj}^\frac{-8\beta(p-2)}{\beta-1}\big|\mathfrak M^{(j)})\\
& + \frac{C^p}{(nv)^{2(p-2)}} \E^\frac{\beta-1}{8\beta} (\imag^\frac{16\beta(p-2)}{\beta-1} \RR_{jj}\big|\mathfrak M^{(j)})
\E^\frac{\beta-1}{8\beta} (\RR_{jj}^\frac{-16\beta(p-2)}{\beta-1}\big|\mathfrak M^{(j)}).
\end{align*}
Hence we get
\begin{align*}
\mathcal T_4 &\le \frac{C^p p^{p-2}}{(nv)^{p-2}n} \sum_{j=1}^n \E^\frac12 |K^{(j)}|^{2(p-1)} \E^\frac14  \Gamma_1^4  \E^\frac{\beta-1}{8\beta} \imag^\frac{8\beta(p-2)}{\beta-1} \RR_{jj}  \\
&+\frac{C^p p^{p-2}}{(nv)^{2(p-2)}n} \sum_{j=1}^n \E^\frac12 |K^{(j)}|^{2} \E^\frac14  \Gamma_1^4  \E^\frac{\beta-1}{8\beta} \imag^\frac{16\beta(p-2)}{\beta-1} \RR_{jj} \\
&\le \frac{C^p p^{p-2}}{(nv)^{p-2} n } \sum_{j=1}^n \E^\frac12 |K^{(j)}|^{2(p-1)} \E^\frac14  \Gamma_1^4  \mathcal A^{2(p-2)}(\kappa p) \\
&+ \frac{C^p p^{p-2}}{(nv)^{2(p-2)} n } \sum_{j=1}^n \E^\frac12 |K^{(j)}|^{2} \E^\frac14  \Gamma_1^4  \mathcal A^{2(p-2)}(\kappa p) ,
\end{align*}
where (as introcuced in \eqref{eq: definition of E})
$$
\kappa = \frac{16\beta}{\beta-1}.
$$
Similarly as in the previous bounds of $\mathcal T_1$ we get
\begin{align*}
\mathcal T_4 &\le \frac{C^p p^{p-2} \mathcal A^2(2p) \mathcal A^{p-2}(\kappa p) }{(nv)^2 (nv)^{p-2} } \left[|b(z)|^{p-1} + \E^\frac{p-1}{2p}|T_n|^p + \frac{\mathcal A^{p-1}(4p)}{(nv)^{p-1}}  \right ] \\
&+ \frac{C^p p^{p-2} \mathcal A^2(2p) \mathcal A^{2(p-2)}(\kappa p) }{(nv)^2 (nv)^{2(p-2)} } \left[|b(z)| + \E^\frac{1}{2p}|T_n|^p + \frac{\mathcal A(4p)}{nv}  \right ].
\end{align*}
We may now apply  inequality~\eqref{eq: Young inequality} and obtain
\begin{align*}
\mathcal T_4 &\le \frac{C^p p^{p} |b(z)|^{p} \mathcal A^{p}(\kappa p)}{(nv)^p} +  \frac{C^p |b(z)|^{\frac{p}{2}} \mathcal A^p(2p)}{(nv)^p}   + \frac{C^p p^{p} \mathcal A^{p}(\kappa p)}{(nv)^p} \E^\frac12 |T_n|^p\\
& +  \frac{C^p \mathcal A^p(2p)}{(nv)^p}  \E^\frac14 |T_n|^p
+ \frac{C^p p^{p} \mathcal A^{2p}(\kappa p)}{(nv)^{2p}}  +   \frac{C^p \mathcal A^\frac{3p}{2}(2p)}{(nv)^\frac{3p}{2}} .
\end{align*}
The last inequality may be simplified as follows
\begin{align*}
\mathcal T_4 \le \frac{C^p p^{p} \mathcal A^{p}(\kappa p)}{(nv)^p}  + \frac{C^p p^{p} \mathcal A^{p}(\kappa p)}{(nv)^p} \E^\frac12 |T_n|^p +   \frac{C^p \mathcal A^p(2p)}{(nv)^p} \E^\frac14 |T_n|^p.
\end{align*}
By the same argument we obtain the estimate for $\mathcal T_5$
\begin{align*}
\mathcal T_5 &\le \frac{C^p p^{p} |b(z)|^{p} \mathcal A^{p}(\kappa p)}{(nv)^p} +   \frac{C^p |b(z)|^{\frac{p}{2}} \mathcal A^\frac{p}{2}(2p)}{n^p v^p}
+ \frac{C^p p^{p} \mathcal A^{p}(\kappa p)}{(nv)^p} \E^\frac12 |T_n|^p\\
& +   \frac{C^p \mathcal A^\frac{p}{2}(2p)}{n^p v^p} \E^\frac14 |T_n|^p
+ \frac{C^p p^{p} \mathcal A^{2p}(\kappa p)}{(nv)^{2p}}   + \frac{C^p \mathcal A^p(2p)}{(nv)^\frac{3p}{2}  }.
\end{align*}
Finally, the routine check that $\mathcal T_6$ may be estimated as follows is left to the reader
\begin{align*}
\mathcal T_6 \le \frac{C^p p^{p} \mathcal A^{p}(\kappa p)}{(nv)^p} + \left [ \frac{\mathcal A^p(4p)}{n^\frac{5p}{4} v^p} + \frac{\mathcal A^\frac{5p}{4} (4p)}{(nv)^\frac{5p}{4}} \right] \le \frac{C^p p^{p} \mathcal A^{p}(\kappa p)}{(nv)^p}.
\end{align*}

\subsubsection{Combining  bounds for $\mathcal A_{\nu 1}$}
We may now collect all bounds for $\mathcal T_\nu, \nu =1, ..., 6$ and~\eqref{eq: bound for A 11} and insert them into~\eqref{eq: T_n^p 2 step} and apply Lemma~\ref{appendix eq: inequality for x power p 2}  to conclude
\begin{align}\label{eq: T n 3 step}
\E|T_n|^p &\le \frac{C^p p^p\mathcal A^p(\kappa p)}{(nv)^p}  +  \frac{C^p p^{2p}}{(nv)^{2p}} + \frac{C^p |b(z)|^\frac{p}{2}\mathcal A^\frac{p}{2}(\kappa p)}{(nv)^p }  +C^p\sum_{\nu =1}^{3}\mathcal A_{\nu 2}.
\end{align}
To finish the proof of the theorem it remains to estimate $\sum_{\nu =1}^{3}\mathcal A_{\nu 2}$.

\subsection{Bound for \texorpdfstring{$\mathcal A_{\nu 2}, \nu = 1, 2, 3$}{Bound for the second term}}. Recall that
$$
\mathcal A_{\nu 2}: = \frac{1}{n} \sum_{j=1}^n \E \varepsilon_{\nu j}\left[\RR_{jj} + a_n^{(j)} \right] \varphi(T_n).
$$
From the representation  $\RR_{jj} + a_n^{(j)} = \hat \varepsilon_j a_n^{(j)} \RR_{jj}$ (see~\eqref{eq: R_jj representation} with $a_n, \varepsilon_j$ replaced by $a_n^{(j)}$ and $\hat \varepsilon_j$ respectively) it follows that
$$
\mathcal A_{\nu 2} = \frac{1}{n} \sum_{j=1}^n \E \varepsilon_{\nu j} \hat \varepsilon_j a_n^{(j)} \RR_{jj} \varphi(T_n).
$$
Using the  obvious inequality $2 \varepsilon_{\nu j} \hat \varepsilon_j \le \varepsilon_{1j}^2 + \varepsilon_{2j}^2 + \varepsilon_{3j}^2, \nu = 1,2,3$,  we may bound $\mathcal A_{\nu 2}, \nu = 1, 2, 3$, by the sum of two terms (up to some constant) $\mathcal N_{\nu, 1}$ and $\mathcal N_{\nu, 2}, \nu = 1,2,3$, where
\begin{align*}
&\mathcal N_{\nu 1}: =  \frac{e}{n} \sum_{j=1}^n \E |\varepsilon_{\nu j}|^2 |a_n^{(j)} \RR_{jj}| |\widetilde T_n^{(j)}|^{p-1}, \\
&\mathcal N_{\nu 2}: =  \frac{p^{p-1}}{n} \sum_{j=1}^n \E |\varepsilon_{\nu j}|^2 |a_n^{(j)} \RR_{jj}|  |T_n - \widetilde T_n^{(j)}|^{p-1}.
\end{align*}
Let us consider $\mathcal N_{\nu 1}$.
Applying  H{\"o}lder's inequality  we obtain that
$$
\mathcal N_{\nu 1} \le \frac{C}{n} \sum_{j=1}^n \E^{\frac{p-1}{p}} |\widetilde T_n^{(j)} |^p  \E^{\frac{1}{2p}} \E^{\frac{2p}{\beta}}(|\varepsilon_{\nu j}|^{2\beta} \big|\mathfrak M^{(j)} ) \le \frac{1}{n} \sum_{j=1}^n   \E^{\frac{1}{2p}} \E^{\frac{2p}{\beta }}(|\varepsilon_{\nu j}|^{2\beta}  \big|\mathfrak M^{(j)} ) \E^{\frac{p-1}{p}} |T_n |^p.
$$
Calculating  conditional expectations and applying~\eqref{eq: inequality for v} we conclude
\begin{align*}
\sum_{\nu=1}^{3}\mathcal N_{\nu 1} &\le \frac{C}{n}\left [1+ \frac{\max_j \E^{\frac{1}{2p}} \imag^{2p} m_n^{(j)}(z) }{v} \right] \E^{\frac{p-1}{p}} |T_n |^p \\
&\le \frac{C}{n} \E^{\frac{p-1}{p}} |T_n |^p + \frac{\mathcal A(2p)}{nv} \E^{\frac{p-1}{p}}|T_n|^p \le \frac{C \mathcal A(2p)}{nv} \E^{\frac{p-1}{p}}|T_n|^p.
\end{align*}
We now turn our attention to the second term $\mathcal N_{\nu 2}$. Applying~\eqref{eq: difference between Lambda and Lambda j} and~\eqref{T - T tilde inequality}  we may write
\begin{align}\label{eq: T n - T tilde one more time}
|T_n -  \widetilde T_n^{(j)}| &\le \frac{|K^{(j)}|}{nv}\frac{\imag \RR_{jj}}{ |\RR_{jj}|} + \frac{C}{n^2 v^2}\frac{\imag^2 \RR_{jj}}{ |\RR_{jj}|^2} \nonumber \\
&+ \frac{|K^{(j)}|}{nv}\E\left(\frac{\imag \RR_{jj}}{ |\RR_{jj}|}\big| \mathfrak M^{(j)} \right) + \frac{1}{n^2 v^2}\E\left(\frac{\imag^2 \RR_{jj}}{ |\RR_{jj}|^2}\big| \mathfrak M^{(j)} \right).
\end{align}
Substitution of this inequality in $\mathcal N_{\nu 2}$ will give us
\begin{align*}
\mathcal N_{\nu 2} &\le  \frac{C^p p^{p-1} }{(nv)^{p-1} n} \sum_{j=1}^n \E |K^{(j)}|^{p-1}  |\varepsilon_{\nu j}|^2 \imag^{p-1} \RR_{jj} |a_n^{(j)}| |\RR_{jj}|^{2-p} \\
& + \frac{C^p p^{p-1} }{(nv)^{2(p-1)} n} \sum_{j=1}^n \E |\varepsilon_{\nu j}|^2 \imag^{2(p-1)} \RR_{jj} |a_n^{(j)}| |\RR_{jj}|^{3-2p} \\
&+ \frac{C^p p^{p-1} }{(nv)^{p-1} n} \sum_{j=1}^n \E |K^{(j)}|^{p-1}  |\varepsilon_{\nu j}|^2 \E(\imag^{p-1} \RR_{jj} |\RR_{jj}|^{1-p}\big| \mathfrak M^{(j)}) |a_n^{(j)}| |\RR_{jj}| \\
&+ \frac{C^p p^{p-1} }{(nv)^{2(p-1)} n} \sum_{j=1}^n \E |\varepsilon_{\nu j}|^2 \E(\imag^{2(p-1)} \RR_{jj} |\RR_{jj}|^{2(p-1)} \big |\mathfrak M^{(j)}) |a_n^{(j)}| |\RR_{jj}|\\
&= : \mathcal L_1 + \mathcal L_2 + \mathcal L_3 + \mathcal L_4.
\end{align*}
Let us consider the term $\mathcal L_1$.  We get
\begin{align*}
\mathcal L_1 &\le \frac{C^p p^{p-1} }{(nv)^{p-1} n} \sum_{j=1}^n \E |K^{(j)}|^{p-1}  \E^{\frac{1}{\beta}}(|\varepsilon_{\nu j}|^{2\beta} \big|\mathfrak M^{(j)}) \E^{\frac{\beta-1}{\beta}}(\imag^{p-1} \RR_{jj} |a_n^{(j)}| |\RR_{jj}|^{2-p} \big|\mathfrak M^{(j)}) \\
&\le \frac{C^p p^{p-1} }{(nv)^{p-1} n} \sum_{j=1}^n \E^{\frac{1}{2}} |K^{(j)}|^{2(p-1)} \E^{\frac14} \E^{\frac{4}{\beta}}(|\varepsilon_{\nu j}|^{2\beta} \big|\mathfrak M^{(j)})     \E^{\frac18} \E^{\frac{4(\beta-1)}{\beta}} (\imag^{\frac{2\beta (p-1)}{\beta-1}} \RR_{jj}\big|\mathfrak M^{(j)}).
\end{align*}
We distinguish two cases. If $4(\beta-1)\beta^{-1} \le 1$, then applying Lyapunov's inequality we obtain
$$
\E^{\frac18} \E^{\frac{4(\beta-1)}{\beta}} (\imag^{\frac{2\beta (p-1)}{\beta-1}} \RR_{jj}\big|\mathfrak M^{(j)}) \le \E^{\frac{(\beta-1)}{2\beta}}\imag^{\frac{2\beta (p-1)}{\beta-1}} \RR_{jj}.
$$
In the opposite case use Jensen's inequality to get the following estimate
$$
\E^{\frac18} \E^{\frac{4(\beta-1)}{\beta}} (\imag^{\frac{2\beta (p-1)}{\beta-1}} \RR_{jj}\big|\mathfrak M^{(j)}) \le \E^{\frac18} \imag^{8(p-1)} \RR_{jj}.
$$
Both inequalities lead to the bound
$$
\E^{\frac18} \E^{\frac{4(\beta-1)}{\beta}} (\imag^{\frac{2\beta (p-1)}{\beta-1}} \RR_{jj}\big|\mathfrak M^{(j)}) \le \mathcal A^{p-1}(\kappa p).
$$
Applying this inequality we arrive at a bound for $\mathcal L_1$
\begin{align*}
\mathcal L_1 &\le \frac{C^p p^{p-1} }{(nv)^{p-1} n} \sum_{j=1}^n \E^{\frac{1}{2}} |K^{(j)}|^{2(p-1)} \E^{\frac14} \E^{\frac{4}{\beta}}(|\varepsilon_{\nu j}|^{2\beta} \big|\mathfrak M^{(j)}) \mathcal A^{p-1}(\kappa p).
\end{align*}
Lemmas~\ref{appendix lemma varepsilon_2 4 moment} and~\ref{appendix lemma varepsilon_3 small p} together imply that $\mathcal L_1$ is bounded by the sum of the following terms
\begin{align*}
&\mathcal L_{11}: = \frac{C^p p^{p-1} }{n^p v^{p-1} n} \sum_{j=1}^n \E^{\frac{1}{2}} |K^{(j)}|^{2(p-1)}  \mathcal A^{p-1}(\kappa p), \\
&\mathcal L_{12}: = \frac{C^p p^{p-1} }{(nv)^p n} \sum_{j=1}^n \E^{\frac{1}{2}} |K^{(j)}|^{2(p-1)}  \mathcal A^{p}(\kappa p).
\end{align*}
Since (see~\eqref{eq: inequality for K} for details)
$$
\E^{\frac{1}{2}} |K^{(j)}|^{2(p-1)} \le C\E^{\frac{1}{2}} |T_n|^{p-1} + |b(z)|^{p-1} + \frac{1}{(nv)^{p-1}}
$$
we get
\begin{align*}
\mathcal L_{11} \le \frac{C^p p^{p-1} }{n^p v^{p-1}}  \E^{\frac{p-1}{2p}} |T_n|^{p}  \mathcal A^{p-1}(\kappa p) + \frac{C^p p^{p-1} |b(z)|^{p-1} }{n^p v^{p-1}}    \mathcal A^{p-1}(\kappa p) + \frac{C^p p^{p-1} }{n^{2p-1} v^{2(p-1)}}.
\end{align*}
It remains to apply~\eqref{eq: Young inequality} and~\eqref{eq: inequality for v}. Finally we obtain
\begin{align*}
\mathcal L_{11} &\le \frac{C^p p^{p} \mathcal A^{p}(\kappa p) }{(nv)^{p}} \E^{\frac{1}{2}} |T_n|^p +\frac{C^p p^{p} \mathcal A^{p}(\kappa p) }{(nv)^{p}} .
\end{align*}
The term $\mathcal L_{12}$ is estimated as follows
\begin{align*}
\mathcal L_{12} &\le \frac{C^p p^{p} \mathcal A^{p}(\kappa p) }{(nv)^{p}} \E^{\frac{1}{2}} |T_n|^p +\frac{C^p p^{p} |b(z)|^{p} \mathcal A^{p}(\kappa p) }{(nv)^{p}} +\frac{C^p p^{p} \mathcal A^{p}(\kappa p)}{(nv)^{2p}} + \frac{C^p p^{p} \mathcal A^{p}(\kappa p)}{(nv)^{p}} \\
&\le \frac{C^p p^{p} \mathcal A^{p}(\kappa p) }{(nv)^{p}} \E^{\frac{1}{2}} |T_n|^p +\frac{C^p p^{p} \mathcal A^{p}(\kappa p) }{(nv)^{p}}.
\end{align*}
The bound for $\mathcal L_3$ may be derived in a similar way. It remains to estimate the  terms $\mathcal L_2$ and $\mathcal L_3$. Similarly  as before we obtain
\begin{align*}
\mathcal L_2 &\le \frac{C^p p^{p} \mathcal A^p(\kappa p)}{(nv)^{2p}} + \frac{C^p}{n^p} + \frac{C^p p^{p} \mathcal A^{p}(\kappa p)}{(nv)^{p}} \le \frac{C^p p^{p} \mathcal A^{p}(\kappa p)}{(nv)^{p}}.
\end{align*}
The same estimate holds for $\mathcal L_4$. Finally we arrive at the following inequality for the sum of $\mathcal A_{\nu 2}, \nu = 1, 2, 3$
\begin{align}\label{eq: bound for A alpha 2}
\sum_{\nu = 1}^{3} \mathcal A_{\nu 2} &\le \frac{\mathcal A(2p)}{nv} \E^{\frac{p-1}{p}}|T_n|^p + \frac{C^p p^{p} \mathcal A^{p}(\kappa p) }{(nv)^{p}} \E^{\frac{1}{2}} |T_n|^p +\frac{C^p p^{p}  \mathcal A^{p}(\kappa p) }{(nv)^{p}}.
\end{align}
\subsection{Combining  bounds} We now combine the bounds~\eqref{eq: T n 3 step},~\eqref{eq: bound for A alpha 2} and apply Lemma~\ref{appendix eq: inequality for x power p 2}, obtaining
\begin{align*}
\E|T_n|^p &\le \frac{C^p p^p\mathcal A^p(\kappa p)}{(nv)^p}  +  \frac{C^p p^{2p}}{(nv)^{2p}} + \frac{C^p |b(z)|^\frac{p}{2}\mathcal A^\frac{p}{2}(\kappa p)}{(nv)^p }.
\end{align*}
In view of the definition of $\mathcal E_p$ this concludes the proof of the theorem.
\end{proof}

\section{Bounds for moments of diagonal entries of the resolvent}\label{sec: diagonal entries}
The main result of this section is the following lemma which provides a bound for moments of the diagonal entries of the resolvent.
Recall that (see the definition~\eqref{eq: definition reginon D})
\begin{equation*}
\mathbb D: = \{z = u+iv \in \C: |u| \le u_0, V \geq v \geq v_0: = A_0 n^{-1} \},
\end{equation*}
where $u_0, V > 0 $ are any fixed real numbers and $A_0$ is some large constant determined below.
\begin{lemma}\label{main lemma}
Assuming the conditions $\CondTwo$ there exist a positive constant $C_0$ depending on $u_0, V$ and positive constants $A_0, A_1$ depending on $C_0, \alpha$
such that for all $z \in \mathbb D$ and $1 \le p \le A_1(nv)^{\frac{1-2\alpha}{2}}$ we have
\begin{equation}\label{eq: main lemma first statement}
\max_{j \in \T} \E|\RR_{jj}(z)|^p \le C_0^p
\end{equation}
and
\begin{equation}\label{eq: main lemma second statement}
\E \frac{1}{|z+m_n(z)|^p} \le C_0^p.
\end{equation}
\end{lemma}

The proof of Lemma~\ref{main lemma}  is based on several auxiliary results and will be given at the end of this section. In this proof  will shall  use  ideas from~\cite{GotzeTikh2014rateofconv} and~\cite{GotTikh2015}. One of main ingredients of the proof is the descent method for $\RR_{jj}$ which is based on Lemma~\ref{lemma: recurence relation for R_jj} below and Lemma~\ref{appendix lemma resolvent relations on different v} in the Appendix, which in this form appeared in~\cite{Schlein2014}.

Since $u$ is fixed and $|u| \le u_0$ we shall omit $u$ from the notation of the resolvent and denote $\RR(v): = \RR(z)$. Sometimes in order to simplify notations we shall also omit the  argument $v$ in $\RR(v)$ and just write $\RR$. For any $j \in \T_{\J}$ we may express $\RR_{jj}^{(\J)}$ in the following way (compare~\eqref{eq: R_jj representation 0})
\begin{equation}\label{eq: representation for RR_jj}
\RR_{jj}^{(\J)} = \frac{1}{-z + \frac{X_{jj}}{\sqrt n} - \frac{1}{n}\sum_{l,k \in \T_{\J,j}} X_{jk} X_{jl} \RR_{lk}^{(\J,j)}}.
\end{equation}
Let $\varepsilon_j^{(\J)} : = \varepsilon_{1j}^{(\J)} + \varepsilon_{2j}^{(\J)}+\varepsilon_{3j}^{(\J)}+\varepsilon_{4j}^{(\J)}$, where
\begin{align*}
&\varepsilon_{1j}^{(\J)} =  \frac{1}{\sqrt n}X_{jj}, \quad \varepsilon_{2j}^{(\J)} = -\frac{1}{n}\sum_{l \ne k \in T_j} X_{jk} X_{jl} \RR_{kl}^{(\J,j},
\quad \varepsilon_{3j}^{(\J)} = -\frac{1}{n}\sum_{k \in T_{\J,j}} (X_{jk}^2 -1) \RR_{kk}^{(\J)}(z), \\
&\varepsilon_{4j}^{(\J)}= \frac{1}{n} (\Tr \RR^{(\J)} - \Tr \RR^{(\J,j)}(z)).
\end{align*}
We also introduce the quantities $\Lambda_n^{(\J)}(z) : = m_n^{(\J)} (z) - s(z)$ and
$$
T_n^{(\J)}: = \frac{1}{n} \sum_{j \in \T_{\J}} \varepsilon_j^{(\J)}\RR_{jj}^{(\J)}.
$$
The following lemma, Lemma~\ref{lemma: recurence relation for R_jj}, allows  to  recursively estimate the moments of the diagonal entries of the resolvent.
 The proof of the first part of this lemma may be found in~\cite{Schlein2014} and
it is included here for the readers convenience.
\begin{lemma}\label{lemma: recurence relation for R_jj}
For an arbitrary set $\J \subset \T$ and all $j \in \T_\J$ there exist a positive constant $C_1$ depending on $u_0, V$ only such that for all $z = u + i v$ with $V \geq v > 0$ and $|u| \le u_0$ we have
\begin{equation}\label{inequality for R_jj}
|\RR_{jj}^{(\J)}| \le C_1\Big(1 + |T_n^{(\J)}|^{\frac{1}{2}}|\RR_{jj}^{(\J)}| + |\varepsilon_j^{(\J)}||\RR_{jj}^{(\J)}|\Big)
\end{equation}
and
\begin{equation}\label{inequality for 1/(z+m_n(z))}
\frac{1}{|z+m_n^{(\J)}(z)|} \le C_1\left(1 + \frac{|T_n^{(\J)}|^{\frac{1}{2}}}{|z+m_n^{(\J)}(z)|}\right).
\end{equation}
\end{lemma}
\begin{proof}
We first prove~\eqref{inequality for R_jj}. We may rewrite~\eqref{eq: representation for RR_jj} in the following way
$$
\RR_{jj}^{(\J)}  = -\frac{1}{z + m_n^{(\J)} (z)} + \frac{1}{z + m_n^{(\J)} (z)} \varepsilon_j^{(\J)}  \RR_{jj}^{(\J)}.
$$
Applying the definition of $\Lambda_n^{(j)}$ we rewrite the previous equation as
\begin{equation}\label{eq: representation for RR_jj 2}
\RR_{jj}^{(\J)} = s(z)+ s(z)(\Lambda_n^{(\J)} - \varepsilon_j^{(\J)} ) \RR_{jj}^{(\J)}.
\end{equation}
Since $|s(z)| \le 1$ we get
\begin{equation}\label{inequality for R_jj 1}
|\RR_{jj}^{(\J)}| \le 1 + (|\Lambda_n^{(\J)}| + |\varepsilon_j^{(\J)}| )|\RR_{jj}^{(\J)}|.
\end{equation}
Rewriting~\eqref{eq: representation for RR_jj 2} in a
$$
\RR_{jj}^{(\J)} = s(z) + s(z)(b_n^{(\J)}(z) -  \varepsilon_j^{(\J)} ) \RR_{jj}^{(\J)} - s(z) b(z)\RR_{jj}^{(\J)}.
$$
Since $1+zs(z) + s^2(z) = 0$ we get the following inequality
\begin{equation}\label{inequality for R_jj 2}
|\RR_{jj}^{(\J)}| \le \frac{1}{|s(z)|} + \frac{|b_n^{(\J)}||\RR_{jj}^{(\J)}|}{|s(z)|}+ \frac{|\varepsilon_j^{(\J)}| |\RR_{jj}^{(\J)}|}{|s(z)|}.
\end{equation}
From $|s(z)|^{-1} \le 1 + |z|$,~\eqref{inequality for R_jj 1},~\eqref{inequality for R_jj 2} and Lemma~\ref{appendix inequality for lambda}[Inequality~\eqref{eq: min of abs values lambda}] we conclude that there exists a positive constant $C_1$ such that
\begin{align*}
|\RR_{jj}^{(\J)}| &\le C_1(1 + \min ( |b_n^{(\J)}| ,|\Lambda_n^{(\J)}| ) |\RR_{jj}^{(\J)}| + |\varepsilon_j^{(\J)}||\RR_{jj}^{(\J)}|) \\
&\le C(1 + |T_n^{(\J)}|^{\frac{1}{2}}|\RR_{jj}^{(\J)}| + |\varepsilon_j^{(\J)}||\RR_{jj}^{(\J)}|) .
\end{align*}

Consider now the second inequality~\eqref{inequality for 1/(z+m_n(z))}. From the representation
\begin{equation}\label{eq: representation for 1/(z+m_n(z))}
\frac{1}{z+m_n^{(\J)}(z)} = \frac{1}{z + s(z)} - \frac{\Lambda_n^{(\J)}}{(z+m_n^{(\J)}(z))(z+s(z))}
\end{equation}
we conclude with $|z+s(z)| \ge 1$ that
\begin{equation}\label{inequality for 1/(z+m_n(z)) 1}
\frac{1}{|z+m_n^{(\J)}(z)|} \le 1+ \frac{\Lambda_n^{(\J)}}{|z+m_n^{(\J)}(z)|}.
\end{equation}
Rewriting~\eqref{eq: representation for 1/(z+m_n(z))} as follows we get
$$
\frac{1}{z+m_n^{(\J)}(z)} = \frac{1}{z + s(z)} - \frac{\Lambda_n^{(\J)}+2s(z)+z}{(z+m_n^{(\J)}(z))(z+s(z))} + \frac{2s(z)+z}{(z+m_n^{(\J)}(z))(z+s(z))}.
$$
This equation may be rewritten as
$$
-\frac{s(z)}{z+m_n^{(\J)}(z)} = 1 - \frac{\Lambda_n^{(\J)}+2s(z)+z}{z+m_n^{(\J)}(z)}.
$$
Taking  absolute values and applying the triangular inequality we get
\begin{equation}\label{inequality for 1/(z+m_n(z)) 2}
\frac{1}{|z+m_n^{(\J)}(z)|} \le \frac{1}{|s(z)|} + \frac{1}{|s(z)|}\frac{|\Lambda_n^{(\J)}+2s(z)+z|}{|z+m_n^{(\J)}(z)|}.
\end{equation}
From $|s(z)|^{-1} \le 1 + |z|$,~\eqref{inequality for 1/(z+m_n(z)) 1},~\eqref{inequality for 1/(z+m_n(z)) 2} and~\ref{appendix inequality for lambda}[Inequality~\eqref{eq: min of abs values lambda}] it follows
\begin{equation*}
\frac{1}{|z+m_n^{(\J)}(z)|} \le C_1\left(1 + \frac{|T_n^{(\J)}|^{\frac{1}{2}}}{|z+m_n^{(\J)}(z)|}\right).
\end{equation*}
\end{proof}

Let us introduce the following quantities for an integer $K > 0$
\begin{align}\label{eq: definition of A and F}
&A_{\nu,q}:= A_{\nu,q}^{(K)}(v):=  \max_{\J: |\J| \le \nu + K } \max_{l \neq k \in \T_{\J}} \E|\RR_{lk}^{(\J,j)}(v)|^q, \nonumber \\
&F_{\nu,q}: = F_{\nu,q}^{(K)}(v): =   \max_{\J: |\J| \le \nu + K } \max_{l \in \T_{\J}} \E \imag^q \RR_{ll}^{(\J,j)}(v).
\end{align}
In the following lemma we show that  $A_{1,q}$ is uniformly bounded with respect to $v$ and $n$.
\begin{lemma}\label{lemma off diagonal entries assumption}
Let $\tilde v > 0$ be an arbitrary number and $C_0$ be some large positive constant. Suppose that the conditions $\CondTwo$ hold. There exist positive constant $s_0$ depending on $\alpha$ and positive constants $A_0, A_1$ depending on $C_0, s_0$ such that assuming
\begin{equation}\label{main condition off diagonal}
\max_{\J: |\J| \le K + 2} \max_{l,k \in \T_\J}\E |\RR_{lk}^{(\J)}(v')|^q \le C_0^q
\end{equation}
for all $v' \geq \tilde v, |u| \le u_0$ and $1 \le q \le A_1(nv')^{\frac{1-2\alpha}{2}}$ we have
$$
A_{1,q} \le C_0^q
$$
for all $v \geq \tilde v/s_0, |u| \le u_0$ and $1 \le q \le A_1(nv)^{\frac{1-2\alpha}{2}}$.
\end{lemma}
\begin{proof}
We start from the assumption that we have already chosen the value of $s_0$. Let us fix an  arbitrary $v \geq \tilde v/s_0$, $\J \subset \T$ such that $|\J| \le K+1$ and $j \in \T_\J$.  We may express $\RR_{jk}^{(\J)}$ as follows
$$
\RR_{jk}^{(\J)} = -\frac{1}{\sqrt n} \sum_{l \in \T_{\J,j}} X_{jl} \RR_{lk}^{(\J,j)} \RR_{jj}^{(\J)}.
$$
Applying  H{\"o}lder's inequality we obtain
$$
\E |\RR_{jk}^{(\J)}|^{q} \le n^{-\frac{1}{2}} \E^{\frac{1}{2}} \left| \sum_{l \in \T_{\J,j}} X_{jl} \RR_{lk}^{(\J,j)} \right|^{2q}  \E^{\frac{1}{2}} |\RR_{jj}^{(\J)}|^{2q}.
$$
From Lemma~\eqref{appendix lemma resolvent relations on different v} we conclude that  for $s_0 \geq 1$ the following relation holds
\begin{equation}\label{eq: recurrence relation 0}
|\RR_{jj}^{(\J)}(v)|^{q} \le s_0^{q}|\RR_{jj}^{(\J)}(s_0v)|^{q}.
\end{equation}
We choose $s_0: = 2^{\frac{2}{1-2\alpha}}$. Since $ v': = s_0 v \geq \tilde v$ and $2 q \le A_1 (nv')^{\frac{1-2\alpha}{2}}$ we may apply~\eqref{eq: recurrence relation 0} and the assumption~\eqref{main condition off diagonal} to estimate  $\E^{\frac{1}{2}} |\RR_{jj}^{(\J)}|^{2q} \le (s_0 C_0)^{q}$. We get the following bound
\begin{equation}\label{off diagonal eq 0}
\E |\RR_{jk}^{(\J)}|^{q} \le n^{-\frac{q}{2}} (s_0 C_0)^{q} \E^{\frac{1}{2}} \left| \sum_{l \in \T_{\J, j}} X_{jl} \RR_{lk}^{(\J,j)} \right|^{2q}.
\end{equation}
Since $X_{jk}, k \in \T_{\J,j},$ and $\RR^{(\J,j)}$ are independent we may apply Rosenthal's inequality  and get
\begin{align*}
\E\left| \sum_{l \in \T_{\J,j}} X_{jl} \RR_{lk}^{(\J,j)} \right|^{2q}  \le C^{q}\left( q^{q} \E\left(\sum_{l \in \T_{\J,j}} |\RR_{lk}^{(\J,j)}|^2 \right)^{q} + \mu_{2q} q^{2q}
\sum_{l \in \T_{\J,j}} \E|\RR_{lk}^{(\J,j)}|^{2q} \right).
\end{align*}
From Lemma~\ref{appendix lemma resolvent inequalities 1} [Inequality~\eqref{appendix lemma resolvent inequality 2}] and~\eqref{eq: recurrence relation 0}  we  derive the following inequality
\begin{align}\label{off diagonal eq 1}
\E \left(\sum_{l \in \T_{\J,j}} |\RR_{lk}^{(\J,j)}|^2 \right)^{q} \le \frac{1}{v^{q}}   F_{2,q} \le \frac{s_0^q C_0^q}{v^{q}}.
\end{align}
Since $|X_{jk}| \le D n^{\alpha}$ we get
\begin{align}\label{off diagonal eq 2}
\mu_{2q} \le \mu_4 D^{2q-4}n^{\alpha(2q-4)}.
\end{align}
By definition
\begin{equation}\label{off diagonal eq decomposition}
\sum_{l \in \T_{\J,j}} \E |\RR_{lk}^{(\J,j)}|^{2q} = \sum_{l \in \T_{\J,j,k}} \E |\RR_{lk}^{(\J,j)}|^{2q} + \E |\RR_{kk}^{(\J,j)}|^{2q}.
\end{equation}
The second term on the right hand side of the previous equality can be bounded by~\eqref{eq: recurrence relation 0} and the assumption~\eqref{main condition off diagonal}. To the first term we may apply the following bound
\begin{equation}\label{off diagonal eq first bound for sum}
\sum_{l \in \T_{\J,j,k}} \E |\RR_{lk}^{(\J,j)}|^{2q} \le n A_{2,2q}.
\end{equation}
It follows from~\eqref{off diagonal eq 0}, ~\eqref{off diagonal eq 1},~\eqref{off diagonal eq 2}~\eqref{off diagonal eq decomposition} and~\eqref{off diagonal eq first bound for sum} that
\begin{align}\label{off diagonal recurrent formula 1}
&A_{1,q} \le (C C_0 s_0)^q  \left ( \frac{q^\frac{q}{2} (s_0 C_0)^\frac{q}{2}   }{(nv)^{\frac{q}{2}}}  +  \frac{q^{q}}{n^{\frac{q}{2}(1-2\alpha)}} A_{2,2q}^{\frac{1}{2}}     +  \frac{q^{q} (C_0 s_0)^q}{n^{\frac{q}{2}(1-2\alpha)  } } \right).
\end{align}
To estimate $A_{2,2q}$ we apply the resolvent equality and obtain that for arbitrary $l \neq k \in \T_\J$
$$
|\RR_{lk}^{(\J)}(v) - \RR_{lk}^{(\J)}(s_0v)| \le v(s_0-1) |[\RR^{(\J)}(v) \RR^{(\J)}(s_0v)]_{lk}|.
$$
The Cauchy-Schwartz inequality and Lemma~\ref{appendix lemma resolvent inequalities 1} [Inequality~\eqref{appendix lemma resolvent inequality 2}] together imply that
$$
|\RR_{lk}^{(\J)}(v) - \RR_{lk}^{(\J)}(s_0v)| \le \sqrt{s_0} \sqrt{|\RR^{(\J)}_{ll}(v)| |\RR^{(\J)}_{kk}(s_0 v)|}.
$$
It remains to apply~\eqref{eq: recurrence relation 0} and the assumption~\eqref{main condition off diagonal} to get
$$
A_{2,2q} \le 2^{2q} (s_0 C_0)^{2q} + 2^{2q} s_0^{3q} C_0^{2q}.
$$
It follows from the last inequality and~\eqref{off diagonal recurrent formula 1} that
$$
A_{1,q} \le (C C_0 s_0)^q  \left ( \frac{q^\frac{q}{2} (s_0 C_0)^\frac{q}{2}   }{(nv)^{\frac{q}{2}}}  +  \frac{q^q  (C_0 s_0^{3/2})^q }{n^{\frac{q}{2}(1-2\alpha)}}    \right).
$$
We may choose the constants $A_0$ and $A_1$ depending on $C_0, s_0$ in such way that
$$
A_{1,q} \le C_0^q.
$$
\end{proof}

\begin{lemma}\label{lemma varepsilon_2 assumption}
Let $\tilde v > 0$ be an arbitrary number and $C_0$ be some large positive constant. Suppose that the conditions $\CondTwo$ hold. There exist positive constant $s_0$ depending on $\alpha$ and positive constants $A_0, A_1$ depending on $C_0, s_0$ such that assuming
\begin{equation}\label{main condition 1}
\max_{\J: |\J| \le K + 2} \max_{l,k \in \T_\J}\E |\RR_{lk}^{(\J)}(v')|^q \le C_0^q
\end{equation}
for all $v' \geq \tilde v, |u| \le u_0$ and $1 \le q \le A_1(nv')^{\frac{1-2\alpha}{2}}$ we have
$$
\max_{\J: |\J| \le K} \max_{j \in \T_\J} \E |\varepsilon_{2j}^{(\J)}|^{2q} \le \frac{(C C_0 s_0)^{q} q^{4q}}{(nv)^{2q(1-2\alpha)}}
$$
for all $v \geq \tilde v/s_0, |u| \le u_0$ and $1 \le q \le A_1(nv)^{\frac{1-2\alpha}{2}}$.
\end{lemma}
\begin{proof}
Let us fix an arbitrary $v \geq \tilde v/s_0$, $\J \subset \T$ such that $|\J| \le K$ and $j \in \T_\J$. Applying Lemma~\ref{appendix lemma varepsilon_2} we get
\begin{align*}
\E |\varepsilon_{2j}^{(\J)}|^{2q} &\le C^{p}   \left( \frac{q^{2q}}{(nv)^{q}}   \E[\imag m_n^{(\J,j)}(z)]^{q} + \frac{q^{3q}}{(nv)^{q}} F_{1,q}+\frac{q^{4q}}{n^{2q(1-2\alpha)}} A_{1,2q} \right).
\end{align*}
To estimate $\E[\imag m_n^{(\J,j)}(z)]^{q}$ and $F_{1,q}$ we use Lemma~\eqref{appendix lemma resolvent relations on different v} which states that  for all $s_0 \geq 1$ the following relation holds
\begin{equation}\label{eq: recurrence relation 1}
|\RR_{kk}^{(\J,j)}(v)|^{q} \le s_0^{q}|\RR_{kk}^{(\J,j)}(s_0v)|^{q}.
\end{equation}
We choose $s_0: = 2^{\frac{4}{1-2\alpha}}$. This choice implies that $ v': = s_0 v \geq \tilde v$ and $4 q \le A_1 (nv')^{\frac{1-2\alpha}{2}}$.  We may apply~\eqref{eq: recurrence relation 1} and~\eqref{main condition 1} to estimate
\begin{equation}\label{eq: bounds for RR jj 1}
\E[\imag m_n^{(\J,j)}(z)]^{q} \le s_0^{q} C_0^{q} \text{ and } F_{1,q} \le s_0^{q} C_0^{q}.
\end{equation}
In view of these inequalities we may write
\begin{align}\label{eq: inequality fo ve j2 p}
\E |\varepsilon_{2j}^{(\J)}|^{2q} &\le \frac{(C C_0 s_0)^q q^{3q}}{(nv)^{q}} +\frac{C^q q^{4q}}{n^{2q(1-2\alpha)}} A_{1,2q} .
\end{align}
Applying Lemma~\ref{lemma off diagonal entries assumption} we obtain that there exist some positive constants $A_0$ and $A_1$ depending on $C_0, s_0$ in such way that
\begin{equation}\label{eq: inequality fo A 1 2p}
A_{1,2q} \le C_0^q.
\end{equation}
Combining~\eqref{eq: inequality fo ve j2 p} and~\eqref{eq: inequality fo A 1 2p} we obtain
\begin{align*}
\E |\varepsilon_{2j}^{(\J)}|^{2q} &\le \frac{(C C_0 s_0)^{q} q^{4q}}{(nv)^{2q(1-2\alpha)}}.
\end{align*}
\end{proof}

\begin{lemma}\label{lemma varepsilon_3 assumption}
Let $\tilde v > 0$ be an arbitrary number and $A_0, A_1$ and $C_0$ be some positive constants. Assume that the conditions $\CondTwo$ hold and
\begin{equation}\label{main condition 2}
\max_{\J: |\J| \le K + 2} \max_{l \in \T_\J}\E |\RR_{ll}^{(\J)}(v')|^q \le C_0^q
\end{equation}
for all $v' \geq \tilde v, |u| \le u_0$ and  $1 \le q \le A_1(nv')^{\frac{1-2\alpha}{2}}$. Then there exists $s_0 \geq 1$ depending on $\alpha$ such that for all $v \geq \tilde v/s_0, |u| \le u_0$ and $1 \le q \le A_1(nv)^{\frac{1-2\alpha}{2}}$ we have
$$
\max_{\J: |\J| \le K} \max_{j \in \T_\J} \E |\varepsilon_{3j}^{(\J)}|^{2q} \le \frac{ (s_0 C C_0)^{2q} q^{2q} }{n^{2q(1-2\alpha)}}.
$$
\end{lemma}
\begin{proof}
Let us take an arbitrary $v \geq \tilde v/s_0$, $\J \subset \T$ such that $|\J| \le K$ and $j \in \T_\J$. Applying Lemma~\ref{appendix lemma varepsilon_3} we get
$$
\E |\varepsilon_{3j}^{(\J)}|^{2q} \le C^q\left (\frac{q^{q}}{n^{2q}}\E\left(\sum_{k \in \T_{\J}} |\RR_{kk}^{(\J,j)}|^2 \right)^{q}
+  \frac{q^{2q}}{n^{2q(1-2\alpha)}}  \frac{1}{n} \sum_{k \in \T_{\J}} \E|\RR_{kk}^{(\J,j)}|^{2q}       \right ).
$$
From Lemma~\eqref{appendix lemma resolvent relations on different v} we get for $s_0 \geq 1$
$$
|\RR_{kk}^{(\J,j)}(v)|^{2q} \le s_0^{2q}|\RR_{kk}^{(\J,j)}(s_0v)|^{2q}.
$$
As in the previous lemmas we may take $s_0 := 2^{\frac{4}{1-2\alpha}}$ and set $v':=s_0 v \geq \tilde v$ (note that in this lemma it actually suffices to take $s_0 := 2^{\frac{2}{1-2\alpha}}$, but for simplicity we shall use the same value for $s_0$) and get that $2q \le A_1 (nv')^{\frac{1-2\alpha}{2}}$. We may apply~\eqref{main condition 2} to obtain the bound
\begin{equation*}
\E|\RR_{kk}^{(\J,j)}(v)|^{2q} \le s_0^{2q} C_0^{2q}.
\end{equation*}
It is easy to see that
$$
\E\left(\frac{1}{n}\sum_{k \in \T_{\J}} |\RR_{kk}^{(\J,j)}|^2 \right)^{q} \le \frac{1}{n} \sum_{k \in \T_{\J}} |\RR_{kk}^{(\J,j)}|^{2q}.
$$
The last two inequalities together imply that
$$
\E |\varepsilon_{3j}^{(\J)}|^{2q} \le \frac{ (s_0 C C_0)^{2q} q^{2q} }{n^{2q(1-2\alpha)}}.
$$
\end{proof}

\begin{lemma}\label{lemma step for resolvent}
Let $\tilde v > 0$ be an arbitrary number and $C_0$ be some large positive constant. Suppose that the conditions $\CondTwo$ hold. There exist positive constant $s_0$ depending on $\alpha$ and positive constants $A_0, A_1$ depending on $C_0, s_0$ such that assuming
\begin{equation}\label{main condition 3}
\max_{\J: |\J| \le K + 2} \max_{l,k \in \T_\J}\E |\RR_{lk}^{(\J)}(v')|^q \le C_0^q
\end{equation}
for all $v' \geq \tilde v, |u| \le u_0$ and $1 \le q \le A_1(nv')^{\frac{1-2\alpha}{2}}$ we have
$$
\max_{\J: |\J| \le K} \max_{l \in \T_\J} \E|\RR_{ll}^{(\J)}(v)|^{q} \le C_0^{q}
$$
for all $v \geq \tilde v/s_0, |u| \le u_0$ and $1 \le q \le A_1(nv)^{\frac{1-2\alpha}{2}}$.
\end{lemma}
\begin{proof}
Let us take an arbitrary $v \geq \tilde v/s_0$, $\J \subset \T$ such that $|\J| \le K$ and $j \in \T_\J$. From Lemma~\ref{lemma: recurence relation for R_jj}, inequality~\eqref{inequality for R_jj}, it follows that
$$
\E|\RR_{jj}^{(\J)}(v)|^{q} \le C_1^{q}(1 + \E^{\frac{1}{2}}|T_n^{(\J)}|^{q}\E^{\frac{1}{2}}|\RR_{jj}^{(\J)}(v)|^{2q} + \E^{\frac{1}{2}}|\varepsilon_j^{(\J)}|^{2q} \E^{\frac{1}{2}}|\RR_{jj}^{(\J)}(v)|^{2q}).
$$
Let us choose again $s_0 := 2^{\frac{4}{1-2\alpha}}$. As shown in proof of the previous lemmas choosing $v':=s_0 v \geq v_1$ ensures $2q \le A_1 (nv')^{\frac{1-2\alpha}{2}}$. We now apply Lemma~\eqref{appendix lemma resolvent relations on different v} and the assumption  ~\eqref{main condition 3} to estimate
\begin{equation}\label{eq: bound for RR_jj^2p}
\E|\RR_{jj}^{(\J)}(v)|^{2q} \le s_0^{2q}\E|\RR_{jj}^{(\J)}(s_0v)|^{2q} \le (C_0 s_0)^{2q}.
\end{equation}
The Cauchy-Schwartz inequality  implies that
$$
\E|T_n^{(\J)}|^{q} \le \left(\frac{1}{n} \sum_{j \in \T_{\J}} \E |\varepsilon_j^{(\J)}|^{2q} \right)^{1/2} \left(\frac{1}{n} \sum_{j \in \T_{\J}} \E |\RR_{jj}^{(\J)}(v)|^{2q} \right)^{1/2}
$$
From~\eqref{eq: bound for RR_jj^2p} it follows that
$$
\left(\frac{1}{n} \sum_{j \in \T_{\J}} \E |\RR_{jj}^{(\J)}(v)|^{2q} \right)^{1/2} \le s_0^q C_0^q.
$$
By an obvious inequality we get
$$
\E |\varepsilon_j^{(\J)}|^{2q} \le 4^{2q} (\E |\varepsilon_{1j}^{(\J)}|^{2q} + \E |\varepsilon_{2j}^{(\J)}|^{2q} + \E |\varepsilon_{3j}^{(\J)}|^{2q} + \E |\varepsilon_{4j}^{(\J)}|^{2q}).
$$
Now we may use Lemmas~\ref{lemma varepsilon_2 assumption},~\ref{lemma varepsilon_3 assumption},~\ref{appendix lemma varepsilon_1} and~\ref{appendix lemma varepsilon_4} and obtain the following bound
\begin{align*}
\E |\varepsilon_j^{(\J)}|^{2q} &\le 4^{2q} \left[ \frac{\mu_4 D^{2q-4}}{n^{q(1-2\alpha)+4\alpha}}  +  \frac{(C C_0 s_0)^{q} q^{4q}}{(nv)^{2q(1-2\alpha)}} +  \frac{ (s_0 C C_0)^{2q} q^{2q} }{n^{2q(1-2\alpha)}} + \frac{1}{(nv)^{2q}}\right].
\end{align*}
It is easy to see that  since the estimates for $T_n^{(\J)}$ and $\RR_{jj}^{(\J)}$ depend on $\E |\varepsilon_j^{(\J)}|^{2q}$ we may choose the constants $A_0$ and $A_1$ (correcting the previous choice if needed) depending on $C_0, s_0$ in such way that
\begin{align*}
\E|\RR_{jj}^{(\J)}(v)|^{q} \le C_0^q
\end{align*}
for all $1 \le q \le A_1 (nv)^{\frac{1-2\alpha}{2}}$ and $z \in \mathbb D$.
\end{proof}

\begin{proof}[Proof of Lemma~\ref{main lemma}]
We first prove~\eqref{eq: main lemma first statement}. Let us take $v = 1$ and some large constant $C_0 \gg \max(1, C_1)$. We also take $s_0, A_0$ and $A_1$ as chosen in the previous  Lemma~\ref{lemma step for resolvent}. Set $L = [-\log_{s_0} v_0]+1$. Since $\|\RR^{(\J)}(v)\| \le v^{-1} = 1$ we may write
$$
\max_{\J: |\J| \le 2L} \max_{l,k  \in \T_\J}|\RR_{lk}^{(\J)}(v)| \le C_0
$$
and
$$
\max_{\J: |\J| \le  2L} \max_{l,k \in \T_\J} \E|\RR_{lk}^{(\J)}(v)|^q \le C_0^q
$$
for $1 \le q \le A_1 n^{\frac{1-2\alpha}{2}}$. Taking $K: = 2(L-1)$ and applying Lemma~\ref{lemma step for resolvent} we get
$$
\max_{\J: |\J| \le 2(L-1)} \max_{l,k \in \T_\J} \E|\RR_{lk}^{(\J)}(v)|^q \le C_0^q
$$
for $1 \le q \le A_1 n^\frac{1-2\alpha}{2} s_0^{-\frac{1-2\alpha}{2}}$, $v \geq 1/s_0$. We may repeat this procedure $L$ times and finally obtain
$$
\max_{l,k \in \T}\E|\RR_{lk}(v)|^q \le C_0^q
$$
for $1 \le q \le A_1 n^\frac{1-2\alpha}{2} s_0^{-L\frac{1-2\alpha}{2}} \le A_1(n v_0)^{\frac{1-2\alpha}{2}}$ and $v \geq 1/s_0^{L} \geq v_0$.

To prove~\eqref{eq: main lemma second statement} it is enough to repeat all the previous steps and apply the inequality~\eqref{appendix eq resolvent relations on different v 1} from Lemma~\ref{appendix lemma imag resolvent relations on different v}.  We omit the details.
\end{proof}

\section{Imaginary part of diagonal entries of the resolvent}\label{sec: imagionary part}
In this section we estimate the moments of the imaginary part of diagonal entries of the  resolvent.  Let us introduce the following quantity
\begin{equation}\label{eq: Psi definition}
\Psi(z): = \imag s(z) + \frac{p^2}{nv}.
\end{equation}
To simplify notations we will often write $\Psi(v)$ and $\Psi$ instead of $\Psi(z)$. The main result of this section is the following lemma.
\begin{lemma}\label{lemma: imag part of R_jj}
Assuming conditions $\CondTwo$ there exist positive constants $H_0$ depending on $u_0, V$ and positive constants $A_0, A_1$ depending on $H_0, \alpha$ such that for all $1 \le p \le A_1(nv)^{\frac{1-2\alpha}{2}}$ and $z \in \mathbb D$  we get
$$
\max_{j \in \T }\E \imag^p \RR_{jj}(z) \le H_0^p \Psi^p(z).
$$
\end{lemma}
We remark here that the values of $A_0$ and $A_1$ in this lemma are different from the values of respective quantities in Lemma~\ref{main lemma}, but for simplicity we shall use the same notations. Applying both Lemmas we shall restrict the upper limit of the moment $q$ to the minimum of the two $A_1$'s and  the lower end of the range of $v$ to the maximum of the two $A_0$'s via $v \ge A_0n^{-1}$. Throughout this section we shall assume that for all $1\le p \le A_1(nv)^{\frac{1-2\alpha}{2}}$ and $z \in \mathbb D$ we have
\begin{equation}\label{eq: main assumption about diagonal entries of resolvent for imag R_jj}
\E |\RR_{jj}(u+iv)|^p \le C_0^p,
\end{equation}
where the value of $C_0$ is defined in Lemma~\ref{main lemma}.

The following lemma is the analogue of Lemma~\ref{lemma: recurence relation for R_jj} and provides a recurrence relation for $\imag \RR_{jj}$.
\begin{lemma}\label{lemma recurence relation for imag R_jj}
For any set $\J$ and $j \in \T_\J$ there exists a positive constant $C_1$ depending on $u_0, V$ such that for all $z = u + iv$ with $V \geq v > 0$ and $|u| \le u_0$ we have
\begin{align*}
\imag \RR_{jj}^{(\J)}(z) &\le C_1 \left[\imag s(z) (1 + (|\varepsilon_j^{(\J)}| + |T_n^{(\J)}|^{\frac{1}{2}})|\RR_{jj}^{(\J)}(z)|)   +  |\imag \varepsilon_j^{(\J)} + \imag \Lambda_n^{(\J)}| |\RR_{jj}^{(\J)}(z)|\right.\\
&\left.\qquad\qquad + (|\varepsilon_j^{(\J)}| + |T_n^{(\J)}|^{\frac{1}{2}})\imag \RR_{jj}^{(\J)}(z) \right ].
\end{align*}
\end{lemma}
\begin{proof}
The proof is similar to the proof of Lemma~\ref{lemma: recurence relation for R_jj} is omitted.
\end{proof}

\begin{lemma}\label{lem: imag RR jj recur}
Let $\tilde v > 0$ be an arbitrary number and $H_0$ be some large positive constant. Assume that the conditions $\CondTwo$ hold. There exist a positive  constant $s_0$ depending on $\alpha$ and positive constants $A_0, A_1$ depending on $H_0, s_0$ such that assuming
\begin{equation}\label{main condition imag part of R_jj}
\max_{\J: |\J| \le K + M} \max_{l \in \T_\J} \E \imag^q \RR_{ll}^{(\J)}(v') \le H_0^q \Psi^q(v')
\end{equation}
for all $\quad v' \geq \tilde v, |u|\le u_0$ and $1 \le q \le A_1(nv')^{\frac{1-2\alpha}{2}}$ we have
$$
\max_{\J: |\J| \le K} \max_{l \in \T_\J} \E \imag^q \RR_{ll}^{(\J)}(v) \le H_0^q \Psi^q(v)
$$
for all $v \geq \tilde v/s_0, |u| \le u_0$ and $1 \le q \le A_1(nv)^{\frac{1-2\alpha}{2}}$.
\end{lemma}
\begin{proof}
From Lemma~\ref{lemma recurence relation for imag R_jj} it follows that
\begin{align}\label{eq: imag R step 1}
\E\imag^q \RR_{jj}^{(\J)} &\le (C C_0)^q \imag^q s(z) \E^\frac12 (1 + (|\varepsilon_j^{(\J)}| + |T_n^{(\J)}|^{\frac{1}{2}})^{2q}  \nonumber \\
&\qquad + (C C_0)^q \E^{\frac{1}{2}}|\imag \varepsilon_j^{(\J)} + \imag \Lambda_n^{(\J)}|^{2q} \nonumber \\
&\qquad + C^q \E^{\frac{1}{2}}(|\varepsilon_j^{(\J)}| + |T_n^{(\J)}|^{\frac{1}{2}})^{2q} \E^{\frac{1}{2}} \imag^{2q} \RR_{jj}^{(\J)}.
\end{align}
To estimate $\E |\varepsilon_j^{(\J)}|^{2q}$ and $\E|T_n^{(\J)}|^{q}$ we may proceed as in Lemma~\ref{lemma step for resolvent}. We obtain the following inequalities
\begin{align}\label{eq: imag R step 2}
\E |\varepsilon_j^{(\J)}|^{2q} &\le 4^{2q} \left[ \frac{\mu_4 D^{2q-4}}{n^{q(1-2\alpha)+4\alpha}}  +  \frac{C^{q} q^{4q}}{(nv)^{2q(1-2\alpha)}} +  \frac{ C^{2q} q^{2q} }{n^{2q(1-2\alpha)}} + \frac{1}{(nv)^{2q}}\right].
\end{align}
and
\begin{align}\label{eq: imag R step 3}
\E|T_n^{(\J)}|^{q} \le C_0^q \left(\frac{1}{n} \sum_{j \in \T_{\J}} \E |\varepsilon_j^{(\J)}|^{2q} \right)^{1/2}.
\end{align}
Take as before $s_0 := 2^{\frac{4}{1-2\alpha}}$. Choosing $v':=s_0 v \geq v_1$ we may show that $2q \le A_1 (nv')^{\frac{1-2\alpha}{2}}$. Applying Lemma~\ref{appendix lemma imag resolvent relations on different v} and using the assumption~\eqref{main condition imag part of R_jj} we get
\begin{equation}\label{eq: bound for imag part of RR jj}
\E \imag^{2q} \RR_{jj}^{(\J)}( v) \le s_0^{2q} \E \imag^{2q} \RR_{jj}^{(\J)}(s_0 v) \le s_0^{2q} H_0^{2q} \Psi^{2q}(s_0 v).
\end{equation}
Since we need an estimate involving $\Psi^{2q}(v)$ instead of $\Psi^{2q}(s_0 v)$ on the r.h.s. of the previous inequality we need to perform a descent along the imaginary line from $s_0 v$ to $v$. To this purpose we again apply Lemma~\ref{appendix lemma imag resolvent relations on different v}.  Choosing suitable constants $A_0$ and $A_1$ in~\eqref{eq: imag R step 2} and~\eqref{eq: imag R step 3} one may show that
\begin{align}\label{eq: imag R step 1 0}
\E\imag^q \RR_{jj}^{(\J)} &\le (C C_0)^q \imag^q s(z)  + (C C_0)^q \E^{\frac{1}{2}}|\imag \varepsilon_j^{(\J)} + \imag \Lambda_n^{(\J)}|^{2q} + \frac{H_0^q}{3} \Psi^q.
\end{align}
Applying Lemmas~\ref{appendix lemma imag epsilon 2}--\ref{appendix lemma imag epsilon 3} we obtain
$$
\E |\imag \varepsilon_j^{(\J)}|^{2q} \le 3^{2q} \left[ C^q  \left( \frac{q^{2q} \E \imag^{q} m_n^{(\J,j)}(z)  }{(nv)^{q}}  + \frac{q^{4q}}{n^{2q(1-2\alpha)}} F_{1,2q} \right) + \frac{C^q q^{2q}}{n^{2q(1-2\alpha)}} F_{1,2q} + \frac{1}{(nv)^{2q}} \right].
$$
We may write
\begin{align}\label{eq: imag R step 4}
(C C_0)^q \E |\imag \varepsilon_j^{(\J)}|^{2q} &\le C^q \left[  \frac{q^{2q} s_0^{2q} H_0^q \Psi^q  }{(nv)^{q}}  + \frac{q^{4q} s_0^{4q} H_0^{2q} \Psi^{2q} }{n^{2q(1-2\alpha)}} + \frac{1}{(nv)^{{2q}}} \right]  \nonumber \\
& \le \frac{H_0^{2q} }{C_2} \Psi^{2q}  + \frac{ (C s_0)^{4q}q^{4q} }{(nv)^{2q}} +  \frac{q^{4q} s_0^{4q} H_0^{2q} \Psi^{2q} }{n^{2q(1-2\alpha)}},
\end{align}
where $C_2$ will be chosen later.
To estimate $\E |\imag \Lambda_n^{(\J)}|^{q}$ we may proceed as in the proof of Theorem~\ref{th:main}. We will apply Theorem~\ref{th: general bound} (one has to replace in the definition of~\eqref{definition of A} the maximum over $|\J| \le 1$ by the maximum over $|\J| \le L$) and assumption~\eqref{main condition imag part of R_jj}. Hence,
\begin{align}\label{eq: imag R step 5}
(C C_0)^q \E |\imag \Lambda_n^{(\J)}|^{2q} \le \frac{H_0^{2q}}{C_2} \Psi^{2q} + \frac{(C s_0)^{4q} q^{4q}}{(nv)^{2q}}.
\end{align}
Combining the estimates~\eqref{eq: imag R step 4} and \eqref{eq: imag R step 5} we may choose constants $C_2, A_0$ and $A_1$ (correcting the previous choice if needed) such that
$$
(C C_0)^q \E^{\frac{1}{2}}|\imag \varepsilon_j^{(\J)} + \imag \Lambda_n^{(\J)}|^{2q} \le \frac{H_0^q}{3} \Psi^q, \quad \text{ and } \quad
(C C_0)^q \imag^q s(z) \le \frac{H_0^q}{3} \Psi^q.
$$
The last two inequalities and~\eqref{eq: imag R step 1 0} together imply the desired bound
$$
\E \imag^q \RR_{jj}^{(\J)} \le H_0^q \Psi^q.
$$
\end{proof}
\begin{proof}[Proof of Lemma~\ref{lemma: imag part of R_jj}]
Let us take any $u_0>0$ and any $\hat v \ge 2+u_0, |u|\le u_0$. We also fix arbitrary $\J\subset \T$. We claim that
\begin{equation}\label{bound of imag RR via imag s}
\imag s(u+i\hat v)\ge \frac12\imag \RR_{jj}^{(\J)}(u+i\hat v).
\end{equation}
Indeed, we first mention that for all $u$ (and $|u|\le u_0$ as well)
\begin{equation}\label{bound for imag RR_{jj} for big V}
\imag \RR_{jj}^{(\J)}(u+i \hat v)\le \frac{1}{\hat v}.
\end{equation}
For all $|u|\le u_0$ and $|x|\le 2$ we obtain
\begin{equation}
\frac{\hat v}{(x-u)^2+\hat v^2}\ge\frac{\hat v}{(2+u_0)^2+\hat v^2} \geq \frac{1}{2\hat v}.
\end{equation}
It follows from the last inequality that
\begin{equation}\label{bound for imag s for big V}
\imag s(u+i \hat v)=\frac1{2\pi}\int_{-2}^2\frac{\hat v}{(u-x)^2+\hat v^2}\sqrt{4-x^2}dx\ge \frac{1}{2\hat v}.
\end{equation}
Comparing~\eqref{bound for imag RR_{jj} for big V} and~\eqref{bound for imag s for big V} we arrive at~\eqref{bound of imag RR via imag s}.

We not take $v \geq \hat v$. Let $H_0$ be some large constant, $H_0 \geq \max(C', C'')$. We choose $s_0, A_0$ and $A_1$  as in the previous Lemma~\ref{lem: imag RR jj recur} obtaining
$$
\max_{\J: |\J| \le  2 L} \max_{j \in \T_\J} \imag^q \RR_{jj}^{(\J)}(z) \le H_0^q  \Psi^q(z)
$$
with $L = [-\log_{s_0} v_0] + 1$. We may now proceed recursively  in $L$ steps
and arrive at
$$
\max_{j \in \T} \imag^q \RR_{jj}(z) \le H_0^q  \Psi^q(z)
$$
for $v \geq v_0$ and  $1 \le q \le A_1(nv)^{\frac{1-2\alpha}{2}}$ .
\end{proof}
\appendix

\section{Moment inequalities for linear and quadratic forms of random variables}
In this section we present some inequalities for linear and quadratic forms.

\begin{theorem}[Rosenthal type inequality]\label{th: Rosenthal}
Let $X_j, j = 1, ..., n$, be independent random variables with $\E X_j = 0, \E X_j^2 = \sigma^2$ and $\mu_p: = \max_j \E|X_j|^p < \infty$.
Then there exists some absolute constant $C$ such that for $p \geq 2$
$$
\E\big|\sum_{k=1}^n a_{k } X_k\big|^p \le (C p)^{p/2} \sigma^p \left (\sum_{k=1}^n a_{k}^2 \right )^{\frac{p}{2}} +  \mu_p (C p)^p \sum_{k=1}^n |a_k|^p.
$$
\end{theorem}
\begin{proof}
For a proof see~\cite{Rosenthal1970}[Theorem~3] and~\cite{JohnSchecttmanZinn1985}[Inequality~(A)].
\end{proof}
To estimate the moments of quadratic forms in our proof we shall  use the following inequality due to Gin{\'e}, Lata{\l}a, Zinn~\cite{GineLatalaZinn2000}.  Let us denote
$$
Q: = \sum_{1 \le j \neq k \le n} a_{jk} X_j X_k.
$$
\begin{theorem}\label{appendix lemma for quadratic forms}
Let $X_j, j = 1, ..., n$, be independent random variables with $\E X_j = 0, \E X_j^2 = \sigma^2$ and $\mu_p: = \max_j \E|X_j|^p < \infty$. Then for $p \geq 2$
$$
\E |Q|^p \le C^{p} \left[p^p\left(\sum_{j=1}^n \sum_{k \in \T_j} |a_{jk}|^2\right)^{\frac{p}{2}} + \mu_p p^{\frac{3p}{2}} \sum_{j=1}^n \left(\sum_{k \in \T_j} |a_{jk}|^2 \right )^{\frac{p}{2}} +
\mu_p^2 p^{2p} \sum_{j=1}^n \sum_{k \in \T_j} |a_{jk}|^p \right].
$$
\end{theorem}
\begin{proof}
See~\cite{GineLatalaZinn2000}[Proposition~2.4] or~\cite{GotTikh2003}[Lemma A.1].
\end{proof}
We remark here that the sequence of papers  ~\cite{Schlein2014}, \cite{ErdosSchleinYau2009},~\cite{ErdosSchleinYau2009b},~\cite{ErdosSchleinYau2010} is relying on the Hanson--Wright large deviation inequality~\cite{HansonWright}
for quadratic forms instead of this estimate.

The following result is trivial but we formulate it as a lemma since we shall use it many times during the proof of the main result.
\begin{lemma}\label{appendix lemma varepsilon_1}
Assuming conditions $\CondTwo$ for $p \geq 1$ we have
$$
\E|\varepsilon_{1j}|^{2p} \le \frac{\mu_4 D^{2p-4}}{n^{p(1-2\alpha)+4\alpha}}.
$$
\end{lemma}
\begin{proof}
The proof follows directly from the definition of $\varepsilon_{1j}: = \frac{1}{\sqrt n} X_{jj}$.
\end{proof}
The rest of this section is devoted to the proof of moment inequalities for linear and quadratic forms based on the entries of the resolvent $\RR^{(\J)}$ or some functions of it.
\begin{lemma}\label{appendix lemma sum of varepsilon_1}
Under conditions $\CondTwo$ for $p \geq 1$ and $q = 1,2$ we have
\begin{equation}\label{appendix eq sum of varepsilon_1 q = 1,2}
\E \left| \frac{1}{n} \sum_{j=1}^n \varepsilon_{1j}^q \right|^p \le  \frac{(C p)^p}{n^p}.
\end{equation}
\end{lemma}
\begin{proof}
From the definition $\varepsilon_{1j}: = \frac{1}{\sqrt n} X_{jj}$. Let us introduce the following notations for moments $\beta_q: = \E X^q$. In these notations
$\beta_1 = 0$ and  $\beta_2 = 1$.
It is easy to see that
\begin{equation}\label{appendix eq varepsilon_1 1}
\E \left| \frac{1}{n n^{\frac{q}{2}}} \sum_{j=1}^n X_{jj}^q \right|^{p} \le 2^{p}\E \left| \frac{1}{n n^{\frac{q}{2}}} \sum_{j=1}^n [X_{jj}^q - \beta_q] \right|^{p} + \frac{(2\beta_q)^p}{ n^{\frac{qp}{2}}} \one[q = 2].
\end{equation}
Applying Rosenthal's inequality~\ref{th: Rosenthal}  we get
$$
\E \left| \sum_{j=1}^n [X_{jj}^q - \beta_q] \right|^{p} \le (Cp)^{\frac{p}{2}} (n\E |X_{jj}|^{2q})^{\frac{p}{2}} + \E |X_{jj}|^{qp} (Cp)^{p} n.
$$
Since $\max_{1 \le j \le n} \E|X_{jj}|^4 \le \mu_4 < \infty$ and $|X_{jj}| \le D n^{\alpha}$ it follows that
$$
\mu_{qp} \le \mu_4 D n^{\alpha (pq - 4)}.
$$
and we get
\begin{equation}\label{appendix eq varepsilon_1 2}
\E \left| \sum_{j=1}^n [X_{jj}^q - \beta_q] \right|^{p} \le (Cp)^{p} n^{q p} .
\end{equation}
Inequalities ~\eqref{appendix eq varepsilon_1 1} and~\eqref{appendix eq varepsilon_1 2} conclude the proof of the  lemma.
\end{proof}

Recall the definition of the following quantities given in ~\eqref{eq: definition of A and F}
\begin{align*}
&A_{\nu,q}:= A_{\nu,q}^{(K)}:= \max_{\J: |\J| \le \nu + K } \max_{l \neq k \in \T_{\J}} \E|\RR_{lk}^{(\J)}|^q,
&F_{\nu,q}: = F_{\nu,q}^{(K)}: =   \max_{\J: |\J| \le \nu + K } \max_{l \in \T_{\J}} \E \imag^q \RR_{ll}^{(\J)},
\end{align*}
where $K > 0$ is some integer.

\begin{lemma}\label{appendix lemma varepsilon_2}
Assuming conditions $\CondTwo$ for $p \geq 2$ and $|\J| \le K$ we have
\begin{align*}
\E |\varepsilon_{2j}^{(\J)}|^p &\le C^{p}   \left( \frac{p^p}{(nv)^{\frac{p}{2}}}   \E[\imag m_n^{(\J,j)}(z)]^{\frac{p}{2}} + \frac{p^{\frac{3p}{2}}}{(nv)^{\frac{p}{2}}} F_{1,\frac{p}{2}}+\frac{p^{2p}}{n^{p(1-2\alpha)}} A_{1,p} \right).
\end{align*}
\end{lemma}
\begin{proof}
By definition
$$
\varepsilon_{2j}^{(\J)}: = -\frac{1}{n}\sum_{k\neq l \in \T_{\J,j}} X_{jk} X_{jl} \RR_{lk}^{(\J,j)}.
$$
Applying Lemma~\ref{appendix lemma for quadratic forms} we get
\begin{align}\label{appendix: vareps_2 eq 1}
\E |\varepsilon_{2j}^{(\J)}|^p &\le \frac{C^{p}}{n^p} \left[p^p \E\left(\sum_{k \in \T_{\J, j}} \sum_{l \in \T_{\J,j,k}} |\RR_{lk}^{(\J,j)}|^2\right)^{\frac{p}{2}} +
\mu_p p^{\frac{3p}{2}} \sum_{k \in \T_{\J, j}} \E\left(\sum_{l \in \T_{\J,j,k}} |\RR_{lk}^{(\J,j)}|^2 \right )^{\frac{p}{2}} \right. \nonumber\\
&\qquad\qquad\left.+ \mu_p^2 p^{2p}\sum_{k \neq l \in \T_{\J,j}} \E |\RR_{lk}^{(\J,j)}|^p \right].
\end{align}
From Lemma~\ref{appendix lemma resolvent inequalities 1} [Inequality~\eqref{appendix lemma resolvent inequality 1}] we get
\begin{align}\label{appendix: vareps_2 eq 2}
\left(\sum_{k \in \T_{\J, j}} \sum_{l \in \T_{\J,j,k}} |\RR_{lk}^{(\J,j)}|^2\right)^{\frac{p}{2}} \le \left (\frac{n}{v} \imag m_n^{(\J,j)}(z)\right)^{\frac{p}{2}} .
\end{align}
Since $|X_{jk}| \le D n^{\alpha}$ we get
\begin{align}\label{appendix: vareps_2 eq 3}
\mu_p \le \mu_4 D^{p-4} n^{\alpha(p-4)}.
\end{align}
From Lemma~\ref{appendix lemma resolvent inequalities 1} [Inequality~\eqref{appendix lemma resolvent inequality 2}] we may conclude that
\begin{align}\label{appendix: vareps_2 eq 4}
\sum_{k \in \T_{\J, j}} \left(\sum_{l \in \T_{\J,j,k}} |\RR_{lk}^{(\J,j)}|^2 \right )^{\frac{p}{2}} \le \sum_{k \in \T_{\J, j}} \left( \frac{1}{v}\imag \RR_{kk}^{(\J,j)}  \right )^{\frac{p}{2}}.
\end{align}
Substituting~\eqref{appendix: vareps_2 eq 2}--\eqref{appendix: vareps_2 eq 4} into~\eqref{appendix: vareps_2 eq 1} concludes the proof of the  lemma.
\end{proof}

For $p = 2$ and $4$ we may give a better bound for the quadratic form $\varepsilon_{2j}$.
\begin{lemma}\label{appendix lemma varepsilon_2 4 moment}
Let $\mathfrak M^{(j)}: = \sigma\{X_{lk}, l, k \in \T_j\}$. Assuming conditions $\Cond$ for $q = 2$ and $4$ we have
$$
\E(|\varepsilon_{2j}|^q \big|\mathfrak M^{(j)}) \le \frac{C}{(nv)^{\frac{q}{2}}} \imag^{\frac{q}{2}} m_n^{(j)}(z).
$$
\end{lemma}
\begin{proof}
Recall that
$$
\varepsilon_{2j}: = -\frac{1}{n}\sum_{k\neq l \in \T_{j}} X_{jk} X_{jl} \RR_{lk}^{(j)}.
$$
For $q = 2$ the proof follows immediately from Lemma~\ref{appendix lemma resolvent inequalities 1} [Inequality~\eqref{appendix lemma resolvent inequality 1}] and
\begin{equation}\label{appendix lemma varepsilon_2 4 moment eq 1}
\sum_{k \in \T_{j}} \sum_{l \in \T_{j,k}} |\RR_{lk}^{(j)}|^2 \le \frac{n}{v} \imag m_n^{(j)}(z).
\end{equation}
For $q = 4$ we apply Lemma~\ref{appendix lemma for quadratic forms} and get
\begin{align*}
\E(|\varepsilon_{2j}|^q|\mathfrak M^{(j)}) \le \frac{C \mu_4^2}{n^q} \left(\sum_{k \in \T_{j}} \sum_{l \in \T_{j,k}} |\RR_{lk}^{(j)}|^2\right)^{2}.
\end{align*}
and use~\eqref{appendix lemma varepsilon_2 4 moment eq 1}.
\end{proof}

\begin{lemma}\label{appendix lemma varepsilon_3}
Assuming conditions $\Cond$ for $p \geq 2$ we have
$$
\E |\varepsilon_{3j}^{(\J)}|^p \le C^p\left (\frac{p^{\frac{p}{2}}}{n^p}\E\left(\sum_{k \in \T_{\J}} |\RR_{kk}^{(\J,j)}|^2 \right)^{\frac{p}{2}}
+  \frac{p^p}{n^{p(1-2\alpha)}}  \frac{1}{n} \sum_{k \in \T_{\J}} \E|\RR_{kk}^{(\J,j)}|^p       \right ).
$$
\end{lemma}
\begin{proof}
By  definition
$$
\varepsilon_{3j}^{(\J)}: = -\frac{1}{n}\sum_{k \in \T_{\J,j}} \RR_{kk}^{(\J,j)} [X_{jk}^2 - 1].
$$
Applying Rosenthal's inequality~\ref{th: Rosenthal} we get
\begin{align}\label{appendix: vareps_3 eq 1}
\E |\varepsilon_{3j}^{(\J)}|^p \le \frac{(Cp)^p}{n^p} \left( \E\left(\sum_{k \in \T_{\J,j}} |\RR_{kk}^{(\J,j)}|^2 \right)^{\frac{p}{2}}  +  \E|X_{11}|^{2p} \sum_{k \in \T_{\J,j}} \E|\RR_{kk}^{(\J,j)}|^p \right ).
\end{align}
Since $|X_{jk}| \le D n^{\alpha}$ we get
\begin{align}\label{appendix: vareps_3 eq 2}
\mu_{2p} \le \mu_4 D^{2p-4} n^{\alpha(2p-4)}.
\end{align}
Substituting~\eqref{appendix: vareps_3 eq 2} to~\eqref{appendix: vareps_3 eq 1} we finish the proof of the lemma.
\end{proof}

For small $p$ that is $p \le \frac{1}{\alpha}$ we may state a better bound for $\varepsilon_{3j}$.
\begin{lemma}\label{appendix lemma varepsilon_3 small p}
Assuming conditions $\CondTwo$ for $2 \le p \leq \frac{1}{\alpha}$ we have
$$
\E(|\varepsilon_{3j}|^p \big|\mathfrak M^{(j)}) \le \frac{C}{n^{\frac{p}{2} + 1}} \sum_{k \in \T_{j}} \E|\RR_{kk}^{(j)}|^p.
$$
\end{lemma}
\begin{proof}
Applying Rosenthal's inequality~\ref{th: Rosenthal} and $|X_{jk}| \le D n^{\alpha}$ we get
\begin{align*}
\E(|\varepsilon_{3j}|^p \big|\mathfrak M^{(j)})\le \frac{C}{n^\frac{p}{2}} \left( \E\left(\frac{1}{n}\sum_{k \in \T_{j}} |\RR_{kk}^{(j)}|^2 \right)^{\frac{p}{2}}  +   \frac{1}{n} \sum_{k \in \T_{j}} \E|\RR_{kk}^{(j)}|^p \right ) \le \frac{C}{n^{\frac{p}{2} + 1}} \sum_{k \in \T_{j}} \E|\RR_{kk}^{(j)}|^p,
\end{align*}
where the last inequality holds since $2 \le p \le \frac{1}{\alpha}$.
\end{proof}

\begin{lemma}\label{appendix lemma varepsilon_4}
For $p \geq 2$ we have
$$
\E |\varepsilon_{4j}^{(\J)}|^p \le \frac{1}{(nv)^p}.
$$
\end{lemma}
\begin{proof}
Applying the  Schur complement formula, see~\cite{GotzeTikh2014rateofconv}[Lemma~7.23] or~\cite{GotTikh2003}[Lemma~3.3], one may write
$$
\varepsilon_{4j}^{(\J)} = \frac{1}{n} (\Tr \RR^{(\J)} - \Tr \RR^{(\J,j)}) =
\frac{1}{n}\left( 1 + \frac{1}{n} \sum_{l,k \in \T_{\J,j}} X_{jk} X_{jl} [(\RR^{(\J,j)})^2]_{kl}\right) \RR_{jj}^{(\J)} = \frac{(\RR_{jj}^{(\J)})^{-1}}{n} \frac{d \RR_{jj}^{(\J)}}{dz}.
$$
Applying now Lemma~\ref{appendix lemma resolvent inequalities 1} concludes the proof of the lemma.
\end{proof}
Recall the definition of the quantities $\eta_{\nu j}, \nu = 0, 1, 2$
\begin{align*}
&\eta_{0j}: = \frac{1}{n} \sum_{k \in T_j} [(\RR^{(j)})^2]_{kk} , \qquad \eta_{1j}: = \frac{1}{n} \sum_{k \neq l \in \T_j} X_{jl} X_{jk} [(\RR^{(j)})^2]_{kl},\\
&\eta_{2j}: = \frac{1}{n} \sum_{k \in \T_j} [X_{jk}^2 - 1] [(\RR^{(j)})^2]_{kk}.
\end{align*}

\begin{lemma}\label{appendix lemma eta_1}
Assuming conditions $\Cond$ for $2 \le p \le 4$ we have
$$
\E (|\eta_{1j}|^p \big|\mathfrak M^{(j)})\le  \frac{C}{n^\frac{p}{2}} \left (\frac{1}{n} \Tr |\RR^{(j)}|^4 \right)^\frac{p}{2}.
$$
\end{lemma}
\begin{proof}
Applying Lemma~\ref{appendix lemma for quadratic forms} we get
 \begin{align*}
 \E (|\eta_{1j}|^p \big|\mathfrak M^{(j)}) &\le \frac{C}{n^p} \left[\E\left(\sum_{k \in \T_{j}} \sum_{l \in \T_{j,k}} |[(\RR^{(j)})^2]_{lk}|^2\right)^{\frac{p}{2}} +
 \mu_p  \sum_{k \in \T_{j}} \E\left(\sum_{l \in \T_{j,k}} |[(\RR^{(j)})^2]_{lk}|^2 \right )^{\frac{p}{2}} \right.\\
 &\qquad\qquad\left.+ \mu_p^2 \sum_{k \neq l \in \T_{j}} \E |[(\RR^{(j)})^2]_{lk}|^p \right].
\end{align*}
Since $2 \le p \le 4$ we have that $\mu_p < \infty$ and
\begin{align*}
\E (|\eta_{1j}|^p \big|\mathfrak M^{(j)}) &\le \frac{C}{n^p} \E\left(\sum_{k \in \T_{j}} \sum_{l \in \T_{j,k}} |[(\RR^{(j)})^2]_{lk}|^2\right)^{\frac{p}{2}} \le \frac{C}{n^\frac{p}{2}} \left (\frac{1}{n} \Tr |\RR^{(j)}|^4 \right)^\frac{p}{2}.
\end{align*}
\end{proof}

\begin{lemma}\label{appendix lemma eta_2}
Assuming conditions $\CondTwo$ for $2 \le p \leq \frac{1}{\alpha}$ we have
$$
\E (|\eta_{2j}|^p \big|\mathfrak M^{(j)}) \le \frac{C}{n^{\frac{p}{2} +1}v^p } \sum_{k \in \T_{j}} \E|[(\RR^{(j)})^2]_{kk}|^p.
$$
\end{lemma}
\begin{proof}
Applying Rosenthal's inequality~\ref{th: Rosenthal} we get
$$
\E (|\eta_{2j}|^p \big|\mathfrak M^{(j)}) \le \frac{C}{n^p} \left(  \E\left(\sum_{k \in \T_{j}} |[(\RR^{(j)})^2]_{kk}|^2 \right)^{\frac{p}{2}}  +  \E|X_{11}|^{2p} \sum_{k \in \T_{j}} \E|[(\RR^{(j)})^2]_{kk}|^p \right ).
$$
Applying Lemma~\ref{appendix lemma resolvent square inequalities}  we obtain
$$
\sum_{k \in \T_{j}} \E|[(\RR^{(j)})^2]_{kk}|^p \le \frac{1}{v^p} \sum_{k \in \T_j} \E \imag^p \RR_{kk}^{(j)}.
$$
Repeating  the  arguments in the proof of the previous lemma concludes the proof of the lemma.
\end{proof}

\begin{lemma}\label{appendix lemma eta_3}
For $p \geq 2$ we have
$$
|\eta_{0j}|^p  \le \frac{ \imag^p m_n^{(j)}}{n^p v^{p}}.
$$
\end{lemma}
\begin{proof}
The proof follows directly from Lemma~\ref{appendix lemma resolvent square inequalities}.
\end{proof}

\begin{lemma}\label{appendix lemma imag epsilon 2}
Assuming conditions $\CondTwo$ for all $|\J| \le K$ we have for $p \geq 2$
$$
\E|\imag \varepsilon_{2j}^{(\J)}|^{p}  \le C^p  \left( \frac{p^p \E \imag^{\frac{p}{2}} m_n^{(\J,j)}(z)  }{(nv)^{\frac{p}{2}}}  + \frac{p^{2p}}{n^{p(1-2\alpha)}} F_{\nu,p} \right) .
$$
\end{lemma}
\begin{proof}
Applying Lemma~\ref{appendix lemma for quadratic forms} we get
\begin{align*}\label{appendix: imag vareps_2 eq 1}
\E |\imag \varepsilon_{2j}^{(\J)}|^p &\le \frac{C^{p}}{n^p} \left[p^p\E\left(\sum_{k \in \T_{\J, j}} \sum_{l \in \T_{\J,j,k}} |\imag \RR_{lk}^{(\J,j)}|^2\right)^{\frac{p}{2}} \right. \nonumber \\
&\qquad\left.+
\mu_p p^{\frac{3p}{2}} \sum_{k \in \T_{\J, j}} \E\left(\sum_{l \in \T_{\J,j,k}} |\imag \RR_{lk}^{(\J,j)}|^2 \right )^{\frac{p}{2}}
+ \mu_p^2 p^{2p}\sum_{k \neq l \in \T_{\J,j}} \E |\imag \RR_{lk}^{(\J,j)}|^p \right].
\end{align*}
Since
$$
\imag \RR_{kl}^{(\J,j)} = v [\RR^{(\J,j)} (\RR^{(\J,j)})^{*} ]_{kl}
$$
it follows that
\begin{align*}
&\sum_{k \neq l \in \T_{\J,j}} \E |\imag \RR_{lk}^{(\J,j)}|^p \le n \sum_{k \in   \T_{\J,j}} \E \imag^p \RR_{kk}^{(\J,j)}, \\
&\E\left(\sum_{k \in \T_{\J, j}} \sum_{l \in \T_{\J,j,k}} |\imag \RR_{lk}^{(\J,j)}|^2\right)^{\frac{p}{2}} \le \frac{n^{\frac{p}{2}}}{v^{\frac{p}{2}}}\E \imag^{\frac{p}{2}} m_n^{(\J,j)}(z)
\end{align*}
and
$$
\sum_{k \in \T_{\J, j}}^n \E\left(\sum_{l \in \T_{\J,j,k}} |\imag \RR_{lk}^{(\J,j)}|^2 \right )^{\frac{p}{2}} \le n^{\frac{p}{2}} \sum_{k \in   \T_{\J,j}} \E \imag^p \RR_{kk}^{(\J,j)}.
$$
Since
$$
\mu_{p} \le \mu_4 D^{p-4} n^{\alpha (p - 4)}
$$
we get the statement of the lemma.
\end{proof}

\begin{lemma}\label{appendix lemma imag epsilon 3}
Assuming conditions $\Cond$ we have for $p \geq 2$
$$
\E |\imag \varepsilon_{3j}^{(\J)}|^p \le \frac{C^p p^p}{n^{p(1-2\alpha)}} F_{1,p} .
$$
\end{lemma}
\begin{proof}
Recall that
$$
\imag \varepsilon_{3j}^{(\J)} = \frac{1}{n} \sum_{l \in \T_{\J,j}} (X_{jl}^2 - 1) \imag \RR_{ll}^{(\J,j)}.
$$
Applying Rosenthal's inequality we obtain
$$
\E |\imag \varepsilon_{3j}^{(\J)}|^p \le \frac{C^p}{n^p} \left(p^{\frac{p}{2}} \E\left( \sum_{l \in \T_{\J,j}} \imag^2 \RR_{ll}^{(\J,j)} \right)^{\frac{p}{2}}  + p^p \mu_{2p} \sum_{l \in \T_{\J,j}} \E \imag^p \RR_{ll}^{(\J,j)}\right).
$$
Since
$$
\mu_{2p} \le \mu_4 D^{2p-4} n^{\alpha (2p - 4)}
$$
we get that
$$
\E |\imag \varepsilon_{3j}^{(\J)}|^p \le \frac{C^p}{n^p} \left(p^{\frac{p}{2}} \E \left( \sum_{l \in \T_{\J,j}} \imag^2 \RR_{ll}^{(\J,j)} \right)^{\frac{p}{2}}  + p^p  n^{\alpha (2p - 4)} \sum_{l \in \T_{\J, j}} \E \imag^p \RR_{ll}^{(\J,j)}\right).
$$
Thus we arrive at the following bound
\begin{align*}
\E |\imag \varepsilon_{3j}^{(\J)}|^p &\le C^p F_{\nu,p} \left(\frac{p^{\frac{p}{2}}  }{n^{\frac{p}{2}}}  + \frac{p^p}{n^{p(1-2\alpha)}}  \right ).
\end{align*}
\end{proof}

\section{Auxiliary lemmas I}
Recall the notations $\Lambda_n: = \Lambda_n(z): = m_n(z) - s(z), \Lambda_n^{(j)}: = m_n^{(j)}(z) - s(z)$  and
$$
T_n: = \frac{1}{n} \sum_{j=1}^n \varepsilon_j \RR_{jj},
$$
where $\varepsilon_j = \varepsilon_{1j} + \varepsilon_{2j} + \varepsilon_{3j} + \varepsilon_{4j}$ and $\varepsilon_{\alpha j}, \alpha = 1, 2, 3, 4$ are defined
in~\eqref{eq: R_jj representation}. Recall the identity
\begin{equation}\label{appendix: lambda and T}
T_n(z) =  (z + m_n(z) + s(z)) \Lambda_n(z),
\end{equation}
as well as the notations
$$
b(z) = z + 2 s(z) \text{ and } b_n(z) = b(z) + \Lambda_n(z).
$$
The following lemma plays a crucial role in the proof of Theorem~\ref{th:main}. It has been proved in ~\cite{Schlein2014}[Proposition~2.2].
For the readers convenience we include  its short proof below.
\begin{lemma}\label{appendix inequality for lambda}
For all $v > 0$ and $u \le 2+v$
\begin{equation}\label{eq: abs value lambda}
|\Lambda_n| \le C \min\left\{\frac{|T_n|}{|b(z)|}, \sqrt{|T_n|}\right\}.
\end{equation}
Moreover, for all $v>0$ and $u \in \R$
\begin{equation}\label{eq: abs imag lambda}
|\imag\Lambda_n| \le C \min\left\{\frac{|T_n|}{|b(z)|}, \sqrt{|T_n|}\right\}
\end{equation}
and
\begin{equation}\label{eq: min of abs values lambda}
\min(|\Lambda_n|, |b_n(z)|) \le C\sqrt{|T_n|}.
\end{equation}
\end{lemma}
\begin{proof}
We start the proof from the following identity, see representation~\eqref{eq: R_jj representation},
$$
m_n(z)= -\frac{1}{z + m_n(z)} + \frac{T_n}{z + m_n(z)}.
$$
Solving this quadratic equation we get that
$$
m_n(z) = -\frac{z}{2} + \sqrt{\frac{z^2}{4} - 1 + T_n}.
$$
The Stieltjes transform of the semicircle law may be written explicitly
$$
s(z) = -\frac{z}{2} + \sqrt{\frac{z^2}{4} - 1}.
$$
From the last two equation we conclude that
$$
\Lambda_n = \sqrt{\frac{z^2}{4} - 1 + T_n} - \sqrt{\frac{z^2}{4} - 1}.
$$
In what follows we will use the the following additional notation
$$
a: = \frac{z^2}{4} - 1.
$$
It is easy to see that $a = (s(z) + \frac{z}{2})^2 = b^2(z)/4$.

We start from the proof of~\eqref{eq: min of abs values lambda} since it is trivial.
If $|b_n(z)| \le C \sqrt{|T_n|}$ there is nothing to prove. In the opposite case we get
$$
|\Lambda_n| \le \frac{|T_n|}{|b_n(z)|} \le C \sqrt{|T_n|}.
$$

Now we establish inequality~\eqref{eq: abs imag lambda}. First we show that $|\imag \Lambda_n| \le C  \sqrt{|T_n|}$. Let us consider several cases:\\

\noindent
I) $|a| \le 2 |T_n|$. In this situation
\begin{equation}\label{eq: appendix inequality for difference of roots 0}
|\imag \sqrt{a+T_n} - \imag \sqrt{a}| \le  \sqrt{|a+T_n|}+ \sqrt{|a|} \le (\sqrt 3 + \sqrt 2) \sqrt{|T_n|}.
\end{equation}
\\
\noindent
II) $|a| > 2|T_n|$. We split this case into several sub cases\\
\\
\noindent
II) 1. $\re a < 0$. Since we always take the branch with the positive imaginary part we may write
\begin{equation}\label{eq: appendix inequality for sum of roots 0}
\imag \sqrt{a+T_n} + \imag \sqrt{a} \geq \imag \sqrt{a} \geq \frac{\sqrt 2}{2} |a|^{1/2} \geq \sqrt{|T_n|}.
\end{equation}
The last inequality implies that
\begin{equation}\label{eq: appendix inequality for difference of roots}
|\imag \sqrt{a+T_n} - \imag \sqrt{a}| \le \frac{|T_n|}{|\sqrt{a+T_n} + \sqrt{a}| } \le \sqrt{|T_n|}.
\end{equation}

\noindent II) 2. $\re a > 0$ and $\re(a+T_n) < 0$ . We have
$$
\imag \sqrt{a+T_n} \geq \frac{\sqrt 2}{2} \sqrt{|a+T_n|} \geq \frac{\sqrt 2}{2} (|a| - |T_n|)^{1/2} \geq \frac{\sqrt 2}{2} \sqrt{|T_n|}
$$
and similarly to~\eqref{eq: appendix inequality for difference of roots} we obtain
$$
|\imag \sqrt{a+T_n} - \imag \sqrt{a}| \le \sqrt 2 \sqrt{|T_n|}.
$$
\\
\noindent II) 3. $\re a > 0$ and $\re(a+T_n) >  0$ . We again consider two cases, but both are similar.\\

\noindent II) 3. 1. $\imag a \imag (a+T_n) > 0$. In this situation
\begin{equation}\label{eq: appendix inequality for sum of roots 1}
|\sqrt{a+T_n} + \sqrt{a}| \geq \sqrt{|a|} \geq \sqrt 2 \sqrt{|T_n|}.
\end{equation}
\\
\noindent II) 3. 2. $\imag a \imag (a+T_n) < 0$. Since $\imag \sqrt{a+T_n} = \imag \sqrt{\overline{a+T_n}}$ we have
$$
\imag a \imag (\overline{a+T_n}) < 0
$$
and
\begin{equation}\label{eq: appendix inequality for sum of roots 2}
|\sqrt{\overline{a+T_n}} + \sqrt{a}| \geq \sqrt{|a|} \geq \sqrt 2 \sqrt{|T_n|}.
\end{equation}
Similarly to~\eqref{eq: appendix inequality for difference of roots} we may conclude from~\eqref{eq: appendix inequality for sum of roots 1} and~\eqref{eq: appendix inequality for sum of roots 2} that
$$
|\imag \sqrt{a+T_n} - \imag \sqrt{a}| \le \sqrt 2 \sqrt{|T_n|}.
$$

To finish the proof of~\eqref{eq: abs imag lambda} we need to show that
$$
|\imag \Lambda_n| \le   C \frac{|T_n|}{\sqrt{|a|}}.
$$
The proof follows by similar arguments as in the proof of $|\imag \Lambda_n| \le   C |T_n|^\frac12$. We consider several cases:\\

\noindent I) $\re a < 0$. In this situation one need to repeat the inequalities~\eqref{eq: appendix inequality for sum of roots 0} and ~\eqref{eq: appendix inequality for difference of roots}. We get
$$
|\imag \sqrt{a+T_n} - \imag \sqrt{a}| \le \frac{|T_n|}{\imag \sqrt a} \le \sqrt 2 \frac{|T_n|}{\sqrt{|a|}}.
$$
\\
\noindent II) $\re a > 0$ and $|a| <  2|T_n|$. Then
$$
|\imag \sqrt{a+T_n} - \imag \sqrt{a}| \le C \sqrt{|T_n|} = C\frac{|T_n|}{\sqrt{|T_n|}} \le C' \frac{|T_n|}{\sqrt{|a|}}.
$$
\\
\noindent III) $\re a > 0$ and $|a| >  2|T_n|$. We consider two sub cases\\

\noindent III) 1. $\re(a+T_n) < 0$. Then
$$
\imag \sqrt{a+T_n} \geq \frac{\sqrt 2}{2} |a+T_n|^{1/2} \geq \frac{|a|^{1/2}}{2}
$$
and it follows that
$$
|\imag \sqrt{a+T_n} - \imag \sqrt{a}| \le  C \frac{|T_n|}{\sqrt{|a|}}.
$$
\\
\noindent III) 2. $\re(a+T_n) >  0$. Then without loss of generality we may assume that
$$
\imag a \imag (a+T_n) > 0.
$$
Then
$$
|\sqrt{a+T_n} + \sqrt{a}| \geq \sqrt{|a|}
$$
and we  conclude
$$
|\imag \sqrt{a+T_n} - \imag \sqrt{a}| \le  C \frac{|T_n|}{\sqrt{|a|}}.
$$

It remains to prove~\eqref{eq: abs value lambda} . Let us first suppose that $|\Lambda_n| \le c\sqrt{|T_n|}$. Then
\begin{align*}
\Lambda_n &= \frac{T_n}{z + m_n(z) + s(z)} = \frac{T_n}{z+2s(z)} + \frac{T_n \Lambda_n}{(z+2s(z))(z+m_n(z)+s(z))}\\
&=\frac{T_n}{z+2s(z)} + \frac{\Lambda_n^2}{z+2s(z)}
\end{align*}
and we immediately get that
$$
|\Lambda_n(z)| \le C \frac{|T_n|}{\sqrt{|z^2 - 4|}}.
$$
Finally it remains to prove the assumption $|\Lambda_n| \le c\sqrt{|T_n|}$. There is nothing to prove if $|a| \le 2 |T_n|$, one need to apply the same inequalities as in~\eqref{eq: appendix inequality for difference of roots 0} and get
$$
|\sqrt{a+T_n} - \sqrt{a}| \le \sqrt{|a+T_n|}+ \sqrt{|a|} \le (\sqrt 3+ \sqrt 2) |T_n|^\frac12.
$$
If $|a| \geq 2 |T_n|$ we apply the fact that for $|u| \le 2 + v$ there exists the constant $c > 0$ such that $|\imag a| \geq c \re a$. Applying this fact we get
$$
|\sqrt{a+T_n} - \sqrt{a}| \le \frac{|T_n|}{|\sqrt{a+T_n} - \sqrt{a}|} \le \frac{|T_n|}{\imag \sqrt{a}} \le c' \frac{|T_n|}{|a|^\frac12} \le c'' |T_n|^\frac12.
$$

\end{proof}

Recall that $\varphi(z) = \overline z |z|^{p-1}$. In the following lemma we  estimate the difference between  $\varphi(T_n)$ and $\varphi(\widetilde T_n^{(j)})$.
\begin{lemma}\label{appendix Taylor formula}
	For $p \geq 2$ and arbitrary $j \in \T$ we have
	$$
	|\varphi(T_n)-\varphi(\widetilde T_n^{(j)})| \le p \E_{\tau}|T_n - \widetilde T_n^{(j)}||\widetilde T_n^{(j)} + \tau (T_n -\widetilde T_n^{(j)})|^{p-2},
	$$
	where $\E_{\tau}$ is a mathematical expectation with respect to the uniformly distributed on $[0, 1]$ random variable $\tau$.
\end{lemma}
\begin{proof}
	The proof follows from the Newton-Leibniz formula applied to
	$$
	\hat \varphi(x) = \varphi( \widetilde T_n^{(j)} + x (T_n -\widetilde T_n^{(j)})), \quad x \in [0,1],
	$$
	and
	$$
	|\hat \varphi'(x)| \le p | \widetilde T_n^{(j)} + x (T_n -\widetilde T_n^{(j)}) |^{p-2}.
	$$
\end{proof}

The proofs of the following two lemmas are rather straightforward, but will be used many times in the proof of Theorem~\ref{th: general bound}.

\begin{lemma}\label{appendix l: inequality for x power p}
	Let us assume that for all $p > q \geq 1$ and $a,b > 0$ the following inequality holds
	\begin{equation}\label{appendix eq: inequality for x power p}
	x^p \le a + b x^{q}.
	\end{equation}
	Then
	$$
	x^p \le 2^{\frac{p}{p-q}} (a + b^{\frac{p}{p-q}}).
	$$
\end{lemma}
\begin{proof}
	The proof is easy. We may assume that $x > a^{\frac{1}{p}}$ since in the opposite case the inequality is trivial. Dividing  both parts of~\eqref{appendix eq: inequality for x power p} by
	$x^q$ we obtain
	$$
	x^{p-q} \le a^{\frac{p-q}{p}} + b.
	$$
	Finally we get
	$$
	x^p \le 2^{\frac{p}{p-q}} (a + b^{\frac{p}{p-q}}).
	$$
\end{proof}

\begin{lemma}\label{appendix eq: inequality for x power p 2}
	Let $0 < q_1 \le q_2 \le ... \le q_k < p$ and $c_j, j = 0, ... , k$ be positive numbers such that
	$$
	x^p \le c_0 + c_1 x^{q_1} + c_2 x^{q_2} + ... + c_k x^{q_k}.
	$$
	Then
	$$
	x^p \le \beta \left[ c_0 + c_1^{\frac{p}{p-q_1}} + c_2^{\frac{p}{p-q_2}} + ... + c_k^{\frac{p}{p-q_k}} \right],
	$$
	where
	$$
	\beta: = \prod_{\nu=1}^{k} 2^{\frac{p}{p-q_\nu}} \le 2^{\frac{kp}{p-q_k}}.
	$$
\end{lemma}
\begin{proof}
	Let $a_1: = c_0 + c_2 x^{q_2} + ... + c_{k} x^{q_{k}}$ and $b_1 := c_1$. We may apply Lemma~\ref{appendix l: inequality for x power p} and get
	$$
	x^p \le 2^{\frac{p}{p-q_1}}(a_1 + b_1^{\frac{p}{p-q_1}}).
	$$
	Repeating this step  $k-1$ times we obtain
	$$
	x^p \le \beta \left[ c_0 + c_1^{\frac{p}{p-q_1}} + c_2^{\frac{p}{p-q_2}} + ... + c_k^{\frac{p}{p-q_k}} \right],
	$$
	where $\beta$ is defined above.
\end{proof}

\section{Auxiliary lemmas II}
In this section we collect all inequalities for the resolvent of the matrix $\W$.
\begin{lemma}\label{appendix lemma resolvent relations on different v}
For any $z = u + i v \in \C^{+}$ we have for any $s \geq 1$
\begin{equation}\label{appendix eq resolvent relations on different v 0}
|\RR_{jj}^{(\J)}(u + i v/s)| \le s |\RR_{jj}^{(\J)}(u + i v)|
\end{equation}
and
\begin{equation}\label{appendix eq resolvent relations on different v 1}
\frac{1}{|u+iv/s_0 + m_n^{(\J)}(u+iv/s_0)|} \le \frac{s_0}{|u+iv + m_n^{(\J)}(u+iv)|}.
\end{equation}
\end{lemma}
\begin{proof}
The proof of~\eqref{appendix eq resolvent relations on different v 0} is given in~\cite{Schlein2014}, but for the readers convenience we will present it here.
To simplify all formulas we shall omit the index $\J$ from the notation of $\RR_{jj}$. Since
\begin{equation}\label{appendix eq log R_jj}
\left| \frac{d}{d v} \log \RR_{jj}(v) \right | \le \frac{1}{|\RR_{jj}(v)|} \left | \frac{d}{dv} \RR_{jj}(v) \right |.
\end{equation}
Furthermore,
$$
\frac{d}{dv} \RR_{jj}(v) = [\RR^2]_{jj}(v)
$$
and
$$
|[\RR^2]_{jj}(v)| \le v^{-1} \imag \RR_{jj}.
$$
Applying this inequality to~\eqref{appendix eq log R_jj} we get
$$
\left| \frac{d}{d v} \log \RR_{jj}(v) \right | \le \frac{1}{v}.
$$
This inequality implies that
$$
|\log \RR_{jj}(v) - \log \RR_{jj}(v/s)| \le \left| \int_{v/s}^v \frac{d}{d v} \log \RR_{jj}(\eta)\, d\eta \right | \le \log s.
$$
Since the last inequality holds for the real parts of the logarithm as well, we may conclude that
$$
|\RR_{jj}(u + i v/s)| \le s |\RR_{jj}(u + i v)|.
$$
The proof of~\eqref{appendix eq resolvent relations on different v 1} is similar and we omit it.
\end{proof}

\begin{lemma}\label{appendix lemma imag resolvent relations on different v}
Let $g(v): = g(u+iv)$ be the Stieltjes transform of some distribution function $G(x)$. Then for any $s \geq 1$
\begin{align}
\imag g(v/s) \le s \imag g(v) \, \text{ and } \, \imag g(v) \le s \imag g(v/s).
\end{align}
\end{lemma}
\begin{proof}
Recall that
$$
\imag g(v) = \int_{-\infty}^\infty \frac{v}{(x-u)^2 + v^2}\, dG(x).
$$
Hence,
$$
\left|\frac{d \imag g(v)}{d v}\right| = \int_{-\infty}^\infty \frac{|(x-u)^2 - v^2|}{((x-u)^2 + v^2)^2}\, dG(x) \le \frac{1}{v} \imag g(v).
$$
We may conclude that
$$
\left |\frac{d \imag g(v)}{d v}\right| \le \frac{1}{v}.
$$
We may repeat now the second part of the previous lemma and get the desired bounds.
\end{proof}

\begin{lemma}\label{appendix lemma inequality v le  imag s RR jj}
For any $z = u + i v \in \C^{+}$ there exists a constant $c=c(z) > 0$ such that
\begin{equation}\label{appendix inequality v le  imag s}
v \le c \imag s(z)
\end{equation}
and
\begin{equation}\label{appendix inequality v le  imag RR jj}
v \le \frac{\imag \RR_{jj}}{|\RR_{jj}|^2}.
\end{equation}
\end{lemma}
\begin{proof}
It is easy to see that
$$
|s(z)|^2 \le \int_{-2}^2 \frac{g_{sc}(\lambda) }{|\lambda - z|^2}\, d\lambda  = \frac{1}{v} \imag s(z).
$$
Since $|s(z)|^2 \geq c^{-1}$ for some $c = c(z)$  the inequality~\eqref{appendix inequality v le  imag s} follows. In order to prove~\eqref{appendix inequality v le  imag RR jj} one should repeat the calculations above using  the following spectral representation of $\RR_{jj}$
$$
\RR_{jj} = \int_{-\infty}^{\infty} \frac{1}{z - \lambda} \, d F_{nj}(\lambda), \quad F_{nj}(\lambda) : = \sum_{k =1}^n |u_{jk}|^2 \one[\lambda_j \le \lambda].
$$
\end{proof}
We finish this section with two lemmas. These lemmas are proved in~\cite{GotTikh2015}[Lemma~7.10] and~\cite{GotzeTikh2014rateofconv}[Lemma~7.6] , but for the readers convenience we include them  here.
\begin{lemma}\label{appendix lemma resolvent inequalities 1}
For any $z = u + i v \in \C^{+}$ we have
\begin{equation}\label{appendix lemma resolvent inequality 1}
\frac{1}{n} \sum_{l,k \in \T_{\J}} |\RR_{kl}^{(\J)}|^2 \le \frac{1}{v} \imag m_n^{(\J)}(z).
\end{equation}
For any $l \in \T_{\J}$
\begin{equation}\label{appendix lemma resolvent inequality 2}
\sum_{k \in \T_{\J}} |\RR_{kl}^{(\J)}|^2 \le \frac{1}{v} \imag \RR_{ll}^{(\J)}.
\end{equation}
\end{lemma}
\begin{proof}
We denote by $u_k^{(\J)}: = (u_{lk}^{(\J)})_{l \in \T_{\J}}$ the eigenvector of $\W^{(\J)}$ corresponding to the eigenvalue $\lambda_k^{(\J)}$. It follows from the eigenvector decomposition that
\begin{equation}\label{appendix lemma eigenvector decomposition}
\RR_{kl}^{(\J)} = \sum_{s \in \T_{\J}} \frac{1}{\lambda_s^{(\J)} - z} u_{ks}^{(\J)} u_{ls}^{(\J)}.
\end{equation}
Since $\U: = [u_{lk}^{(\J)}]_{l,k \in \T_{\J}}$ is a unitary matrix we get
$$
\frac{1}{n} \sum_{l,k \in \T_{\J}} |\RR_{kl}^{(\J)}|^2 \le \frac{1}{n} \sum_{s \in \T_{\J}} \frac{1}{|\lambda_s^{(\J)} - z|^2} \le \frac{1}{v} \imag m_n^{(\J)}(z).
$$
To prove~\eqref{appendix lemma resolvent inequality 2} we may conclude from~\eqref{appendix lemma eigenvector decomposition}
that
$$
\sum_{k \in \T_{\J}} |\RR_{kl}^{(\J)}|^2 \le \sum_{s \in \T_{\J}} \frac{|u_{ls}^{(\J)}|^2}{|\lambda_s^{(\J)} - z|^2}  =
\frac{1}{v} \imag \left(\sum_{s \in \T_{\J}} \frac{|u_{ls}^{(\J)}|^2}{\lambda_s^{(\J)} - z} \right) = \frac{1}{v}\imag \RR_{ll}^{(\J)}.
$$
\end{proof}

\begin{lemma}\label{appendix lemma resolvent square inequalities}
For any $z = u + i v \in \C^{+}$ we have
\begin{align}
\label{appendix lemma resolvent square inequality 1}
&\frac{1}{n} \big |\Tr (\RR^{(\J)})^2 \big | \le \frac{1}{v} \imag m_n^{(\J)}(z),\\
\label{eq: appendix lemma resolvent square inequality 2}
&\frac{1}{n}\sum_{k, l \in \T_{\J}} |[(\RR^{(\J)})^2]_{lk}|^2 \le \frac{1}{v^3} \imag m_n^{(\J)}(z), \\
\label{eq: appendix lemma resolvent square inequality 3}
&\frac{1}{n}\sum_{k \in \T_{\J}} |[(\RR^{(\J)})^2]_{kk}|^2 \le \frac{1}{v^3} \imag m_n^{(\J)}(z),\\
\label{eq: appendix lemma resolvent square inequality 4}
&\frac{1}{n}\sum_{k \in \T_{\J}} |[(\RR^{(\J)})^2]_{kk}|^p \le \frac{1}{nv^p} \sum_{ k \in \T_{\J}} \imag^p \RR_{kk}^{(\J)}(z) \text{ for any } p \geq 1.
\end{align}
For any $l \in \T_{\J}$
\begin{align}
\label{eq: appendix lemma resolvent square inequality 5}
\sum_{k \in \T_{\J}} |[(\RR^{(\J)})^2]_{lk}|^2 \le \frac{1}{v^3}\imag \RR_{ll}^{(\J)}.
\end{align}
\end{lemma}
\begin{proof}
The proof of~\eqref{appendix lemma resolvent square inequality 1} follows from
$$
\frac{1}{n} \big |\Tr (\RR^{(\J)})^2 \big | \le \frac{1}{n} \sum_{j \in \T_{\J}} \frac{1}{|\lambda_j^{(\J)} - z|^2} \le \frac{1}{v} \imag m_n^{(\J)}(z).
$$
We denote by $u_k^{(\J)}: = (u_{lk}^{(\J)})_{l \in \T_{\J}}$ the eigenvector of $\W^{(\J)}$ corresponding to the eigenvalue $\lambda_k^{(\J)}$. It follows from the eigenvector decomposition that
\begin{equation}\label{appendix lemma eigenvector decomposition 2}
[(\RR^{(\J)})^2]_{kl} = \sum_{s \in \T_{\J}} \frac{1}{(\lambda_s^{(\J)} - z)^2} u_{ks}^{(\J)} u_{ls}^{(\J)}.
\end{equation}
Since $\U: = [u_{lk}^{(\J)}]_{l,k \in \T_{\J}}$ is a unitary matrix we get
$$
\frac{1}{n} \sum_{l,k \in \T_{\J}} |[(\RR^{(\J)})^2]_{lk}|^2 \le \frac{1}{n} \sum_{s \in \T_{\J}} \frac{1}{|\lambda_s^{(\J)} - z|^4} \le
\frac{1}{v^3} \imag \left(\frac{1}{n} \sum_{s \in \T_{\J}} \frac{1}{\lambda_s^{(\J)} - z} \right) = \frac{1}{v^3} \imag m_n^{(\J)}(z).
$$
This proves~\eqref{eq: appendix lemma resolvent square inequality 2}.
Inequality~\eqref{eq: appendix lemma resolvent square inequality 3} follows from~\eqref{appendix lemma resolvent inequality 2} and observation that
$$
[(\RR^{(\J)})^2]_{kk} = \sum_{l \in \T_{\J}} (\RR_{lk}^{(\J)})^2.
$$
The proof of~\eqref{eq: appendix lemma resolvent square inequality 4} is similar.
To prove~\eqref{eq: appendix lemma resolvent square inequality 5}
we may conclude from~\eqref{appendix lemma eigenvector decomposition 2}
that
$$
\sum_{k \in \T_{\J}} |[(\RR^{(\J)})^2]_{lk}|^2 \le \sum_{s \in \T_{\J}} \frac{|u_{ls}^{(\J)}|^2}{|\lambda_s^{(\J)} - z|^4} \le \frac{1}{v^3}\imag \RR_{ll}^{(\J)}.
$$
\end{proof}

\section{Truncation of matrix entries}

In this section we will show that the  conditions $\Cond$ allows us to assume that  for all $1 \le j,k \le n$ we have $|X_{jk}| \le D n^{\alpha}$, where $D$ is some positive constant and
$$
\alpha = \frac{2}{4+\delta}.
$$
Let $\hat X_{jk}: = X_{jk} \one[|X_{jk}| \leq D n^\alpha]$, $\tilde X_{jk}: = X_{jk} \one[|X_{jk}| \geq D n^\alpha] - \E X_{jk} \one[|X_{jk}| \geq D n^\alpha]$ and finally
$\breve X_{jk}: = \tilde X_{jk} \sigma^{-1}$, where $\sigma^2: = \E |\tilde X_{11}|^2$. We denote symmetric random matrices by $\hat \X, \tilde \X$ and $\breve \X$ formed from $\hat X_{jk}, \tilde X_{jk}$ and $\breve X_{jk}$ respectively. Similar notations are used for the resolvent matrices and corresponding Stieltjes transforms.
\begin{lemma}\label{appendix: lemma trunc 1}
Assuming the conditions of Theorem~\ref{th:main} we have
$$
\E|m_n(z) - \hat m_n(z)|^p \le \left(\frac{Cp}{nv}\right)^p.
$$
\end{lemma}
\begin{proof}
From Bai's rank inequality (see~\cite{BaiSilv2010}[Theorem~A.43])  we conclude that
$$
\sup |\mathcal F_n(x) - \hat {\mathcal F}_n(x)| \le \frac{1}{n} \Rank(\X - \hat \X) \le \frac{1}{n} \sum_{j,k=1}^n \one[|X_{jk}| \geq D n^\alpha].
$$
Integrating by parts we get
$$
\E|m_n(z) - \hat m_n(z)|^p \le \frac{1}{(nv)^p} \E \left(\sum_{j,k=1}^n \one[|X_{jk}| \geq D n^\alpha] \right)^p.
$$
It is easy to see that
$$
\left(\sum_{j,k=1}^n \E\one[|X_{jk}| \geq D n^\alpha] \right)^p \le C^p.
$$
Applying Rosenthal's inequality (Theorem~\ref{th: Rosenthal}) we get that
\begin{align*}
&\E \left(\sum_{j,k=1}^n [\one[|X_{jk}| \geq D n^\alpha] - \E\one[|X_{jk}| \geq D n^\alpha]] \right)^p \\
&\qquad\qquad\qquad\le C^p p^p \left( \left(\frac{1}{n^2} \sum_{j,k=1}^n E |X_{jk}|^{4 + \delta} \right)^\frac{p}{2}  + \frac{1}{n^2} \sum_{j,k=1}^n E |X_{jk}|^{4 + \delta}  \right ) \le C^p p^p.
\end{align*}
From these inequalities we may conclude the statement of Lemma.
\end{proof}

\begin{lemma}\label{appendix: lemma trunc 0}
Assuming the conditions of Theorem~\ref{th:main} we have
$$
\E|\tilde m_n(z) - \breve m_n(z)|^p \le \frac{C^p p^p \mathcal A^p(2p)}{(nv)^p}.
$$
\end{lemma}
\begin{proof}
It is easy to see that
\begin{equation}\label{eq: tilde R representation}
\tilde \RR(z) = (\tilde \W - z\I)^{-1} = \sigma^{-1} (\breve \W - z \sigma^{-1}\I)^{-1}  = \sigma^{-1}\breve \RR(\sigma^{-1}z).
\end{equation}
Applying the resolvent equality we get
\begin{equation}\label{eq: overline R representation}
\breve \RR(z) - \breve \RR(\sigma^{-1}z) = (z - \sigma^{-1}z) \breve \RR(z) \breve \RR(\sigma^{-1}z).
\end{equation}
From~\eqref{eq: tilde R representation} and~\eqref{eq: overline R representation} we may conclude
\begin{align*}
|\tilde m_n(z) - \breve m_n(z)| &= \frac{1}{n} | \Tr \tilde \RR(z) - \Tr \breve \RR(z)| = \frac{1}{n} | \sigma^{-1}\Tr \breve \RR(\sigma^{-1}z) - \Tr \breve \RR(z)| \\
&= \frac{1}{n} | \sigma^{-1}\Tr \breve \RR(z) - \Tr \breve \RR(z)- (z - \sigma^{-1}z) \Tr \breve \RR(z) \breve \RR(\sigma^{-1}z)|\\
&\le \frac{1}{n}(\sigma^{-1} - 1) |\Tr \breve \RR(z)| + (\sigma^{-1} - 1)  \frac{|z|}{n}|\Tr \breve \RR(z) \breve \RR(\sigma^{-1}z)|.
\end{align*}
Taking the $p$-th power and mathematical expectation we get
\begin{align*}
\E|\tilde m_n(z) - \breve m_n(z)|^p &\le  \frac{1}{n^p}(\sigma^{-1} - 1)^p \E|\Tr \breve \RR(z)|^p + (\sigma^{-1} - 1)^p  \frac{C^p}{n^p}\E|\Tr \breve \RR(z) \breve \RR(\sigma^{-1}z)|^p.
\end{align*}
Since $\breve \X$ satisfies conditions $\CondTwo$ we may apply  Lemma~\ref{main lemma} and conclude
$$
\frac{1}{n^p} \E|\Tr \breve \RR(z)|^p \le C_0^p.
$$
We also have
\begin{equation}\label{appendix variance truncation}
\sigma^{-1} - 1 \le \sigma^{-1} (1  - \sigma) \le \sigma^{-1} (1  - \sigma^2) \le \sigma^{-1} \E |X_{jk}|^2 \one[|X_{jk}| \geq D n^\alpha]   \le \frac{C}{n}.
\end{equation}
To finish the proof it remains to estimate the term
$$
\frac{1}{n^p}\E|\Tr \breve \RR(z) \breve \RR(\sigma^{-1}z)|^p.
$$
Applying the obvious inequality $|\Tr \A \B| \le \|\A\|_2 \| \B\|_2$ we get
\begin{align*}
\frac{1}{n^p}\E|\Tr \breve \RR(z) \breve \RR(\sigma^{-1}z)|^p &\le \frac{1}{n^p} \E^\frac12\|\breve \RR(z)\|_2^{2p} \E^\frac{1}{2}\| \breve \RR(\sigma^{-1}z) \|_2^{2p}\\
&\le \frac{\E^\frac12 \imag^p \breve m_n(z) \E^\frac12 \imag^p \breve m_n(\sigma^{-1} z)}{v^p}.
\end{align*}
From this inequality and~\eqref{appendix variance truncation} we conclude the statement of the lemma.
\end{proof}

\begin{lemma}\label{appendix: lemma trunc 2}
Assuming the conditions of Theorem~\ref{th:main} we have
$$
\E|\tilde m_n(z) - \hat m_n(z)|^p \le \left(\frac{C}{nv}\right)^\frac{3p}{2}.
$$
\end{lemma}
\begin{proof}
It is easy to see that
$$
\tilde m_n(z) - \hat m_n(z)  = \frac{1}{n} \Tr (\tilde \W - \hat \W) \hat \RR \tilde \RR.
$$
Applying the obvious inequalities $|\Tr \A \B| \le \|\A\|_2 \| \B\|_2$ and $\|\A \B\|_2 \le \|\A\|\|\B\|_2$ we get
$$
|\tilde m_n(z) - \hat m_n(z)|  \le \|\tilde \W - \hat \W\|_2 \|\hat \RR\|_2 \|\tilde \RR\| = \|\E \hat \W\|_2 \|\tilde \RR\|_2 \|\hat \RR\|.
$$
From
$$
|\E \hat X_{jk}| = |\E X_{jk} \one[|X_{jk}| \geq D n^\alpha]| \le \frac{C}{n^\frac{2(3+\delta)}{4+\delta}}
$$
we obtain
$$
\|\E \hat \W\|_2 \le \frac{C}{n^\frac{8+3 \delta}{2(4+\delta)}}.
$$
By Lemma~\ref{appendix: lemma trunc 0} we know $\E |\tilde m_n(z)|^p \le C^p$. This implies that
$$
\frac{1}{n^\frac{p}{2}} \E\|\tilde \RR\|_2^p \le \frac{C^p}{v^\frac{p}{2}}.
$$
Finally
$$
\E|\tilde m_n(z) - \hat m_n(z)|^p \le\frac{C^p }{v^\frac{3p}{2} n^\frac{p}{2} n^\frac{(8+3 \delta)p}{2(4+\delta)}} \le  \left(\frac{C}{nv}\right)^\frac{3p}{2}.
$$
\end{proof}

\bibliographystyle{plain}
\bibliography{literatur}

\end{document}